\documentclass[3p,times]{article}
\usepackage{bbm,amsthm}

\newcommand{\Model}{\mathcal{M}} 

\newcommand{\cbf}{\mathbf{c}}

\newcommand{\R}{\mathbf{R}}
\newcommand{\B}{\mathbf{B}}

\newcommand{\x}{\mathbf{x}}
\newcommand{\y}{\mathbf{y}}

\newcommand{\M}{\mathbf{M}}

\newcommand{\J}{\mathcal{J}}

\newcommand{\qoi}{{\sc q}o{\sc i} }

\usepackage{subfigure}
\usepackage{amsmath}
\usepackage{amssymb}
\usepackage{graphicx}
\usepackage{capt-of}
\usepackage{color}
\usepackage{url}
\usepackage{algorithm}
\usepackage{algpseudocode,caption,ulem}
\usepackage[toc,page]{appendix}

\newtheorem{comment}{Comment}
\newtheorem{theorem}{Theorem}
\algblock{ParFor}{EndParFor}
\algnewcommand\algorithmicparfor{\textbf{For all}}
\algnewcommand\algorithmicpardo{\textbf{do in parallel}}
\algnewcommand\algorithmicendparfor{\textbf{end\ For all}}
\algrenewtext{ParFor}[1]{\algorithmicparfor\ #1\ \algorithmicpardo}
\algrenewtext{EndParFor}{\algorithmicendparfor}

\title{A-posteriori error estimates for inverse problems}
\author{
	Vishwas Rao and Adrian Sandu
}
\date{\today}

\begin{document}
\thispagestyle{empty}
\setcounter{page}{0}

\begin{Huge}
\begin{center}
%Computer Science Technical Report CSTR-{\tt insert number here} \\
Computational Science Laboratory Technical Report CSL-TR-16-2014\\
\today
\end{center}
\end{Huge}
\vfil
\begin{huge}
\begin{center}
Vishwas Rao and Adrian Sandu
\end{center}
\end{huge}

\vfil
\begin{huge}
\begin{it}
\begin{center}
``A-posteriori error estimates for inverse problems''
\end{center}
\end{it}
\end{huge}
\vfil
\textbf{Cite as:} Vishwas Rao and Adrian Sandu.  A posteriori error estimates for DDDAS inference problems. Procedia Computer Science Journal. Volume 29, Pages 1256 -- 1265, ``2014 International Conference on Computational Science.''
\vfil

\begin{large}
\begin{center}
Computational Science Laboratory \\
Computer Science Department \\
Virginia Polytechnic Institute and State University \\
Blacksburg, VA 24060 \\
Phone: (540)-231-2193 \\
Fax: (540)-231-6075 \\ 
Email: \url{sandu@cs.vt.edu} \\
Web: \url{http://csl.cs.vt.edu}
\end{center}
\end{large}

\vspace*{1cm}

\begin{tabular}{ccc}
\includegraphics[width=2.5in]{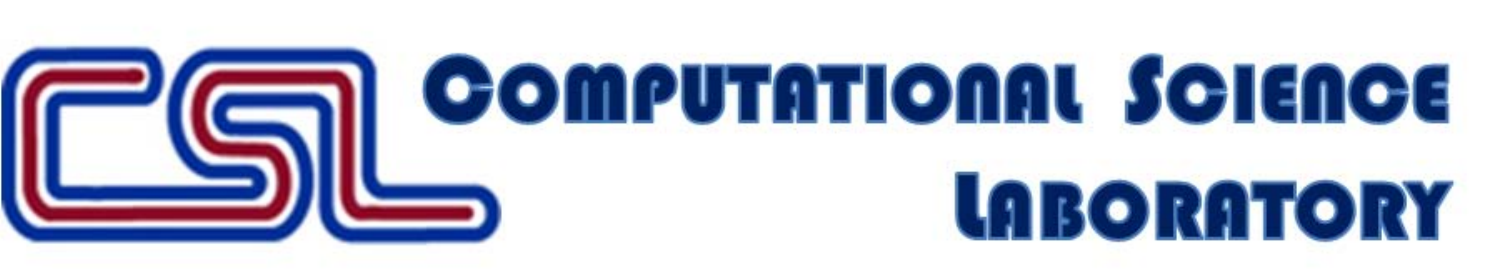}
&\hspace{2.5in}&
\includegraphics[width=2.5in]{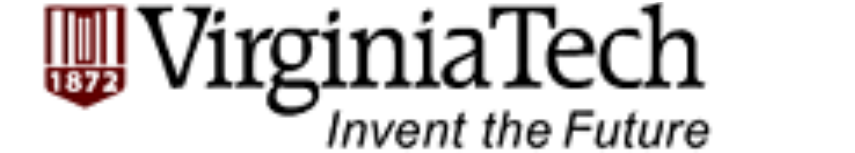} \\
{\bf\em Innovative Computational Solutions} &&\\
\end{tabular}

\newpage

 \maketitle
\begin{abstract}
Inverse problems use physical measurements along with a computational model to estimate the parameters or state of a system of interest.  
Errors in measurements and uncertainties in the computational model lead to inaccurate estimates.  This work develops a methodology to estimate the impact of different errors on the variational solutions of inverse problems. The focus is on time evolving systems described by ordinary differential equations, and on a particular class of inverse problems, namely, data assimilation. The computational algorithm uses first-order and second-order adjoint models. 
In a deterministic setting the methodology provides a posteriori error estimates for the inverse solution. In a probabilistic setting it provides an a posteriori quantification of uncertainty in the inverse solution, given the uncertainties in the model and data. Numerical experiments with the shallow water equations in spherical coordinates illustrate the use of the proposed error estimation machinery in both deterministic and probabilistic settings. 
%The mean and variance of the error estimates are cross validated using statistics from the ensemble.
\end{abstract}

\newpage
\tableofcontents

\newpage\setcounter{page}{1}
%%%%%%%%%%%%%%%%%%%%%%%
\section{Introduction}
%%%%%%%%%%%%%%%%%%%%%%%
Inverse problems use information from different sources in order to infer the state or parameters of a system of interest.
Data assimilation is a class of inverse problems that combines information from an imperfect computational model 
(which encapsulates our knowledge of the physical laws that govern the evolution of the real system), from noisy observations (sparse snapshots of 
reality), and from an uncertain prior (which encapsulates our current knowledge of reality). Data assimilation combines these three sources of information and the associated uncertainties in a Bayesian framework to provide the posterior, i.e., the best description of reality when considering the new information from the data.
In a variational approach data assimilation is formulated as an optimization problem whose solution represents a maximum likelihood estimate of the state or parameters. The errors in the underlying computational observation as well as the errors in the observations lead to error in the optimal solution. Our goal is to quantitatively estimate the impact of various errors on the accuracy of the optimal solution. 

A posteriori error estimation is concerned with quantifying the error associated with a particular -- and already computed -- solution of the problem of interest \cite{Estep:1995,Estep:2002}.  A posteriori error estimation is a well-established methodology in the context of numerical approximations of partial differential equations \cite{ainsworth2011posteriori}.  The approach has been extended to the solution of inverse problems \cite{becker2001optimal} and has been applied to guide mesh refinement  \cite{becker2005mesh}.  The Ph.D. dissertation of M. Alexe \cite{Alexe:2011} develops systematic methodologies for quantifying the impact of various errors on the optimal solution in variational inverse problems. Recent related work in the context of variational data assimilation has developed tools to quantify the impact of errors in the background, observations, and the associated error covariance matrices on the accuracy of resulting analyses \cite{Gejadze:2008,Gejadze:2013,Gejadze:2012}. The choice of 
optimal error covariances for estimating parameters such as distributed coefficients and boundary conditions for a convection-diffusion model has been discussed in \cite{Gejadze:2010}. 

While previous work has considered the impact of data errors, no method is available to date to estimate the impact of model errors on the optimal solution of a variational inverse problem.This paper develops a coherent framework to estimate the impact of both model and data errors on the optimal solution. The computational procedure makes use of first order and the second order adjoint information and builds upon our previous work \cite{Alexe:2011,Alexe:2013,Cioaca:2012}.

The remainder of the paper is organized as follows. In Section \ref{sec:probdef} we define the problem and derive the optimality conditions for the problem in \ref{sec:kkt}. We use the super Lagrangian technique in Section \ref{sec:supLag} to develop a general algorithm to obtain the super Lagrange multipliers, which are necessary to perform the error estimates. We define the perturbed inverse problem in Section \ref{sec:pinv} and obtain the first order optimality conditions for it in Section \ref{sec:kktP}. In Section \ref{sec:errest} we derive the expression to estimate the error in the optimal solution for a general inverse problem. 
In Section \ref{sec:discretemodels} we present the discrete-time model framework. In Section \ref{sec:aposterioriErr} we present the error estimation methodology for discrete models. 
In Section \ref{sec:da} we present a detailed procedure to perfrom the error estimation for the data assimilation problem. We show the numerical results to support our theory for the heat equation and the shallow water model in spherical co-ordinates in Section \ref{sec:numexp}. The error estimates are statistically validated in Section \ref{sec:statValEsts}. Finally we give the concluding remarks in Section \ref{sec:conc}.
%%%%%%%%%%%%%%%%%%%%%%%%%%%%%%%%%%%%%%%%%%%%%%%%%%%%%%%%%%%%%%%%%%%%%%%%%%
\section{Inverse problems with continuous-time models} \label{sec:probdef}
%%%%%%%%%%%%%%%%%%%%%%%%%%%%%%%%%%%%%%%%%%%%%%%%%%%%%%%%%%%%%%%%%%%%%%%%%%
We consider a time-evolving physical system modeled by ordinary differential equations (ODEs):
\begin{equation}\label{eqn:ode}
\x' = f\left(t,\x,\theta\right),\quad t_0 \leq t \leq t_F, \quad \x(t_0) = \x_0(\theta)\,,
\end{equation}
where $t\in \mathbb{R}$ is time, $\x \in \mathbb{R}^n$ is the state vector, and $\theta \in \mathbb{R}^m$ is the vector of parameters. In many practical situations \eqref{eqn:ode} represents an evolutionary partial differential equation (PDE) after the semi--discretization in space. We call \eqref{eqn:ode} the continuous {\it forward model}.

A cost function defined on the solution and on the parameters of \eqref{eqn:ode} has the general form
\begin{equation}\label{eqn:cf}
 \J\left(\x, \theta\right) = \displaystyle\int\limits_{t_0}^{t_F}\, r\left(\x(t),\theta\right)\, \mathrm{d}t + w\left(\x(t_F),\theta\right).
\end{equation}
We consider the following inverse problem that seeks the optimal values of the model parameters:
\begin{equation}\label{eqn:ip}
 \begin{aligned}
 \theta^{\rm a}  =  &\underset{\theta} {\text{arg\, min}}\,
&& \J\left(\x, \theta\right) \\
  & \text{subject to}
  & & \text{\eqref{eqn:ode}\,.} \\
 \end{aligned}
\end{equation}
The inverse problem in \eqref{eqn:ip} is constrained by the dynamics of the system \eqref{eqn:ode}. Solving this system for a given value of the parameters finds the solution 
$\x(t,\theta)$. Using this solution in \eqref{eqn:ip} eliminates the constraints and leads to the equivalent  unconstrained problem
\begin{equation}\label{eqn:ipunc}
\theta^{\rm a}  =  \underset{\theta} {\text{arg\, min}} \,\mathcal{J}\left(\x(\theta), \theta\right),
\end{equation}
where $\mathcal{J}\left(\x(\theta), \theta\right)$ is the reduced cost function.
The problem \eqref{eqn:ip} or \eqref{eqn:ipunc} can be solved numerically using gradient based methods. The derivative information required for the computation of gradients and Hessian can be computed using sensitivity analysis \cite{Cao:2002,Cioaca:2012,Petzold:2004}.

We are interested in estimating the impact of observation and model errors on the optimal solution $\theta^{\rm a}$. Specifically, we will quantify the effect of errors on a certain quantity of interest (\qoi) defined by a scalar error functional $\mathcal{E} : \mathbb{R}^m \to \mathbb{R}$ that measures a certain aspect of the the optimal parameter value 
\begin{equation}\label{eqn:error_functionalCont}
\textnormal{\qoi} = \mathcal{E}\left(\theta^{\rm a}\right)\,.
\end{equation}
An example of error functional is the $k$-th component of the optimal parameter vector $\mathcal{E}\left(\theta^{\rm a}\right) = \theta^{\rm a}_{\rm k}$.

%%%%%%%%%%%%%%%%%%%%%%%%%%%%%%%%%%%%%%%%%%%%%%%%%%%%%%%%%%%%%%
\subsection{First order optimality conditions} \label{sec:kkt}
%%%%%%%%%%%%%%%%%%%%%%%%%%%%%%%%%%%%%%%%%%%%%%%%%%%%%%%%%%%%%%
The Lagrangian function associated with the cost function in \eqref{eqn:cf} and the constraint in \eqref{eqn:ode} is
\begin{equation}\label{eqn:Lag}
 \mathcal{L} = \displaystyle\int\limits_{t_0}^{t_F}\,r\left(\x(t),\theta\right)\, \mathrm{d}t + w\left(\x(t_F),\theta\right) - \displaystyle\int\limits_{t_0}^{t_F}\,\lambda^{\rm T}(t)\cdot\left(\x' - f(t,\x,\theta)\right)\,
 \mathrm{d}t.
\end{equation}
Setting to zero the variations of $\mathcal{L}$ with respect to the independent perturbations $\delta\lambda$, $\delta \x$, and $\delta\theta$
leads to the following optimality equations:
\begin{subequations}
\label{eqn:kkt}
\begin{align}
\label{eqn:odekkt}
\textnormal{forward model:}~~ & -\x' + f(t,\x,\theta) = 0, \\
\nonumber
& \quad t_0 \le t \le t_F, \quad \x(t_0) = \x_0\,, \\
\label{eqn:adjoint}
\textnormal{adjoint model:}~~ & \lambda' +  r_\x^{\rm T}\left(\x(t), \theta \right) + f_\x^{\rm T} \left(t,\x,\theta \right) \cdot \lambda = 0,  \\
\nonumber
& \quad t_F \leq t \leq t_0, \quad \lambda\left( t_F \right) = w_{\x}^{\rm T} \left( \x \left(t_F\right), \theta \right)\,, \\
\label{eqn:opt}
\textnormal{optimality:} ~~& \xi\left(t_0 \right) + \x_\theta^{\rm T}(t_0)\cdot \lambda (t_0) = 0\,, \\
\nonumber
&\textnormal{where}~~\xi' = - r_\theta^{\rm T} \left(\x(t),\theta \right) - f_{\theta}^{\rm T} (t,\x,\theta)\cdot \lambda, \\
\nonumber
& \qquad t_F \leq t \leq t_0, \quad \xi\left(t_F\right) = w^{\rm T}_{\theta}\left(\x\left(t_F \right), \theta \right)\,.
\end{align}
\end{subequations}
Equations \eqref{eqn:kkt} constitute the first order optimality conditions for the inverse problem \eqref{eqn:ip}. Subscripts denote partial derivatives, e.g.,
$f_\x=\partial f/\partial \x$. For a detailed derivation of the first order optimality conditions, please see the Appendix \ref{app:kkt}.
%%%%%%%%%%%%%%%%%%%%%%%%%%%%%%%%%%%%%%%%%%%%%%%%%%%%
\subsection{The super-Lagrangian} \label{sec:supLag}
%%%%%%%%%%%%%%%%%%%%%%%%%%%%%%%%%%%%%%%%%%%%%%%%%%%%
We follow the methodology discussed in \cite{Alexe:2011,becker2005mesh} to develop a posteriori error estimates applicable to our problem of interest. 

The Lagrangian associated with the error functional of the form \eqref{eqn:error_functionalCont} and the constraints posed by the first order optimality conditions \eqref{eqn:kkt} is:
\begin{eqnarray}\label{eqn:supLag}
 \mathcal{L^E} &= & \,\mathcal{E}\left(\theta^{\rm a}\right)
 -\displaystyle\int\limits_{t_0}^{t_F}\,\nu^{\rm T} \cdot \left(-\x' + f \right)\,\mathrm{d}t  - \nu^{\rm T}\left(t_0 \right) \cdot \left( \x(t_0) - \x_0\right) \\
  \nonumber
&& - \displaystyle\int\limits_{t_0}^{t_F}\,\mu^{\rm T} \cdot \left(\lambda'+r_{\x}^{\rm T}  +  f_{\x}^{\rm T} \cdot \lambda \right) \, \mathrm{d}t 
\nonumber
 - \mu^{\rm T}\left( t_F \right)  \cdot \left( \lambda\left( t_F \right) - w^{\rm T}_{\x} \left( \x\left(t_F \right), \theta \right) \right)\\
\nonumber
&& - \displaystyle\int\limits_{t_0}^{t_F}\,\zeta^{\rm T} \cdot \left(\xi' + r_\theta^T + f_{\theta}^T \cdot \lambda \right) \, \mathrm{d}t \\
\nonumber
&&-\zeta^{\rm T} \cdot \left(\xi\left(t_F \right) - w^{\rm T}_{\theta} \left( \x \left(t_F\right), \theta \right)  \right) \\
\nonumber
&&-\zeta^{\rm T} \cdot \left(\xi\left(t_0 \right) + \x_\theta^{\rm T}(t_0)\cdot \lambda (t_0) \right)\,. \nonumber
\end{eqnarray}
We have removed the arguments for convenience of notation. Here $\nu$, $\mu$, and $\zeta$ are the super--Lagrange multipliers associated with constraints \eqref{eqn:odekkt} (forward model), \eqref{eqn:adjoint} (adjoint model), and \eqref{eqn:opt} (optimality condition) respectively.
%%%%%%%%%%%%%%%%%%%%%%%%%%%%%%%%%%%%%%%%
\subsubsection{The tangent linear model}
%%%%%%%%%%%%%%%%%%%%%%%%%%%%%%%%%%%%%%%%
Taking the variations of \eqref{eqn:supLag} and imposing the stationarity condition $\nabla_{\lambda}\mathcal{L^E} = 0$ leads to the following {\it tangent linear model} (TLM): 
\begin{eqnarray}\label{eqn:tlmsup}
&& -  \mu' + f_{\x} \cdot \mu   + f_{\theta}  \cdot \zeta = 0, \quad t_0 \le t \le t_F; \\
\nonumber
&& \mu\left( t_0 \right) =   \x_\theta(t_0)\cdot \zeta\,.
\end{eqnarray}
%%%%%%%%%%%%%%%%%%%%%%%%%%%%%%%%%%%%%%%%%%%%%%%%%
\subsubsection{The second order adjoint equation}
%%%%%%%%%%%%%%%%%%%%%%%%%%%%%%%%%%%%%%%%%%%%%%%%%
The stationarity condition $\nabla_{\x}\mathcal{L^E} = 0$ leads to the following {\it second order adjoint} ODE (SOA):
\begin{eqnarray}\label{eqn:soasup}
 && \nu' +   f_{\x}^{\rm T} \cdot \nu +  r_{\x,\x} \cdot \mu  + \left(f_{\x,\x} \cdot \mu \right)^{\rm T} \cdot \lambda  \\ 
 \nonumber
&& \qquad +  r_{\theta,\x} \cdot \zeta +  \left(f_{\theta,\x} \cdot \zeta \right)^{\rm T} \cdot \lambda   =  0, \qquad t_F \ge t \ge t_0; \\
\nonumber
&& \nu \left(t_F \right) = w_{\theta,\x} \left(\x\left(t_F \right), \theta\right) \cdot \zeta  +  w_{\x,\x} \left(\x\left(t_F \right), \theta\right) \cdot \mu \left(t_F \right)
\end{eqnarray}
%
%%%%%%%%%%%%%%%%%%%%%%%%%%%%%%%%%%%%%%%
\subsubsection{The optimality equation}
%%%%%%%%%%%%%%%%%%%%%%%%%%%%%%%%%%%%%%%
The stationarity condition $\nabla_{\theta}\mathcal{L^E} = 0$ leads to the following  {\it optimality equation}:
\begin{equation} \label{jetheta}
 \left. \left( \frac{d^2}{d\theta^2}\, \mathcal{J}\left(\x(\theta),\theta\right) \right)\right|_{\theta^{\rm a}} \cdot \zeta = \mathcal{E}_{\theta},
\end{equation}
where the reduced Hessian-vector product in the direction of $\delta \theta$ is given by:
\begin{equation}\label{eqn:jtheta}
 \begin{split}
  \frac{d^2\,\mathcal{J}\left(\x(\theta),\theta\right)}{d\theta^2} \cdot \delta \theta 
   & = w^{\rm T}_{\theta, \x} \left(\x(t_F), \theta \right) \cdot \delta \x \left(t_F\right) + w_{\theta,\theta} \left(\x\left(t_F \right), \theta \right) \cdot \delta \theta \\
    & + \left(\frac{d\x_0}{d\theta}\right)^{\rm T}\cdot \nu(t_0) + \left( \frac{d^2\x_0}{d\theta^2} \delta \theta \right)\cdot \lambda(t_0)\\
    &+ \displaystyle \int \limits_{t_0}^{t_F}\, \left(f_\theta^{\rm T} \cdot \nu + \left(f_{\theta, \x} \cdot \delta \x\right)^{\rm T} \cdot \lambda + \left(f_{\theta,\theta}  \cdot \delta \theta \right)^{\rm T} \cdot \lambda \right) \, \mathrm{d}t \\
    &+\displaystyle \int \limits_{t_0}^{t_F}\,\left(r^{\rm T}_{\theta,\x} \cdot \delta \x + r_{\theta,\theta} \cdot \delta \theta \right) \mathrm{d}t\,.
 \end{split}
\end{equation}
The procedure to obtain the super Lagrange parameters $\zeta$, $\mu$ and $\nu$ is summarized in Algorithm \ref{alg:sup}. A detailed derivation of the super-Lagrange parameters is presented in Appendix \ref{app:supLag}.
\begin{algorithm}
\caption{SuperLagrangeMultipliers}\label{alg:sup}
\begin{algorithmic}[1]
\Procedure{SuperLagrangeMultipliers}{}
\State Solve the linear system \eqref{jetheta} to obtain $\zeta$.
\State Solve the tangent linear model \eqref{eqn:tlmsup} to obtain $\mu$.
\State Solve the second order adjoint equation \eqref{eqn:soasup} to obtain $\nu$.
\EndProcedure
\end{algorithmic}
\end{algorithm}
%%%%%%%%%%%%%%%%%%%%%%%%%%%%%%%%%%%%%%%%%%%%%%%%%%%%%%%
\subsection{Perturbed inverse problems}\label{sec:pinv}
%%%%%%%%%%%%%%%%%%%%%%%%%%%%%%%%%%%%%%%%%%%%%%%%%%%%%%%
In practice the forward model \eqref{eqn:ode} is inaccurate and subject to model errors. 
To describe this inaccuracy we consider a forward model that is marred by a time-- and state--dependent model error
\begin{equation}\label{eqn:odeerr}
 \widehat{\x}' = f\left(t,\widehat{\x},\theta \right) + \Delta f\left(t,\widehat{\x}\right)\,, \quad  \widehat{\x}\left(t_0\right) = \x_0 + \Delta\x_0\,.
\end{equation}
Furthermore, the noise in the data leads to errors $\Delta r$ and $\Delta w$ in the corresponding terms of the cost function \eqref{eqn:cf}. 
The inaccurate cost function is given by
\begin{equation}\label{eqn:cferr}
 \widehat{\mathcal{J}}\left(\widehat{\x},\theta\right) = \displaystyle\int\limits_{t_0}^{t_F}\, \left(r\left(\widehat{\x}(t),\theta\right)+\Delta r\right)\, \mathrm{d}t + w\left(\widehat{\x}(t_F),\theta\right) + \Delta w\,.
\end{equation}
Therefore in practice one solves the following perturbed inverse problem:
\begin{equation}\label{eqn:iperr}
 \begin{aligned}
 \widehat{\theta}^{\rm a}  =  &~\underset{\theta} {\text{arg\, min}}~~ \widehat{\J}\left(\widehat{\x},\theta \right) \\
  & \text{subject to}  ~~  \text{\eqref{eqn:odeerr}\,.} \\
 \end{aligned}
\end{equation}
%
%%%%%%%%%%%%%%%%%%%%%%%%%%%%%%%%%%%%%%%%%%%%%%%%%%%%%%%%%%%%%%%%%%%%%%%%%%%%%%%%%%%%%%%%%%%%%%%%%%%
\subsection{First order optimality conditions for the perturbed inverse problems} \label{sec:kktP}
%%%%%%%%%%%%%%%%%%%%%%%%%%%%%%%%%%%%%%%%%%%%%%%%%%%%%%%%%%%%%%%%%%%%%%%%%%%%%%%%%%%%%%%%%%%%%%%%%%%
The Lagrangian function associated with the cost function in \eqref{eqn:cferr} and the constraint in \eqref{eqn:odeerr} is
\begin{eqnarray}\label{eqn:LagP}
 \widehat{\mathcal{L}} &=& \displaystyle\int\limits_{t_0}^{t_F}\,\left(r\left(\widehat{\x}(t),\theta\right) + \Delta r \right)\, \mathrm{d}t + w\left(\widehat{\x}(t_F),\theta\right)+\Delta w \\
 \nonumber &&- \displaystyle\int\limits_{t_0}^{t_F}\,\lambda^{\rm T}(t)\cdot\left(\widehat{\x}' - f(t,\widehat{\x},\theta) -\Delta f\right)\,\mathrm{d}t\,.
\end{eqnarray}
Setting to zero the variations of $\widehat{\mathcal{L}}$ with respect to the independent perturbations $\delta\widehat{\lambda}$, $\delta \widehat{\x}$, and $\delta\theta$
leads to the following optimality equations:
\begin{subequations}
\label{eqn:kktP}
\begin{align}
\label{eqn:odekktP}
\textnormal{perturbed forward model:}~~ & -\widehat{\x}' + f\left(t,\widehat{\x}, \theta\right) +\Delta f\left(t, \widehat{\x}\right)  = 0, \\
\nonumber
& \quad t_0 \le t \le t_F, \quad \widehat{\x}(t_0) = \x_0 + \Delta\x_0\,, \\
\label{eqn:adjointP}
\textnormal{perturbed adjoint model:}~~ & \widehat{\lambda}' +  r_\x^{\rm T}\left(\widehat{\x}(t), \theta \right) + f_\x^{\rm T} \left(t,\widehat{\x},\theta \right) \cdot \widehat{\lambda} = 0,  \\
\nonumber
& \quad t_F \leq t \leq t_0, \quad \widehat{\lambda}\left( t_F \right) = w_\x^{\rm T}  + \Delta w_\x^{\rm T}  \,, \\
\label{eqn:optP}
\textnormal{perturbed optimality:} ~~& \widehat{\xi}\left(t_0 \right) + \widehat{\x}_\theta^{\rm T}(t_0)\cdot \widehat{\lambda} (t_0) = 0\,, \\
\nonumber
&\textnormal{where}~~\widehat{\xi}' = - r_\theta^{\rm T} \left(\widehat{\x}(t),\theta \right) - f_{\theta}^{\rm T} (t,\widehat{\x},\theta)\cdot \widehat{\lambda}, \\
\nonumber
& \qquad t_F \leq t \leq t_0, \quad \widehat{\xi}\left(t_F\right) = w^{\rm T}_{\theta} + \Delta w^{\rm T}_{\theta} \,.
\end{align}
\end{subequations}
Equations \eqref{eqn:kkt} constitute the first order optimality conditions for the inverse problem \eqref{eqn:iperr}. A detailed derivation of the first order optimality conditions is presented in the Appendix \ref{app:kkt}.

%%%%%%%%%%%%%%%%%%%%%%%%%%%%%%%%%%%%%%%%%%%%%%%%%%%%%%%%%%%%%%%%%%%%%%%%
\subsection{A posteriori error estimation methodology}\label{sec:errest}
%%%%%%%%%%%%%%%%%%%%%%%%%%%%%%%%%%%%%%%%%%%%%%%%%%%%%%%%%%%%%%%%%%%%%%%%
Our goal is to estimate the error in the optimal solution ${\widehat\theta}^{\rm a} - \theta^{\rm a}$. Specifically, we seek to estimate the errors in the quantity of interest $\mathcal{E}\left(\theta^{\rm a}\right)$	
\begin{equation}
\label{eqn:errorECont}
\Delta \mathcal{E}  = \mathcal{E}({\widehat\theta}^{\rm a}) - \mathcal{E}\left(\theta^{\rm a}\right)
\end{equation}
due to the errors in both the model and the data. The first order necessary conditions for the perturbed inverse problem \eqref{eqn:iperr} are given by the equations in \eqref{eqn:kktP}
and consist of the perturbed forward, adjoint, and optimality equations. 

The errors in the optimal solution \eqref{eqn:errorECont}  are the result of
errors in the adjoint model \eqref{eqn:adjoint}, in the forward model \eqref{eqn:odekkt}, and 
in the optimality equation \eqref{eqn:opt}, i.e., to differences between the perturbed and the perfect equations.
This leads to the following change in error functional resulting from model and data errors
\begin{equation}\label{eqn:errfuncerr}
 \Delta \mathcal{E} = \Delta \mathcal{E}_{\rm adj} + \Delta \mathcal{E}_{\rm fwd} + \Delta \mathcal{E}_{\rm opt}\,.
\end{equation}
The perturbed super-Lagrangian  can be written as:
\begin{eqnarray}\label{eqn:supLagPert}
 \mathcal{\widehat{L}_E} &= & \,\mathcal{E}(\widehat{\theta}^{\rm a}) 
  -\displaystyle\int\limits_{t_0}^{t_F}\,\nu^{\rm T} \cdot \left(-\widehat{\x}' + \widehat{f} + \Delta \widehat{f}\right)\,\mathrm{d}t    \\
\nonumber
&& - \nu^{\rm T}\left(t_0 \right) \cdot \left( \widehat{\x}(t_0) - \x_0 - \Delta\x_0 \right) \\
\nonumber
 && - \displaystyle\int\limits_{t_0}^{t_F}\,\mu^{\rm T} \cdot \left(\widehat{\lambda}'+\widehat{r}_\x^{\rm T} 
 + \Delta \widehat{r}_\x^{\rm T}  +  \left(\widehat{f}_\x + \Delta \widehat{f}_\x \right)^{\rm T} \cdot \widehat{\lambda} \right) \, \mathrm{d}t \\
\nonumber
&& - \mu^{\rm T}\left( t_F \right)  \cdot \left( \widehat{\lambda}\left( t_F \right) - \widehat{w}^{\rm T}_\x\left(\widehat{\x},\theta \right) - \Delta \widehat{w}^{\rm T}_\x\left(\widehat{\x},\theta \right) \right) \\
\nonumber
&& - \displaystyle\int\limits_{t_0}^{t_F}\,\zeta^{\rm T} \cdot \left(\widehat{\xi}' + \widehat{r}^{\rm T}_\theta\left(\widehat{\x},\theta\right) + \Delta \widehat{r}^{\rm T}_{\theta}\left(\widehat{\x},\theta\right) 
+ \left(\widehat{f}_{\theta} + \Delta \widehat{f}_{\theta} \right)^{\rm T} \cdot \widehat{\lambda} \right) \, \mathrm{d}t \\
\nonumber
&&-\zeta^{\rm T} \cdot \left(\widehat{\xi}\left(t_F \right) - \widehat{w}^{\rm T}_{\theta} - \Delta \widehat{w}^{\rm T}_{\theta}  \right) \\
\nonumber
&&-\zeta^{\rm T} \cdot \left(\widehat{\xi}\left(t_0 \right) + \widehat{\x}_\theta^{\rm T}(t_0)\cdot \widehat{\lambda} (t_0) \right).
\end{eqnarray}
We have denoted by hat the functions evaluated at $\widehat{\x}$, e.g., $\widehat{f}=f(t,\widehat{\x},\theta)$.
The gradient of the super-Lagrangian at the optimal solution is the same as the gradient of the error functional, both being zero. Hence we have, 
\begin{equation}\label{eqn:dE}
 \Delta\mathcal{E} \approx \widehat{\mathcal{L}}_{\mathcal{E}} - \mathcal{L_E}\,.
\end{equation}
The approximate contribution to the error brought by the adjoint model is given by
\begin{equation}\label{eqn:deadj}
 \Delta \mathcal{E}_{\rm adj} \approx \displaystyle \int_{t_0}^{t_F}\mu ^{\rm T} \cdot \left(\Delta \widehat{r}^{\rm T}_{\x} + \Delta \widehat{f}_{\x}^{\rm T} \cdot \widehat{\lambda} \right) \, dt- \mu ^{\rm T} \cdot \Delta \widehat{w}^{\rm T}_{\x} |_{t_F}\,.
\end{equation}
The approximate contribution to the error brought by the forward model only depends on model errors, and is given by:
\begin{equation}\label{eqn:defwd}
 \Delta \mathcal{E}_{\rm fwd} \approx \displaystyle \int_{t_0}^{t_F}\nu ^{\rm T} \cdot \Delta \widehat{f} \, dt\,.
\end{equation}
The contribution to the error by the optimality equation can be computed from equation \eqref{eqn:supLag}, and is given by:
\begin{equation}\label{eqn:deopt}
 \Delta \mathcal{E}_{\rm opt} \approx \displaystyle \int_{t_0}^{t_F}\zeta^{\rm T} \cdot \left( \Delta \widehat{r}^{\rm T}_\theta - \Delta \widehat{f}_\theta^{\rm T} \cdot \widehat{\lambda} \right) \, dt - \zeta^{\rm T} \cdot \Delta \widehat{w}^{\rm T}_{\theta} |_{t_F}\,.
\end{equation}
Appendix \ref{App:fdm} demonstrates that equations \eqref{eqn:deadj}, \eqref{eqn:defwd}, and  \eqref{eqn:deopt} correspond to first order error estimates.

% We note that the paper \cite{Gejadze:2008} describes an algorithm for the evaluation of the error covariance matrix associated with the optimal solution when there are errors in the data. 
% There is a direct relationship between the above a posteriori error estimate and \cite{Gejadze:2008}. In this work we can recover the error covariance matrix column by column by successively solving the system in \eqref{jetheta} for several error functionals. Specifically, if we take $\mathcal{E}$ to be one solution component \eqref{eqn:error_functionalCont}, $\mathcal{E}_{\theta}$ becomes the canonical basis vector $\mathbbm{e}_{\rm k}$. Application of Algorithm 1 then recovers the $k^{th}$ column of the error covariance matrix by solving the linear system in \eqref{jetheta}. In the numerical experiments we demonstrate this methodology to compute the contributions of individual grid points on the optimal solution. 
%%%%%%%%%%%%%%%%%%%%%%%%%%%%%%%%%%%%%%%%%%%%%%%%%%%%%%%%%%%%%%%%%%%%%%%%%%%%%%
\section{Inverse problems with discrete-time models}\label{sec:discretemodels}
%%%%%%%%%%%%%%%%%%%%%%%%%%%%%%%%%%%%%%%%%%%%%%%%%%%%%%%%%%%%%%%%%%%%%%%%%%%%%%
Consider a time-evolving system governed by the following discrete-time model
\begin{equation}
\label{eqn:model}
\x _{k+1} = \Model_{k,k+1}(\x_k,\theta) , \quad k=0,\dots , N-1\,,  \quad \x_0 = \x_0(\theta)\,,
\end{equation}
where $\x_k \in \mathbbm{R}^n$ is the state vector at time $t_k$, $\Model_{k,k+1}$ is the solution operator that advances the state vector from time $t_k$ to $t_{k+1}$, 
and $\theta \in \mathbbm{R}^m$ is the vector of model parameters. At each time $t_k$ the model state approximates the truth, i.e., the state of the physical system, $\x_k \approx \x(t_k)$.

A cost function defined on the solution and on the parameters of \eqref{eqn:model} has the general form
\begin{equation}
\label{eqn:cfd}
 \J\left(\x, \theta\right) = \displaystyle \sum_{k=0}^N\, r_k\left(\x_k,\theta\right).
\end{equation}
For example, in four dimensional variational data assimilation \cite{kalnay2003,Sandu:2011A} the cost function is
 \begin{eqnarray}
 \label{eqn:4D-Var}
 {\J}(\x_0) &=&
 \frac{1}{2}\, \left(\x_0-\x^{\rm b}_0(\theta)\right)^{\rm T} \, \B_0^{-1}(\theta)\, \left(\x_0-\x^{\rm b}_0(\theta)\right)
 \\
\nonumber
 && +
 \sum_{k=0}^N\; \frac{1}{2}\, \left(\mathcal{H}_k (\x_k,\theta) - \y_k\right)^{\rm T}
\R_k^{-1}(\theta)\left(\mathcal{H}_k (\x_k,\theta) - \y_k\right)\,,
 \end{eqnarray}
where, $\x_0^{\rm b}$ is the background state at the initial time (the prior knowledge of the initial conditions), $\mathbf{B}_0$ is the covariance matrix of the background errors, $\y_k$ is the vector of observations at time $t_k$ and $\mathbf{R}_k$ is the corresponding observation error covariance matrix. The observation operators $\mathcal{H}_k$ map the model state space onto the observation space. 
The cost function \eqref{eqn:4D-Var} measures the departure of the initial state $\x_0$ from the background initial state, as well as the discrepancy between the model predictions and measurements of reality $\y_k$ at $t_k$ for $k \ge 1$. The norms of the differences are weighted by the corresponding inverse background error covariance matrices.

An inverse problem that seeks the optimal values of the model parameters is formulated as follows:
\begin{equation}
\label{eqn:ipdiscrete}
 \theta^{\rm a}  =  \underset{\theta} {\text{arg\,min}} ~~ \J\left(\x, \theta\right) \quad
   \text{subject to \eqref{eqn:model}\,.} 
\end{equation}
For example the optimal parameter values lead to a best fit between model predictions and measurements, in a least squares sense.
%

%%%%%%%%%%%%%%%%%%%%%%%%%%%%%%%%%%%%%%%%%%%%%%%%%%%%%%%%%%%%%%%%%
\subsection{First order optimality conditions}\label{sec:discKKT}
%%%%%%%%%%%%%%%%%%%%%%%%%%%%%%%%%%%%%%%%%%%%%%%%%%%%%%%%%%%%%%%%%
The Lagrangian function associated with the problem \eqref{eqn:ipdiscrete} is
\begin{eqnarray}
 \label{eqn:LagDisc}
 \mathcal{L} &=& \displaystyle \sum_{k=0}^{N-1}\, \left(r_k\left(\x_k,\theta\right) - \lambda^{\rm T}_{k+1} \cdot \left(\x _{k+1} - \Model_{k,k+1}(\x_k,\theta)\right)\right) \\ 
 \nonumber &&+ r_N\left(\x_N, \theta\right) - \lambda_0^{\rm T} \cdot \left(\x_0 - \x_0 \left(\theta\right)\right).
\end{eqnarray}
Consider the following Jacobians of the model solution operator with respect to the state and with respect to parameters, respectively:
\begin{equation}
\label{eqn:model-Jacobians}
\M_{k,k+1}(\x,\theta) := \bigl(\Model_{k,k+1}(\x,\theta)\bigr)_{\x}, \qquad \mathfrak{M}_{k,k+1}(\x,\theta) := \bigl(\Model_{k,k+1}(\x,\theta)\bigr)_{\theta}.
\end{equation}
Consider also the Jacobians of the cost function terms
\begin{equation}
\label{eqn:cost-Jacobians}
 \left(r_k\right)_{\x_k} := \left. \left(r_k\left(\x, \theta\right)\right)_{\x} \, \right|_{\x = \x_k},
\quad
\left(r_k\right)_{\theta} := \left. \left(r_k\left(\x, \theta\right)\right)_{\theta}\, \right|_{\x = \x_k}.
\end{equation}
Setting to zero the variations of $\mathcal{L}$ with respect to the independent perturbations $\delta\lambda$, $\delta \x$, and $\delta\theta$ leads to the first order optimality conditions for the inverse problem \eqref{eqn:ipdiscrete}: 
\begin{subequations}
\label{eqn:kktDisc}
\begin{align}
\label{eqn:odekktDisc}
\textnormal{forward model:}~~ & 0 =  \x _{k+1}-\Model_{k,k+1}(\x_k,\theta), \quad k=0,\dots , N-1\,;  \\
\label{eqn:adjointDisc}
\textnormal{adjoint model:}~~& 0 = \lambda_{N}-\left(r_N \right)^{\rm T}_{\x_N} \,,  \\
\nonumber
& 0 = {\lambda}_k - \M_{k,k+1}^{\rm T}\, \lambda_{k+1} - \left(r_k\right)^{\rm T}_{\x_k} , \quad k=N-1,\dots , 0; \\
\label{eqn:optDisc}
\textnormal{optimality:} ~~& 0 = (\x_{0})_{\theta}^{\rm T}\lambda_0  + \displaystyle \sum_{k=0}^{N} \,\left(r_k\right)^{\rm T}_{\theta} + \displaystyle \sum_{k=0}^{N-1}\, \mathfrak{M}_{k,k+1}^{\rm T} \lambda_{k+1}\,.
\end{align}
\end{subequations}
Here ${\lambda}_k \in \mathbbm{R}^n$ are the adjoint variables. 
A detailed derivation of the first order optimality conditions can be found in the Appendix A of \cite{Rao:2014_Supp}. 
%
%%%%%%%%%%%%%%%%%%%%%%%%%%%%%%%%%%%%%%%%%%%%%%%%%%%%%%%%%%%%%%%%%%%%%%%%%%%%%%%%%%%%%%
\subsection{Perturbed inverse problem with discrete-time models}\label{sec:errestDisc}
%%%%%%%%%%%%%%%%%%%%%%%%%%%%%%%%%%%%%%%%%%%%%%%%%%%%%%%%%%%%%%%%%%%%%%%%%%%%%%%%%%%%%%
In practice the evolution of the physical system is represented by the imperfect discrete model 
\begin{equation}
\label{eqn:model-p}
\widehat{\x} _{k+1} = \Model_{k,k+1}(\widehat{\x}_k, \theta) + \Delta\widehat{\x}_{k+1}(\widehat{\x}_k, \theta), \quad k=0,1,\dots , N-1\,.
\end{equation}
Errors in the data lead to the following perturbed cost function:
\begin{equation}\label{eqn:cf-p}
\widehat{\J}\left(\widehat{\x}, \theta\right) = \displaystyle \sum_{k=0}^N\, \bigl( r_k\left(\widehat{\x}_k,\theta\right) + \Delta\widehat{r}_k\left(\widehat{\x}_k,\theta\right) \bigr).
\end{equation}
The perturbed inverse problem solved in practice reads:
\begin{equation}\label{eqn:ipdiscreteErr}
 \widehat{\theta}^{\rm a}  =  \underset{\theta} {\text{arg\,min}}~~ \widehat{\J}\left( \widehat{\x}, \theta\right)\quad
  \text{subject to \eqref{eqn:model-p}\,.}
\end{equation}
%
%Consider the perturbed Lagrangian associated with the cost function \eqref{eqn:cf-p} with the imperfect model \eqref{eqn:model-p} as the constraint:
%%
%\begin{eqnarray}
%  \label{eqn:LagDiscP}
% \mathcal{\widehat{L}} &=& \displaystyle \sum_{k=0}^{N-1}\, \left(r_k\left(\widehat{\x}_k,\theta\right) + \Delta\widehat{r}_k\left(\widehat{\x}_k,\theta\right)\right)\\
% \nonumber && - \widehat{\lambda}^{\rm T}_{k+1} \cdot \left(\widehat{\x} _{k+1} - \Model_{k,k+1}(\widehat{\x}_k, \theta) - \Delta \x_{k+1}\right)\\
% \nonumber && + r^{\rm T}_N\left(\widehat{\x}_N, \theta\right) + \Delta\widehat{r}^{\rm T}_N\left(\widehat{\x}_N, \theta\right) \,  - \widehat{\lambda}_0^{\rm T} \cdot \left(\widehat{\x}_0 - \widehat{\x}_0 \left(\theta\right)\right)
%\end{eqnarray}
%
We consider the model Jacobians \eqref{eqn:model-Jacobians} evaluated at the perturbed state and parameters:
\[
\widehat{\M}_{k,k+1} := \M_{k,k+1}(\widehat{\x},\theta) , \qquad \widehat{\mathfrak{M}}_{k,k+1} := \mathfrak{M}_{k,k+1}(\widehat{\x},\theta).
\]
We also consider the cost function Jacobians \eqref{eqn:cost-Jacobians} evaluated at the perturbed state and parameters:
\[
 \left(\widehat{r}_k\right)_{\widehat{\x}_k} := \left. \left(r_k\left(\x, \theta\right)\right)_{\x} \, \right|_{\x = \widehat{\x}_k},
\quad
\left(\widehat{r}_k\right)_{\theta}:= \left. \left(r_k\left(\x, \theta\right)\right)_{\theta} \, \right|_{\x = \widehat{\x}_k}.
\]
%
%Setting to zero the variations of $\mathcal{L}$ with respect to the independent perturbations $\delta\widehat{\lambda}$, $\delta \widehat{\x}$, and $\delta\theta$
%leads to the following optimality equations.
%
The first order optimality conditions for the perturbed inverse problem \eqref{eqn:ipdiscreteErr} are: 
\begin{subequations}
\label{eqn:kktDiscP}
\begin{alignat}{3}
\label{eqn:odekktDiscP}
\textnormal{forward model:}\qquad &  \Delta \widehat{\x}_{k+1} &=  \widehat{\x}_{k+1}-\Model_{k,k+1}(\widehat{\x}_k, \theta), \quad k=0,\dots , N-1; 
\end{alignat}
\begin{alignat}{4}
\label{eqn:adjointDiscP}
&\textnormal{adjoint model:}& \left(\Delta\widehat{r}_N\right)^{\rm T}_{\widehat{\x}_N} &= \widehat{\lambda}_{N}-\left(\widehat{r}_N\right)^{\rm T}_{\widehat{\x}_N},   \\
\nonumber
&&\left(\Delta\widehat{r}_k\right)^{\rm T}_{\widehat{\x}_k} 
+ \left(\Delta \widehat{\x}_{k+1}\right)^{\rm T}_{\widehat{\x}_k}\, \widehat{\lambda}_{k+1} &= \widehat{{\lambda}}_k - \widehat{\M}_{k,k+1}^{\rm T}\, \widehat{\lambda}_{k+1} -  \left(\widehat{r}_k\right)^{\rm T}_{\widehat{\x}_k} \\
\nonumber &&& \qquad \,  k=N-1,\dots , 0; 
\end{alignat}
\begin{alignat}{4}
\label{eqn:optDiscP}
&\textnormal{optimality:} ~~&  \displaystyle \sum_{k=0}^{N} \,  \left( \Delta \widehat{r}_k \right)^{\rm T}_{\theta}
- \sum_{k=0}^{N-1}\, \left(\Delta \widehat{\x}_{k+1}\right)^{\rm T}_{\theta}\, \widehat{\lambda}_{k+1}
&= \left(\widehat{\x}_{0}\right)_{\theta}^{\rm T}\widehat{\lambda}_0  + \displaystyle \sum_{k=0}^{N} \,\left(\widehat{r}_k\right)^{\rm T}_{\theta}  \\
\nonumber &&& \qquad + \displaystyle \sum_{k=0}^{N-1}\, \widehat{\mathfrak{M}}_{k,k+1}^{\rm T} \widehat{\lambda}_{k+1}.
\end{alignat}
\end{subequations}

The perturbed optimality conditions \eqref{eqn:kktDiscP} differ in two ways from the ideal optimality conditions \eqref{eqn:kktDisc}.
First, the perturbations due to the error terms $\Delta\x$ and $\Delta\widehat{r}$ appear on the left hand side as residuals in each of the forward \eqref{eqn:odekktDiscP}, adjoint \eqref{eqn:adjointDiscP}, and optimality equations \eqref{eqn:optDiscP}.
Next, the linearizations in \eqref{eqn:adjointDiscP} and \eqref{eqn:optDiscP} are performed about the perturbed solution $\widehat{\x}$ and $\widehat{\theta}^{\rm a}$, while the linearizations in \eqref{eqn:adjointDisc} and \eqref{eqn:optDisc} are performed about the ideal solution $\x$ and $\theta^{\rm a}$.

%%%%%%%%%%%%%%%%%%%%%%%%%%%%%%%%%%%%%%%%%%%%%%%%%%%%%%%%%%%%%%%%%%%%%%%%
\subsection{Quantity of interest}\label{sec:qoi}
%%%%%%%%%%%%%%%%%%%%%%%%%%%%%%%%%%%%%%%%%%%%%%%%%%%%%%%%%%%%%%%%%%%%%%%%
Consider a quantity of interest (\qoi) defined by a scalar functional $\mathcal{E} : \mathbb{R}^m \to \mathbb{R}$ that measures a certain aspect of the the optimal parameter value 
\begin{equation}
\label{eqn:error_functional}
\textnormal{\qoi} = \mathcal{E}\left(\theta^{\rm a}\right)\,.
\end{equation}
An example of error functional \eqref{eqn:error_functional} is the $\ell$-th component of the optimal parameter vector, $\mathcal{E}\left(\theta^{\rm a}\right) = \theta^{\rm a}_{\rm \ell}$.

We are interested in estimating the impact of observation and model errors on the optimal solution $\theta^{\rm a}$, or, more specifically, the error impact on the aspect of $\theta^{\rm a}$ captured by the \qoi. The error in the \qoi is
\begin{equation}
\label{eqn:errorE}
\Delta \mathcal{E}  = \mathcal{E}({\widehat\theta}^{\rm a}) - \mathcal{E}\left(\theta^{\rm a}\right)
\end{equation}
where $\widehat{\theta}^{\rm a}$ and $\theta^{\rm a}$ are the solutions of the perturbed inverse problem \eqref{eqn:ipdiscreteErr} 
and of the ideal inverse problem \eqref{eqn:ipdiscrete}, respectively. 

%%%%%%%%%%%%%%%%%%%%%%%%%%%%%%%%%%%%%%%%%%%%%%%%%%%%%%%%%%%%%
\section{Aposteriori error estimation}\label{sec:aposterioriErr}
%%%%%%%%%%%%%%%%%%%%%%%%%%%%%%%%%%%%%%%%%%%%%%%%%%%%%%%%%%%%%

The ideal optimal solution $\theta^{\rm a}$ is obtained (in principle) by solving the nonlinear system \eqref{eqn:kktDisc},
while the perturbed optimal solution $\widehat{\theta}^{\rm a}$ is obtained by solving the system \eqref{eqn:kktDiscP}.
We have seen that \eqref{eqn:kktDiscP} is obtained from \eqref{eqn:kktDisc} by adding residuals to each of the optimality equations. The a posteriori error estimate quantifies, to first order, the impact of these residuals on the solution of the nonlinear system \eqref{eqn:kktDisc}. The methodology presented below follows the approach discussed in \cite{Alexe:2011,becker2005mesh}.

%%%%%%%%%%%%%%%%%%%%%%%%%%%%%%%%%%%%%%%%%%%%%%%%%%%%%%%%%%%%%
\subsection{The error estimation procedure}\label{sec:estimation}
%%%%%%%%%%%%%%%%%%%%%%%%%%%%%%%%%%%%%%%%%%%%%%%%%%%%%%%%%%%%%

It is useful to consider the reduced cost function \eqref{eqn:cfd}
\begin{equation}
\label{eqn:cfd-reduced}
j(\theta) =  \J\left(\x(\theta), \theta\right) = \displaystyle \sum_{k=0}^N\, r_k\left(\x_k(\theta),\theta\right)
\end{equation}
with the solution dependency on the parameters given by the model \eqref{eqn:model}. 

\begin{theorem}[A posteriori error estimates]\label{thm:existence}
Assume that the model operator $\mathcal{M}$ and the functions $r_k$ are twice continuously differentiable. Assume also that reduced Hessian $(\nabla^2_{\theta,\theta} j)(\theta^{\rm a}) \in \mathbbm{R}^{m \times m}$ evaluated at the minimizer of \eqref{eqn:ipdiscrete} is positive definite.

Then there exist ``impact factors'' $\zeta \in \mathbb{R}^m$, $\mu_{k}\in \mathbb{R}^n$ for $k=0,\dots,N$, and $\nu_{k}\in \mathbb{R}^n$ for $k=0,\dots,N$
such that the error in the \qoi is approximated to first order by the formula:
\begin{subequations}
\label{eqn:E-error-estimate}
\begin{eqnarray}
\label{eqn:deFinal}
 \Delta \mathcal{E} &\approx&   \Delta \mathcal{E}^{\rm est}  = \Delta\mathcal{E}_{\rm fwd} + \Delta\mathcal{E}_{\rm adj} + \Delta \mathcal{E}_{\rm opt},
\end{eqnarray}
where the three terms are the contributions of errors in the forward model, adjoint model, and optimality equation, respectively.
Specifically, the estimated contribution of the error in the forward model to the error in \qoi is:
\begin{eqnarray}
\label{eqn:de-fwdDisc}
  \Delta\mathcal{E}_{\rm fwd} & = & \displaystyle\sum_{k=0}^{N-1}\,\nu_{k+1}^{\rm T} \cdot \Delta \widehat{\x}_{k+1} \,.
\end{eqnarray}
Similarly, the estimated contribution of the adjoint model error to the error in \qoi is:
 \begin{eqnarray}
  \label{eqn:de-adjDisc}
  \Delta\mathcal{E}_{\rm adj} & = & \sum_{k=0}^{N}\mu_k^{\rm T} \cdot \left(\Delta\widehat{r}_k \right)^{\rm T}_{\x_k} +  \sum_{k=0}^{N-1}\mu_k^{\rm T} \cdot \left(\Delta \widehat{\x}_{k+1}\right)_{{\x}_k}^{\rm T} \widehat{\lambda}_{k+1}  \,.
 \end{eqnarray}
Finally, the contribution of the error in the optimality equation is given by
\begin{eqnarray}
 \label{eqn:de-optDisc}
 \Delta \mathcal{E}_{\rm opt} & = & \zeta^{\rm T} \cdot \left(\sum_{k=0}^{N} \,  \left( \Delta \widehat{r}_k \right)^{\rm T}_{\theta}
- \sum_{k=0}^{N-1}\, \left(\Delta \widehat{\x}_{k+1}\right)^{\rm T}_{\theta}\, \widehat{\lambda}_{k+1} \right) \,.
\end{eqnarray}
\end{subequations}
\end{theorem}

\begin{proof}
A discrete super-Lagrangian associated with the scalar functional  \eqref{eqn:errorE} and with the constraints posed by the first order optimality conditions \eqref{eqn:kktDisc} is
defined as follows: 
\begin{eqnarray}
\nonumber
\mathcal{L^E}(\theta,\x,\lambda,\mu,\nu,\zeta) &=& \mathcal{E}(\theta) - \nu_0^{\rm T} \cdot\left(\x_0 - \x_0\left(\theta\right)\right) - \sum_{k=0}^{N-1} \nu_{k+1}^{\rm T} \cdot \left(\x _{k+1}-\Model_{k,k+1}(\x_k,\theta) \right)  \\
\label{eqn:supLagDisc}
 &&- \mu_N^{\rm T} \cdot \left(  \lambda_{N}- \left(r_N \right)^{\rm T}_{\x_N}\right) 
  -\sum_{k=0}^{N-1} \mu_k^{\rm T} \cdot \left( {\lambda}_k - \M_{k,k+1}^{\rm T}\, \lambda_{k+1} - \left(r_k\right)^{\rm T}_{\x_k} \right) \\
\nonumber
 &&  -\zeta^{\rm T} \cdot \left( \left(\x_{0}\right)_{\theta}^{\rm T}\lambda_0 + \displaystyle \sum_{k=0}^{N} \,\left(r_k \right)_{\theta}+ \displaystyle \sum_{k=0}^{N-1}\, \mathfrak{M}_{k,k+1}^{\rm T} \lambda_{k+1}\,\right) \,.
\end{eqnarray}
Consider a stationary point $(\theta^{\rm a},\x,\lambda,\mu,\nu,\zeta)$ of the super-Lagrangian $\mathcal{L^E}$ 
\begin{equation}
\label{eqn:stationarity}
\delta \mathcal{L^E} \left|_{(\theta^{\rm a},\x,\lambda,\mu,\nu,\zeta)} \right.  = 0.
\end{equation}
Setting to zero the variations of \eqref{eqn:supLagDisc} with respect to $\mu,\nu,\zeta$ shows that the parameter vector $\theta$, the forward solution $\x$, and the adjoint solution $\lambda$ satisfy the first order optimality conditions \eqref{eqn:kktDisc}. Consequently $\{\theta^{\rm a},\x=\x(\theta^{\rm a}),\lambda=\lambda(\theta^{\rm a})\}$ is the solution of the inverse problem \eqref{eqn:ipdiscrete}. The super-Lagrange multipliers $\zeta$, $\nu$, and $\mu$ for a stationary point of the super-Lagrangian are calculated by setting to zero the variations of  \eqref{eqn:supLagDisc} with respect to $\theta,\x,\lambda$, as discussed in section \ref{sec:supLagDisc}. From \eqref{eqn:supLagDisc} we have that
\begin{equation}
\label{eqn:supLagDisca}
\mathcal{L^E}(\theta^{\rm a},\x,\lambda,\mu,\nu,\zeta) = \mathcal{E}(\theta^{\rm a}).
\end{equation}

We now evaluate \eqref{eqn:supLagDisc} at the solution $\{\widehat{\theta}^{\rm a},\widehat{\x}=\widehat{\x}(\widehat{\theta}^{\rm a}),\widehat{\lambda}=\widehat{\lambda}(\widehat{\theta}^{\rm a})\}$ of the perturbed inverse problem. The super-multipliers $\zeta$, $\eta$, and $\mu$ are not changed and they correspond to the stationary point at the ideal solution \eqref{eqn:stationarity}. We have:
\begin{eqnarray}
\nonumber
\mathcal{L^E}(\widehat{\theta}^{\rm a},\widehat{\x},\widehat{\lambda},\mu,\nu,\zeta) &=& \mathcal{E}(\widehat{\theta}^{\rm a}) - \sum_{k=0}^{N-1} \nu_{k+1}^{\rm T} \cdot  \left(\widehat{\x} _{k+1}-\Model_{k,k+1}(\widehat{\x}_k,\widehat{\theta}^{\rm a}) \right)   \\
 &&- \nu_0^{\rm T} \cdot \left(\widehat{\x}_0 - \x_0(\widehat{\theta}^{\rm a})\right) - \mu_N^{\rm T} \cdot \left(  \widehat{\lambda}_{N} - \left(\widehat{r}_N \right)^{\rm T}_{\widehat{\x}_N} \right)  \nonumber \\
\label{eqn:supLagDiscPert1}
&& -\sum_{k=0}^{N-1} \mu_k^{\rm T} \cdot \left( {\widehat{\lambda}}_k   - \widehat{\M}_{k,k+1}^{\rm T} \, \widehat{\lambda}_{k+1} - \left(\widehat{r}_k\right)^{\rm T}_{\widehat{\x}_k}  \right) \\
\nonumber
 \nonumber &&-\zeta^{\rm T} \cdot \left( \left(\widehat{\x}_{0}\right)_{\theta}^{\rm T}\widehat{\lambda}_0  + \displaystyle \sum_{k=0}^{N} \,\left(\widehat{r}_k \right)^{\rm T}_{\theta} + \sum_{k=0}^{N-1}\, \left(\widehat{\mathfrak{M}}_{k,k+1}^{\rm T} \right) \widehat{\lambda}_{k+1}\right) \,.
\end{eqnarray}

The perturbed inverse problem solution  $\{\widehat{\theta}^{\rm a},\widehat{\x},\widehat{\lambda}\}$ satisfies the perturbed first order optimality conditions \eqref{eqn:kktDiscP}. Substituting \eqref{eqn:kktDiscP} in \eqref{eqn:supLagDiscPert1} leads to
\begin{eqnarray}
\nonumber
\mathcal{L^E}(\widehat{\theta}^{\rm a},\widehat{\x},\widehat{\lambda},\mu,\nu,\zeta) &=& \mathcal{E}(\widehat{\theta}^{\rm a}) - \sum_{k=0}^{N-1} \nu_{k+1}^{\rm T} \cdot  \Delta\widehat{\x} _{k+1}  - \mu_N^{\rm T} \cdot \left(\Delta\widehat{r}_N\right)^{\rm T}_{\widehat{\x}_N}   \nonumber \\
\label{eqn:supLagDiscPert2}
&& -\sum_{k=0}^{N-1} \mu_k^{\rm T} \cdot \left( \left(\Delta\widehat{r}_k\right)^{\rm T}_{\widehat{\x}_k} 
+ \left(\Delta \widehat{\x}_{k+1}\right)^{\rm T}_{\widehat{\x}_k}\, \widehat{\lambda}_{k+1}  \right) \\
\nonumber
 \nonumber &&-\zeta^{\rm T} \cdot \left(\sum_{k=0}^{N} \,  \left( \Delta \widehat{r}_k \right)^{\rm T}_{\theta}
- \sum_{k=0}^{N-1}\, \left(\Delta \widehat{\x}_{k+1}\right)^{\rm T}_{\theta}\, \widehat{\lambda}_{k+1} \right) \,.
\end{eqnarray}

Since the super-Lagrangian is stationary at $(\theta^{\rm a},\x,\lambda,\mu,\nu,\zeta)$ its variation vanishes \eqref{eqn:stationarity}, therefore to first order it holds that 
\begin{equation}
\label{eqn:stationarity-delta}
\Delta\mathcal{L^E}=\mathcal{L^E}(\widehat{\theta}^{\rm a},\widehat{\x},\widehat{\lambda},\mu,\nu,\zeta) - 
\mathcal{L^E}(\theta^{\rm a},\x,\lambda,\mu,\nu,\zeta) \approx 0.
\end{equation}
Subtracting \eqref{eqn:supLagDisca} from \eqref{eqn:supLagDiscPert2} and using the stationarity relation \eqref{eqn:stationarity-delta}  leads to the  error estimate
\eqref{eqn:E-error-estimate}.

The existence of the super-Lagrange multipliers follows from Theorem \ref{thm:calculation} discussed in the next section.
Specifically, the Hessian equation \eqref{eqn:supLagParamsDisc} has a unique solution, and so do the tangent linear model \eqref{eqn:TLMDisc}
and the second order adjoint model \eqref{eqn:soaDisc}. The multipliers exist and can be calculated by Algorithm \ref{alg:discrete}.
\end{proof}

%%%%%%%%%%%%%%%%%%%%%%%%%%%%%%%%%%%%%%%%%%%%%%%%%%%%%%%%%%%%%
\subsection{Calculation of super--Lagrange multipliers}\label{sec:supLagDisc}
%%%%%%%%%%%%%%%%%%%%%%%%%%%%%%%%%%%%%%%%%%%%%%%%%%%%%%%%%%%%%

\begin{theorem}[Calculation of impact factors]\label{thm:calculation}
When the assumptions of Theorem \ref{thm:existence} hold the super-Lagrange multipliers corresponding to a stationary point of \eqref{eqn:supLagDisc} are computed via the following steps. First, solve the following linear system for the multiplier $\zeta  \in \mathbbm{R}^m$:
\begin{subequations}
\label{eqn:supLagParamsDisc}
\begin{equation}
\label{eqn:linearSys}
 (\nabla^2_{\theta,\theta} j)(\theta^{\rm a}) \cdot \zeta = \mathcal{E}_{\theta}^T\,,
 \end{equation}
whose matrix  is the reduced Hessian $\nabla^2_{\theta,\theta} j \in \mathbbm{R}^{m \times m}$ evaluated at the minimizer $\theta^{\rm a}$.
We call \eqref{eqn:linearSys} the ``Hessian equation''. 
Next, solve the following tangent linear model (TLM) for the multipliers $\mu_k \in \mathbb{R}^n$, $k=0,\dots,N$:
\begin{eqnarray}
\label{eqn:TLMDisc}
\mu_0 &=& - \left(\x_{0}\right)_{\theta}\zeta \,;  \\
\nonumber
\mu_k &=& \M_{k-1,k}\,  \mu_{k-1} - \mathfrak{M}_{k-1,k}\, \zeta ,   \qquad k=1,\dots , N\,.
\end{eqnarray}
%
%The initial condition and the forcing term in the TLM \eqref{eqn:TLMDisc} depend on $\zeta$.
%
Finally, solve the following second order adjoint model (SOA) for the multipliers $\nu_k \in \mathbb{R}^n$, $k=N,\dots,0$:
\begin{eqnarray}
\label{eqn:soaDisc}
\nonumber
\nu_N &=& \left(r_N\right)_{\x_N, \x_N} \mu_N -  \left(r_N\right)_{\theta, \x_N} \zeta\,;  \\
\nu_k &=&  \M_{k,k+1}^{\rm T} \nu_{k+1} + \left(\M_{k,k+1}^{\rm T} \lambda_{k+1}\right)^{\rm T}_{\x_k}\mu_{k}\\
\nonumber~~&& - \left(r_k \right)_{\theta, \x_k} \zeta - \left(\mathfrak{M}^{\rm T}_{k,k+1}  \lambda_{k+1}\right)^{\rm T}_{\x_k} \,\zeta,  \qquad k = N-1,\dots,0\,.
\end{eqnarray}
\end{subequations}
\end{theorem}

The computational procedure is summarized in the Algorithm \ref{alg:discrete}. A similar approach is discussed in \cite{Alexe:2011} in the context
of error estimation for inverse problems with elliptical PDEs.
\begin{algorithm}
\caption{Calculation of super-Lagrange multipliers} \label{alg:discrete}
\begin{algorithmic}[1]
\Procedure{DiscreteSuperLagrangeMultipliers}{}
\State Solve the Hessian equation \eqref{eqn:linearSys} for $\zeta$;
\State Solve the TLM \eqref{eqn:TLMDisc} forward in time for $\mu_k$, $k=0,\dots,N$;
\State Solve the SOA model \eqref{eqn:soaDisc} backward in time for $\nu_k$,  $k=N,\dots,0$.
\EndProcedure
\end{algorithmic}
\end{algorithm}

\begin{comment}[Iterative solution of the Hessian equation]
The Hessian equation \eqref{eqn:linearSys} can be solved by iterative methods such as preconditioned conjugate gradients \cite{Cioaca:2012},
which rely on the evaluation of matrix-vector products $v =  (\nabla^2_{\theta,\theta} j)(\theta^{\rm a}) \cdot u$ for any user defined vector $u$.
As explained in \cite{Cioaca:2012}  these products can be computed by first solving a tangent linear model \eqref{eqn:TLMDisc} initialized with $u$, 
and then solving a second order adjoint model \eqref{eqn:soaDisc}, where all linearizations are performed about the optimal solution 
$\{\theta^{\rm a},\x(\theta^{\rm a}),\lambda(\theta^{\rm a})\}$. The matrix-vector product $v$ is obtained from the second 
order adjoint variable at the initial time.
\end{comment}

\begin{comment}[Approximate solution of the Hessian equation]
The numerical solution of \eqref{eqn:ipdiscrete} is usually obtained in a reduced space approach via a gradient-based optimization method.  A reduced gradient $\nabla_\theta\, j(\theta^{(p)})$ is computed at each iteration $p$ of the numerical optimization algorithm. Quasi-Newton approximations of the reduced Hessian inverse $\mathbf{B} \approx (\nabla^2_{\theta,\theta} j)^{-1}$ can be constructed from the sequence of reduced gradients. As proposed in \cite{Alexe:2011}, a convenient way to approximately solve \eqref{eqn:linearSys} is to use the quasi-Newton  matrix: $\zeta \approx \mathbf{B} \cdot \mathcal{E}_{\theta}^T$.
\end{comment}

\begin{proof}
The variation of the super-Lagrangian \eqref{eqn:stationarity} with respect to independent perturbations in
$\theta,\x,\lambda$ is:
\begin{eqnarray*}
\nonumber
\delta \mathcal{L^E} &=& \mathcal{E}_\theta\, \delta\theta - \nu_0^{\rm T} \cdot\left(\delta\x_0 - (\x_0\left(\theta\right))_\theta\, \delta\theta\right) \\
&& - \sum_{k=0}^{N-1} \nu_{k+1}^{\rm T} \cdot \left(\delta\x _{k+1}-\M_{k,k+1}\delta\x_k -\mathfrak{M}_{k,k+1}\delta\theta \right)  \\
&&- \mu_N^{\rm T} \cdot \left(  \delta\lambda_{N} - \left(r_N \right)_{\x_N,\x_N} \delta\x_N - \left(r_N \right)_{\x_N,\theta} \delta\theta\right) \\
&&  -\sum_{k=0}^{N-1} \mu_k^{\rm T} \cdot \left( \delta{\lambda}_k - \M_{k,k+1}^{\rm T}\, \delta\lambda_{k+1} \right) \\
&&  + \sum_{k=0}^{N-1} \mu_k^{\rm T} \cdot \left(  \M_{k,k+1}^{\rm T}\, \lambda_{k+1} + \left(r_k\right)_{\x_k}^T  \right)_{\x_k}\, \delta\x_k 
   \\
&&  +\sum_{k=0}^{N-1} \mu_k^{\rm T} \cdot \left( 
           \M_{k,k+1}^{\rm T}\, \lambda_{k+1} + \left(r_k\right)_{\x_k}^T  \right)_{\theta}\, \delta\theta  \\
 && - \zeta^{\rm T} \cdot \left(\left(\x_{0}\right)^{\rm T}_{\theta}\delta \lambda_0 + \lambda_0^{\rm T} \left(\x_0\right)_{\theta, \theta} \delta \theta
 +\sum_{k=0}^{N} \left(r_k\right)_{\theta, \theta} \delta \theta  + \sum_{k=0}^{N} \left(r_k\right)_{\theta, \x_k} \delta \x_k\right)\, \\
 \nonumber && -  \zeta^{\rm T} \cdot \sum_{k=0}^{N-1}\, \left(\mathfrak{M}^{\rm T}_{k,k+1} \delta \lambda_{k+1} + \left(\mathfrak{M}^{\rm T}_{k,k+1}  \lambda_{k+1}\right)_{\x_k} \delta \x_k + \left(\mathfrak{M}^{\rm T}_{k,k+1}  \lambda_{k+1}\right)_{\theta} \delta \theta \right).
\end{eqnarray*}
The linearization point $\{\theta^{\rm a},\x(\theta^{\rm a}),\lambda(\theta^{\rm a})\}$  satisfies the ideal optimality conditions \eqref{eqn:kktDisc}.

The variation of the super-Lagrangian can be written in terms of dot-products as follows:
\begin{eqnarray*}
 \delta \mathcal{L^E} &=& \delta \mathcal{E} - \sum_{k=0}^N \left \langle \nabla_{\lambda_{k}}\mathcal{L^E},\delta \lambda_{k} \right \rangle 
 -  \sum_{k=0}^N  \left \langle  \nabla_{\x_{k}}\mathcal{L^E},\delta \x_{k} \right \rangle - \left \langle \nabla_{\theta}\mathcal{L^E},\delta \theta \right \rangle,
\end{eqnarray*}
and stationary points are characterized by $\nabla_{\lambda_{k}}\mathcal{L^E }= 0$,  $\nabla_{\x_{k}}\mathcal{L^E} = 0$, and $\nabla_{\theta}\mathcal{L^E} = 0$.

Setting $\nabla_{\lambda_{k}}\mathcal{L^E }= 0$ for $k=0,\dots, N$ leads to the tangent linear model (TLM):
\begin{eqnarray}
\label{eqn:TLMDiscApp}
\mu_0 &=&- \left(\x_{0}\right)_{\theta}\zeta; \\
\nonumber
 \mu_k &=& \M_{k-1,k} \, \mu_{k-1} - \mathfrak{M}_{k-1,k}\, \zeta, \quad k = 1,\dots,N.
\end{eqnarray}
The derivative of the model equation \eqref{eqn:model} with respect to $\theta$ is: 
\begin{eqnarray}
 \label{eqn:dModeldTheta}
\left(\x_{0}\right)_\theta  &=& \left(\x_0\left(\theta\right)\right)_\theta; \\
\nonumber(\x_{k+1})_\theta &=& \M_{k,k+1}\, \left(\x_k\right)_\theta + \mathfrak{M}_{k,k+1}, \quad k=0,\dots, N-1\,.
\end{eqnarray}
Multiplying \eqref{eqn:dModeldTheta} from the right with the vector $\zeta$ gives the variation of the model \eqref{eqn:model} with respect to $\theta$ in the direction $\zeta$:
\begin{eqnarray}
\label{eqn:dModeldTheta2}
 \left(\x_{0}\right)_\theta \zeta &=& \left(\x_0\left(\theta\right)\right)_\theta \zeta; \\
\nonumber
 \left(\x_{k}\right)_\theta\, \zeta &=& \M_{k-1,k} \,\left(\x_{k-1}\right)_{\theta} \, \zeta
 +  \mathfrak{M}_{k-1,k} \, \zeta,\quad k =1,\dots, N; 
\end{eqnarray}
Equations \eqref{eqn:TLMDiscApp} and \eqref{eqn:dModeldTheta2} are identical and consequently we make the identification
\begin{equation}\label{eqn:genXtheta}
\mu_k \equiv - \left(\x_k\right)_{\theta} \zeta, \quad k =0,\dots, N\,.
\end{equation}
Setting $ \nabla_{\x_{k}}\mathcal{L^E} = 0$ for $k=N,\dots,0$ leads to the following second order adjoint (SOA) model:
\begin{eqnarray}
\label{eqn:SOADiscApp}
~~~\nu_N &=& \left(r_N\right)_{\x_N, \x_N} \, \mu_N - \left(r_N\right)_{\theta, \x_N} \, \zeta\,; \\
 \nonumber  
  \nu_k &=& \M_{k,k+1}^{\rm T} \nu_{k+1} + \left(\M_{k,k+1}^{\rm T} \lambda_{k+1}\right)^{\rm T}_{\x_k}\mu_{k} 
  + \left(r_k\right)_{\x_k,\x_k} \, \mu_k 
  - \left(r_k \right)_{\theta, \x_k} \, \zeta\\
 \nonumber &&  - \left(\mathfrak{M}_{k,k+1}^{\rm T} \lambda_{k+1}\right)_{\x_k}^{\rm T} \zeta\,,  \qquad k = N-1,\dots,0.
\end{eqnarray}
%
\iffalse
We also have
%
\begin{eqnarray*}
\left \langle \nabla_{\theta}\mathcal{L^E},\delta \theta \right \rangle &=&  \mathcal{E}_\theta\, \delta\theta 
+ \sum_{k=0}^{N-1} \nu_{k+1}^{\rm T} \cdot \left( \mathfrak{M}_{k,k+1}\, \delta \theta\right) + \nu_0^{\rm T} \cdot\left( \x_{0}\right) _{\theta} \delta \theta   \\
  && +\sum_{k=0}^{N} \mu_k^{\rm T} \cdot\left(   \left(r_k\right)_{\x_k,\theta} \delta \theta  \right)
  +\sum_{k=0}^{N-1} \mu_k^{\rm T} \cdot\left( \left(\M_{k,k+1}^{\rm T}\, \lambda_{k+1}\right)_{\theta} \delta \theta \right)  \\
 && - \zeta^{\rm T} \cdot \left( \left(\left(\x_{0}\right)_{\theta, \theta}\delta \theta\right)^{\rm T} \lambda_0
 +\sum_{k=0}^{N} \left(r_k\right)_{\theta, \theta} \delta \theta + \sum_{k=0}^{N-1}  \left(\mathfrak{M}_{k,k+1}^{\rm T} \lambda_{k+1} \right)_{\theta} \delta \theta\right).
\end{eqnarray*}
\fi
%
Setting $\nabla_{\theta}\mathcal{L^E} = 0$ gives:
\begin{eqnarray}
\label{eqn:refInterApp1}
0 &=&  \mathcal{E}^T_\theta
  + \left( \x_{0}\right)^T_{\theta} \, \nu_0 + \sum_{k=0}^{N-1}  \mathfrak{M}_{k,k+1}^T\, \nu_{k+1}    \\
\nonumber
  && + \left(r_N\right)_{\x_N,\theta}^T \, \mu_N
  +\sum_{k=0}^{N-1} \left(\M_{k,k+1}^{\rm T}\, \lambda_{k+1} + \left(r_k\right)_{\x_k} \right)_{\theta}^T\, \mu_k   \\
\nonumber
 && - \left(\lambda_0^{\rm T} \left(\x_{0}\right)_{\theta,\theta}\right)^{\rm T}\, \zeta 
  -\sum_{k=0}^{N}  \left(r_k\right)_{\theta,\theta}^{\rm T}\, \zeta
 -\sum_{k=0}^{N-1} \left(  \mathfrak{M}_{k,k+1}^{\rm T} \lambda_{k+1} \right)_{\theta}^T  \, \zeta.
\end{eqnarray}
The transposed equation \eqref{eqn:dModeldTheta} times the multiplier $\nu$ gives:
\begin{eqnarray*}
\left(\x_{0}\right)_\theta^T\,\nu_{0}  &=& \left(\x_0\left(\theta\right)\right)_\theta^T\,\nu_{0}; \\
 (\x_{k+1})_\theta^T\, \nu_{k+1} &=& \left(\x_k\right)_\theta^T\,  \M_{k,k+1}^T\,\nu_{k+1}  + \mathfrak{M}_{k,k+1}^T\,\nu_{k+1} , \quad k=0,\dots, N-1\,,
\end{eqnarray*}
and using the SOA model \eqref{eqn:SOADiscApp}
\begin{eqnarray*}
 (\x_{k+1})_\theta^T\, \nu_{k+1} &=&   (\x_{k})_\theta^T\, \nu_{k}  
 - (\x_{k})_\theta^T\,\left(\M_{k,k+1}^{\rm T} \lambda_{k+1}\right)^{\rm T}_{\x_k}\mu_{k} 
 + (\x_{k})_\theta^T\,\left(r_k \right)_{\theta, \x_k} \, \zeta \\
 &&  - (\x_{k})_\theta^T\, \left(r_k\right)_{\x_k,\x_k} \, \mu_k + (\x_{k})_\theta^T\,\left(\mathfrak{M}_{k,k+1}^{\rm T} \lambda_{k+1}\right)_{\x_k}^{\rm T} \zeta
 +  \mathfrak{M}_{k,k+1}^T\,\nu_{k+1}, \\
&&  \quad k=0,\dots, N-1\,.
\end{eqnarray*}
Summing up this equation for times $k=0,\dots,N-1$ leads to
\begin{eqnarray*}
 (\x_{N})_\theta^T\, \nu_{N} &=&   (\x_{0})_\theta^T\, \nu_{0}  
 - \sum_{k=0}^{N-1} (\x_{k})_\theta^T\,\left(\M_{k,k+1}^{\rm T} \lambda_{k+1}\right)^{\rm T}_{\x_k}\mu_{k} 
 + \sum_{k=0}^{N-1}(\x_{k})_\theta^T\,\left(r_k \right)_{\theta, \x_k} \, \zeta \\
 &&  -  \sum_{k=0}^{N-1}(\x_{k})_\theta^T\, \left(r_k\right)_{\x_k,\x_k} \, \mu_k + \sum_{k=0}^{N-1}(\x_{k})_\theta^T\,\left(\mathfrak{M}_{k,k+1}^{\rm T} \lambda_{k+1}\right)_{\x_k}^{\rm T} \zeta
 +  \sum_{k=0}^{N-1}\mathfrak{M}_{k,k+1}^T\,\nu_{k+1}\,,
\end{eqnarray*}
and after inserting the final condition for $\nu_N$:
\begin{eqnarray}
 (\x_{0})_\theta^T\, \nu_{0}  +  \sum_{k=0}^{N-1}\mathfrak{M}_{k,k+1}^T\,\nu_{k+1}
 &=&  
 \sum_{k=0}^{N-1} (\x_{k})_\theta^T\,\left(\M_{k,k+1}^{\rm T} \lambda_{k+1}\right)^{\rm T}_{\x_k}\mu_{k} 
 \nonumber \\ \label{eqn:nuCthetainEtheta1}
&&+\sum_{k=0}^{N}(\x_{k})_\theta^T\, \left(r_k\right)_{\x_k,\x_k} \, \mu_k - \sum_{k=0}^{N}(\x_{k})_\theta^T\,\left(r_k \right)_{\theta, \x_k} \, \zeta   \\
\nonumber && - \sum_{k=0}^{N-1}(\x_{k})_\theta^T\,\left(\mathfrak{M}_{k,k+1}^{\rm T} \lambda_{k+1}\right)_{\x_k}^{\rm T} \zeta.
\end{eqnarray}
Substituting equations \eqref{eqn:nuCthetainEtheta1} and \eqref{eqn:genXtheta} in \eqref{eqn:refInterApp1} we obtain the following 
expression of the equation $\nabla_{\theta}\mathcal{L^E} = 0$:
\begin{eqnarray}
\label{eqn:mod-hessian-eqn}
0 &=&  \mathcal{E}^T_\theta
   -\sum_{k=0}^{N}\left(\x_k\right)^{\rm T}_{\theta} \left(\left(r_k\right)^{\rm T}_{\x_k,\x_k}\right)\left(\x_{k}\right)_{\theta} \zeta\\
 \nonumber&& -\sum_{k=0}^{N-1}\left(\left(\x_{k}\right)^{\rm T}_{\theta}\left(\M_{k,k+1}^{\rm T} \lambda_{k+1}\right)_{\x_k}\left(\x_{k}\right)_{\theta} \zeta\right) \\
 \nonumber && - \sum_{k=0}^{N-1}\left(\x_{k}\right)^{\rm T}_{\theta}\left(\left( r_k\right)_{\theta, \x_k}^{\rm T} +   \left(\mathfrak{M}_{k,k+1}^{\rm T} \lambda_{k+1}\right)_{\x_k}\right)\zeta - \left(\x_N\right)_{\theta}^{\rm T}\left(r_N\right)_{\theta, \x_k}^{\rm T} \, \zeta\\
\nonumber
  && - \left(r_N\right)_{\x_N,\theta}^T \, \left(\x_N\right)_\theta \zeta
  - \sum_{k=0}^{N-1} \left(\M_{k,k+1}^{\rm T}\, \lambda_{k+1} + \left(r_k\right)_{\x_k} \right)_{\theta}^T\, \left(\x_k\right)_{\theta} \zeta   \\
\nonumber
 && - {\left(\lambda_0^{\rm T} \left(\x_{0}\right)_{\theta,\theta}\right)^{\rm T}}\, \zeta 
  -\sum_{k=0}^{N}  \left(r_k\right)_{\theta,\theta}^{\rm T}\, \zeta
 -\left( \sum_{k=0}^{N-1}  \left(\mathfrak{M}_{k,k+1}^{\rm T} \lambda_{k+1} \right)_{\theta}^T \right) \cdot \zeta.
\end{eqnarray}

\paragraph{Hessian of the reduced function}
%
\iffalse
The gradient the reduced cost function \eqref{eqn:cfd-reduced} with respect to $\theta$ is
%
\begin{eqnarray}
\label{eqn:grad-reduced}
\nabla_\theta j &=& \displaystyle \sum_{k=0}^N\, (\x_k)_\theta^T\,(r_k)^T_{\x_k} +  \sum_{k=0}^N\, (r_k)^T_{\theta}.
\end{eqnarray}
%
Taking the variation of \eqref{eqn:grad-reduced}  with respect to $\theta$ in the direction $\delta\theta=\zeta$ yields
%
\begin{eqnarray}
\label{eqn:hess-reduced-a}
\left(\nabla^2_{\theta,\theta} j\right)\, \zeta &=&  
\sum_{k=0}^N\, \left( (\x_k)_{\theta,\theta} \zeta \right)^T\,(r_k)^T_{\x_k}  \\
\nonumber
&& + \sum_{k=0}^N\, (\x_k)_\theta^T\, \left( (r_k)_{\x_k,\theta}\zeta + (r_k)_{\x_k,\x_k}\, (\x_k)_\theta\zeta \right)  \\
\nonumber
&& +  \sum_{k=0}^N\, \left( (r_k)_{\theta,\x_k}\, (\x_k)_\theta\, \zeta + (r_k)_{\theta,\theta}\,\zeta \right) 
\end{eqnarray}
%
Using the tangent linear model \eqref{eqn:TLMDiscApp} equation \eqref{eqn:hess-reduced-a} becomes
%
\begin{eqnarray}
\label{eqn:hess-reduced-b}
\nabla^2_{\theta,\theta} j\cdot \zeta &=&  
\sum_{k=0}^N\, \left( (\x_k)_{\theta,\theta} \zeta \right)^T\,(r_k)^T_{\x_k}  \\
\nonumber
&& + \sum_{k=0}^N\, (\x_k)_\theta^T\, \left( (r_k)_{\x_k,\theta}\zeta + (r_k)_{\x_k,\x_k}\, \mu_k \right)  \\
\nonumber
&& +  \sum_{k=0}^N\, (r_k)_{\theta,\x_k}\, \mu_k + \sum_{k=0}^N\, (r_k)_{\theta,\theta}\,\zeta
\end{eqnarray}
\fi
%
Consider the reduced Lagrangian \eqref{eqn:LagDisc}
\begin{equation}
 \label{eqn:LagDiscApp-reduced}
 \ell(\theta) = j(\theta)
 -  \sum_{k=0}^{N-1}\, \lambda^{\rm T}_{k+1}  \cdot \left(\x _{k+1}(\theta) - \Model_{k,k+1}(\x_k(\theta), \theta)\right) - \lambda_0^{\rm T} \cdot \left(\x_0 - \x_0 \left(\theta\right)\right).
\end{equation}
Since there are only equality constraints  the reduced Lagrangian \eqref{eqn:LagDiscApp-reduced}, its gradient, and its Hessian evaluated at a solution are identically equal to the reduced cost function \eqref{eqn:cfd-reduced}, its reduced gradient, and its reduced Hessian, respectively:
\begin{equation}
\label{eqn:same-l-j}
\ell(\theta) \equiv j(\theta), \quad
\nabla_\theta \ell(\theta) \equiv \nabla_\theta j(\theta), \quad
\nabla^2_{\theta,\theta} \ell(\theta) \equiv \nabla^2_{\theta,\theta} j(\theta).
\end{equation}
The gradient of the reduced Lagrangian \eqref{eqn:LagDiscApp-reduced} with respect to $\theta$ is
\begin{eqnarray}
\label{eqn:LagGrad-reduced}
(\nabla_\theta \ell)^T &=& \sum_{k=0}^{N} \, \left(r_k\right)_{\theta} +  \sum_{k=0}^{N} \, \left(r_k\right)_{\x_k} \left( \x_k \right)_{\theta} \\
\nonumber
&&- \sum_{k=0}^{N-1} \lambda_{k+1}^{\rm T} \bigl(\left(\x_{k+1}\right)_{\theta} - \mathfrak{M}_{k,k+1} -\M _{k,k+1}\, \left(\x_k\right)_{\theta}\bigr) \\
\nonumber
&&  -  \sum_{k=0}^{N-1}\, (\lambda_{k+1}^{\rm T})_\theta  \cdot \left(\x _{k+1}(\theta) - \Model_{k,k+1}(\x_k(\theta), \theta)\right) \\
\nonumber
&& - \lambda_0^{\rm T} \cdot \left((\x_0)_\theta - (\x_0 \left(\theta\right))_\theta\right)
- (\lambda_0^{\rm T})_\theta \cdot \left(\x_0 - \x_0 \left(\theta\right)\right).
\end{eqnarray}
Taking the variation of \eqref{eqn:LagGrad-reduced}  with respect to $\theta$ in the direction $\delta\theta=\zeta$ 
and evaluating all terms at at the optimal point $\{\theta^{\rm a},\x(\theta^{\rm a}),\lambda(\theta^{\rm a})\}$ gives:
\begin{eqnarray}
\nonumber
\left(\nabla^2_{\theta,\theta} \ell\right)\, \zeta &=& \sum_{k=0}^{N} \, \left(\left(r_k\right)_{\theta,\theta}  +   \left(\x_k\right)_{\theta}^{\rm T}\left(r_k\right)_{\x_k, \x_k} \left( \x_k \right)_{\theta}
 + \left( \x_k \right)_{\theta}^{\rm T} \, \left(r_k\right)_{\x_k,\theta}\,\right) \zeta \\
\label{eqn:LagHess-reduced-new2}
&&+ \sum_{k=0}^{N} \left(r_k\right)_{\theta, \x_k} \left(\x_k\right)_{\theta}\, \zeta 
    + \sum_{k=0}^{N-1} \bigl( \lambda_{k+1}^{\rm T} \,  \mathfrak{M}_{k,k+1} \bigr)_\theta\, \zeta \\
\nonumber
&&+ \sum_{k=0}^{N-1} \left(\left(\x_k\right)_{\theta}^{\rm T}\bigl(\lambda_{k+1}^{\rm T} \, \M _{k,k+1}\, \bigr)_\theta +   \bigl( \lambda_{k+1}^{\rm T} \,  \mathfrak{M}_{k,k+1} \bigr)_{\x_k} \left(\x_k\right)_{\theta}\right)\, \zeta\\
\nonumber
&& + \lambda_0^{\rm T} \cdot \left( \x_0 \left(\theta\right)_{\theta,\theta}\right)\,\zeta  + \sum_{k=0}^{N-1}\left(\x_k\right)_{\theta}^{\rm T}\bigl(\lambda_{k+1}^{\rm T} \, \M _{k,k+1}\, \bigr)_{\x_k} \left(\x_k\right)_{\theta} \, \zeta. 
\end{eqnarray}

Substituting equations \eqref{eqn:LagHess-reduced-new2} and \eqref{eqn:same-l-j} into \eqref{eqn:mod-hessian-eqn}
leads to the following simpler form of the equation $\nabla_{\theta}\mathcal{L^E} = 0$:
\begin{equation}
\mathcal{E}^{\rm T}_{\theta} = \left(\nabla^2_{\theta,\theta} \ell\right)^T \cdot \zeta = \left(\nabla^2_{\theta,\theta} \ell\right) \cdot \zeta = \left(\nabla^2_{\theta,\theta} j\right) \cdot \zeta\,.
\end{equation}

\end{proof}

\begin{comment}[Relation to the error covariance matrix of the optimal solution]
The paper \cite{Gejadze:2008} describes an algorithm for the evaluation of the error covariance matrix associated with the optimal solution $\theta^{\rm a}$ when there are errors in the data. There is a direct relationship between the above a posteriori error estimate and \cite{Gejadze:2008}. In this work we can recover the error covariance matrix column by column by successively solving the system in \eqref{eqn:linearSys} for several error functionals. Specifically, if we take $\mathcal{E}$ to be one solution component \eqref{eqn:error_functional}, $\mathcal{E}_{\theta}$ becomes the canonical basis vector $\mathbbm{e}_{\rm k}$. Application of Algorithm \ref{alg:discrete} then recovers the $k^{th}$ column of the a posteriori error covariance matrix by solving the linear system \eqref{eqn:linearSys}. 
%In the numerical experiments we demonstrate this methodology to compute the contributions of individual grid points on the optimal solution. 
\end{comment}

%%%%%%%%%%%%%%%%%%%%%%%%%%%%%%%%%%%%%%%%%%%%%%%%%%%%%%%%%%%%%%%%%
\section{Application to data assimilation problems}\label{sec:da}
%%%%%%%%%%%%%%%%%%%%%%%%%%%%%%%%%%%%%%%%%%%%%%%%%%%%%%%%%%%%%%%%%
Next we apply this methodology to a specific discrete-time inverse problem, namely, four dimensional variational (4D-Var) data assimilation.
Data assimilation is the fusion of information from imperfect model predictions and noisy data available at discrete times, to obtain a consistent description of the state of a physical system \cite{daley1993, kalnay2003}. For a detailed description of the sources of information, sources of error, description of four dimensional variational assimilation problems (4D-Var), approaches to solve the 4D-Var problems and a detailed derivation of a posteriori error estimation for 4D-Var problems, please see \cite{Rao:2014}. 

%%%%%%%%%%%%%%%%%%%%%%%%%%%%%%%%%%%%%
\subsection{The ideal 4D-Var problem}
%%%%%%%%%%%%%%%%%%%%%%%%%%%%%%%%%%%%%

\iffalse
The typical cost function in 4D-Var data assimilation reads \cite{kalnay2003}
%
 \begin{eqnarray}
 \label{eqn:4D-Var}
 {\J}(\x_0) &=&
 \frac{1}{2}\, \left(\x_0-\x^{\rm b}_0\right)^{\rm T} \, \B_0^{-1}\, (\x_0-\x^{\rm b}_0)
 \\
\nonumber
 && +
 \frac{1}{2}\sum_{k=0}^N \left(\mathcal{H}_k (\x_k) - \y_k\right)^{\rm T}
\R_k^{-1}\left(\mathcal{H}_k (\x_k) - \y_k\right)\,,
 \end{eqnarray}
 %
where, $\x_0^{\rm b}$ is the background state at the initial time (the prior knowledge of the initial conditions), $\mathbf{B}_0$ is the covariance matrix of the background errors, $\y_k$ is the vector of observations at time $t_k$ and $\mathbf{R}_k$ is the corresponding observation error covariance matrix. The observation operator $\mathcal{H}$ maps the model state space onto the observation space. 
\fi

We consider the particular case of strongly constrained 4D-Var data assimilation \cite{kalnay2003} where 
the parameters are the initial conditions $\theta := \x_0$ and the cost function \eqref{eqn:4D-Var} is
 \begin{eqnarray}
 \label{eqn:4D-Var-fun}
 {\J}(\x_0) &=&
 \frac{1}{2}\, \left(\x_0-\x^{\rm b}_0\right)^{\rm T} \, \B_0^{-1}\, \left(\x_0-\x^{\rm b}_0\right)
 \\
\nonumber
 && +
\frac{1}{2}\,  \sum_{k=0}^N\; \left(\mathcal{H}_k (\x_k) - \y_k\right)^{\rm T}
\R_k^{-1}\left(\mathcal{H}_k (\x_k) - \y_k\right)\,,
 \end{eqnarray}

The inference problem is formulated as follows:
\begin{equation}\label{eqn:4dvar-opt}
 \x^{\rm a}_0 = \underset{\x_0 \in \mathbbm{R}^n} {\text{arg\,min}}~ \J\left(\x_0\right) \quad \text{subject to \eqref{eqn:model}\,.}
\end{equation}
The first order optimality conditions for the problem \eqref{eqn:4dvar-opt} read:
\begin{subequations}\label{eqn:disckkt}
\begin{eqnarray}
\label{kkt:fwdFDVAR}
\textnormal{forward model:}~~ & 0 = \x _{k+1}-\Model_{k,k+1}(\x_k) , \quad k=0,1,\dots , N-1\,;
\end{eqnarray}
\begin{eqnarray}
\label{kkt:adjFDVAR}
\textnormal{adjoint model:} \quad  \lambda_{N} &=& \mathbf{H}_N^{\rm T}\R_N^{-1}\left(\mathcal{H}_N(\x _N)-\y_N\right) \,,  \\
\nonumber
{\lambda}_k &=&  \M_{k,k+1}^{\rm T}\, \lambda_{k+1} +
\mathbf{H}_k^{\rm T}\R_k^{-1}\left(\mathcal{H}_k(\x _k)-\y_k\right)\,, \\
\nonumber && \quad k=N-1, \dots , 0\,;
\end{eqnarray}
\begin{eqnarray}
\label{kkt:optFDVAR}
\textnormal{optimality:} ~~&0 =  \B_0^{-1}(\x _0-\x^{\rm b}_0)+ \lambda_0 \,.
\end{eqnarray}
\end{subequations}
Here ${\lambda}_k \in \mathbbm{R}^n$ are the adjoint variables, and
\begin{eqnarray*}
\mathbf{H}_k &:=& \left(\mathcal{H}_k\right)_{\x _k}(\x _k)\,,
\end{eqnarray*}
is the state-dependent Jacobian matrix of the observation operator.

%
%%%%%%%%%%%%%%%%%%%%%%%%%%%%%%%%%%%%%%%%%%
\subsection{The perturbed 4D-Var problem}
%%%%%%%%%%%%%%%%%%%%%%%%%%%%%%%%%%%%%%%%%%
%
In this section we use the imperfect data, imperfect model and hence solve a perturbed 4D-Var problem. The evolution of the discrete
state vector $\x \in \mathbbm{R}^n$ is represented by the imperfect discrete model \eqref{eqn:model-p}.
In the presence of data errors $\Delta\y_k$ the discrete cost function reads \cite{kalnay2003}:
 \begin{eqnarray}
 \label{eqn:4D-Var:imperfect}
 \widehat{\J}(\x_0) &=&
 \frac{1}{2}\, \left(\x_0-\x^{\rm b}_0\right)^{\rm T} \, \B_0^{-1}\, (\x_0-\x^{\rm b}_0)
 \\
\nonumber
 && +
 \frac{1}{2}\sum_{k=0}^N \left(\mathcal{H}_k (\widehat{\x}_k) - \y_k- {\Delta\y_k} \right)^{\rm T}
\R_k^{-1}\left(\mathcal{H}_k (\widehat{\x}_k) - \y_k - {\Delta\y_k} \right).
 \end{eqnarray}
The perturbation in each of the cost function terms is
\[
\Delta \widehat{r}_k =
\left(\y_k - \mathcal{H}_k (\widehat{\x}_k) \right)^{\rm T} \, \R_k^{-1}\, \Delta\y_k
+ 
 \frac{1}{2} {\Delta\y_k}^{\rm T} \, \R_k^{-1}\, {\Delta\y_k}.
\]
The perturbed strongly constrained 4D-Var analysis problem solved in reality is
\begin{equation}
\label{eqn:4dvar-opt-p}
 \begin{aligned}
 \widehat{\x}^{\rm a}_0 = &~\underset{\x_0 \in \mathbbm{R}^n} {\text{arg\, min}}\,
 &&\widehat{\J}\left(\x_0\right) \quad \text{subject to} ~~\text{\eqref{eqn:model-p}\,.} 
 \end{aligned}
\end{equation}

%%%%%%%%%%%%%%%%%%%%%%%%%%%%%%%%%%%%%%%%%%%%%%%%%%%%%%%%%%%%%%%%%%%%%%%%%%%%%%%%
\subsection{Super-Lagrangian for the 4D-Var problem} \label{sec:supLagDiscFdvar}
%%%%%%%%%%%%%%%%%%%%%%%%%%%%%%%%%%%%%%%%%%%%%%%%%%%%%%%%%%%%%%%%%%%%%%%%%%%%%%%%
We follow the same procedure as in Section \ref{sec:supLagDisc} to construct the super-Lagrangian 
\eqref{eqn:supLagDisc}  associated with the \qoi functional of the form  \eqref{eqn:error_functional}  and with the first order discrete optimality conditions \eqref{eqn:disckkt} as  constraints.
The super-Lagrange multipliers for a stationary point of $\mathcal{L^E}$ are computed using Algorithm \ref{alg:discrete}. Equations \eqref{eqn:supLagParamsDisc} take the following particular form for the 4D-Var system: 
\begin{subequations}
\label{eqn:supLagParamsDiscFdvar}
\begin{align}
\label{eqn:linSysDiscFdvar}
\textnormal{Linear system:}~~ &  
\left(\nabla_{\x _0, \x _0}^2\, j\right)\, \cdot \zeta =  \nabla_{\x_0} \mathcal{E}  \,; \\
\label{eqn:tlmDiscFdvar}
\textnormal{TLM:}~~&\mu_0 = -\zeta\,; \quad \mu_{k+1} = \M_{k,k+1}\,\,\mu_{k},\quad k=0,\dots,N-1\,; \\
\label{eqn:soaDiscFdvar}
\textnormal{SOA:}
~~ &  \nu_{N} = \mathbf{H}_N^{\rm T}\R_N^{-1}\mathbf{H}_N \, \mu_N\,,\\
\nonumber
& \nu_{k} = \M_{k,k+1}^T\,\nu_{k+1}  + (\M_{k,k+1}^{\rm T}\, \lambda_{k+1})_{\x_k}^T\, \mu_k \\
\nonumber&\qquad + \mathbf{H}_k^{\rm T}\R_k^{-1}\mathbf{H}_k\, \mu_k\,, \quad  k=N-1,\dots,0.
\end{align}
\end{subequations}
%%%%%%%%%%%%%%%%%%%%%%%%%%%%%
\subsection{The 4D-Var a posteriori error estimate}
%%%%%%%%%%%%%%%%%%%%%%%%%%%%%

We apply the a posteriori error estimate \eqref{eqn:E-error-estimate} to the 4D-Var solution. The total error \eqref{eqn:deFinal} is the sum of the contributions of forward model errors
\begin{subequations}
\label{eqn:discimpactFdvar}
\begin{eqnarray}
\label{eqn:discimpactFdvar_fwd}
  \Delta\mathcal{E}_{\rm fwd} & = & \displaystyle\sum_{k=1}^{N}\,\nu_{k}^{\rm T} \cdot \Delta \widehat{\x}_{k} \,,
\end{eqnarray}
the contributions of the adjoint model errors
 \begin{eqnarray}
 \label{eqn:discimpactFdvar_adj}
  \Delta\mathcal{E}_{\rm adj} & = & -\sum_{k=0}^{N}\mu_k^{\rm T} \cdot  \left( {\mathbf{H}_k^{\rm T}\R_k^{-1} \Delta\y_k}  \right) +  \sum_{k=0}^{N-1}\mu_k^{\rm T} \cdot \left(\Delta \widehat{\x}_{k+1}\right)_{{\x}_k}^{\rm T} \widehat{\lambda}_{k+1}  \,,
 \end{eqnarray}
and the contribution of the error in the optimality equation
\begin{eqnarray}
\label{eqn:discimpactFdvar_opt}
 \Delta \mathcal{E}_{\rm opt} & = & - \zeta^{\rm T} \, \left(\Delta \widehat{\x}_{1}\right)^{\rm T}_{\x_0}\,\widehat{\lambda}_{1}.
\end{eqnarray}
\end{subequations}

%%%%%%%%%%%%%%%%%%%%%%%%%%%%%
\subsection{Probabilistic interpretation}
%%%%%%%%%%%%%%%%%%%%%%%%%%%%%

Consider the case where the model errors are given by a state-dependent bias plus state-independent noise: 
\[
\Delta \widehat{\x}_{k} = \beta_k + \eta_k; \quad \texttt{E}[\eta_k] = 0; \quad \texttt{cov}[\eta_k,\eta_\ell] = \mathbf{Q}_{k,\ell}.
\]
Similarly, assume that the data errors are composed of bias and noise (both state-independent) and that data noise at different times is uncorrelated: 
\begin{equation}
\label{eq:data-error-statistics}
\Delta \widehat{\y}_{k} = \rho_k + \varepsilon_k; \quad \texttt{E}[\varepsilon_k] = 0; \quad \texttt{cov}[\varepsilon_k,\varepsilon_k] = \mathbf{R}_k; 
 \quad \texttt{cov}[\varepsilon_k,\varepsilon_\ell] = \mathbf{0}, ~ k \ne \ell.
\end{equation}
Assume in addition that the model and the data noises are uncorrelated.

Consider the super-multipliers evaluated at a given forward and adjoint trajectory, e.g., at the optimum. The super-multiplier values do not depend on the noise in the 
model and in the data. From equations \eqref{eqn:discimpactFdvar} the error estimate reads
\begin{subequations}
\label{eqn:discimpactFdvar-noise}
\begin{eqnarray}
\label{eqn:discimpactFdvar-estimate-noise}
~~  \Delta\mathcal{E}^{\rm est} & = & \displaystyle\sum_{k=1}^{N}\,\nu_{k}^{\rm T} \cdot \beta_{k} + \displaystyle\sum_{k=1}^{N}\,\nu_{k}^{\rm T} \cdot \eta_{k}  
  -\sum_{k=0}^{N}\mu_k^{\rm T} \cdot  \left( {\mathbf{H}_k^{\rm T}\R_k^{-1} \rho_k}  \right) \,,\\
 \nonumber
  &  & -\sum_{k=0}^{N}\mu_k^{\rm T} \cdot  \left( {\mathbf{H}_k^{\rm T}\R_k^{-1} \varepsilon_k}  \right) +  \sum_{k=0}^{N-1}\mu_k^{\rm T} \cdot \left(\beta_{k+1}\right)_{{\x}_k}^{\rm T} \widehat{\lambda}_{k+1}   - \zeta^{\rm T} \, \left(\beta_{1}\right)^{\rm T}_{\x_0}\,\widehat{\lambda}_{1}.
\end{eqnarray}
The mean of the estimated \qoi error is
\begin{eqnarray}
\label{eqn:discimpactFdvar-noise-mean}
 \texttt{E}[\Delta\mathcal{E}^{\rm est}] & = & \displaystyle\sum_{k=1}^{N}\,\nu_{k}^{\rm T} \cdot \beta_{k} 
  -\sum_{k=0}^{N}\mu_k^{\rm T} \cdot  \left( {\mathbf{H}_k^{\rm T}\R_k^{-1} \rho_k}  \right) \\
\nonumber
&& + \sum_{k=0}^{N-1}\mu_k^{\rm T} \cdot \left(\beta_{k+1}\right)_{{\x}_k}^{\rm T} \widehat{\lambda}_{k+1} 
 - \zeta^{\rm T} \, \left(\beta_{1}\right)^{\rm T}_{\x_0}\,\widehat{\lambda}_{1},
\end{eqnarray}
and the last two terms disappear when the model bias is state-independent.
The variance of the estimated \qoi error contributions is:
\begin{eqnarray}
\label{eqn:discimpactFdvar-noise-var}
  \texttt{var}[\Delta\mathcal{E}^{\rm est}] & = & \ \displaystyle\sum_{k,\ell=1}^{N}\,\nu_{k}^{\rm T} \, \mathbf{Q}_{k,\ell}\, \nu_{\ell} 
   +  \sum_{k=0}^{N}\mu_k^{\rm T} \, \left( \mathbf{H}_k^{\rm T}\, \R_k^{-1} \, \mathbf{H}_k \right)\, \mu_k .
\end{eqnarray}
\end{subequations}
More details can be found in \cite[Appendix C]{Rao:2014_Supp}.

%%%%%%%%%%%%%%%%%%%%%%%%%%%%%%%%%%%%%%%%%%%%%%%%%%
\section{Numerical Experiments} \label{sec:numexp}
%%%%%%%%%%%%%%%%%%%%%%%%%%%%%%%%%%%%%%%%%%%%%%%%%%
We now apply the continuous and discrete a posteriori error estimation methodologies to two test problems, the heat equation and the shallow water model on a sphere. 
The a posteriori error estimates for the heat equation is performed using the continuous model procedure, whereas for the shallow water model we calculate the estimates using a discrete model. 
%%%%%%%%%%%%%%%%%%%%%%%%%%
\subsection{Heat equation}
%%%%%%%%%%%%%%%%%%%%%%%%%%
The one dimensional heat equation is given by \cite{Hundsdorfer:2003}:
\begin{equation}\label{eqn:he}
 \frac{\partial \mathbf{u}}{\partial t} = \alpha^2 \frac{\partial^2\mathbf{u}}{\partial \x^2} ,\quad \x \in [-1,1]\, ,
 \quad t \in [0,0.1]\,,
\end{equation}
with the following initial and boundary conditions: 
\begin{eqnarray}\label{eqn:periodicpde}
\begin{cases}
  \mathbf{u}\left(0,\x\right) = u_0\left(\x\right), \\ 
 \mathbf{u}\left(t,-1\right) = \mathbf{u}\left(t,1\right) , \\
\displaystyle \frac{\partial \mathbf{u}}{\partial \x} \left(t,-1\right) = \frac{\partial \mathbf{u}}{\partial \x} \left(t,1\right)\,. 
 \end{cases}
\end{eqnarray}
%
% The error functional can be written as:
% \begin{equation}\label{errfunc}
% \mathcal{E} = \frac{1}{\rm Ngrid}\,\sum_{i=1}^{\rm Ngrid} u^{\rm a}(t_0,x_i)\,.
% \end{equation}
% 

We discretize the PDE \eqref{eqn:periodicpde} in space using a central difference scheme to obtain an ODE of the form \eqref{eqn:ode}, which is our forward model. The evolution of temperature with time is shown in Figure \ref{fig:forward}.  Synthetic observations are obtained by integrating the forward model in time, using a reference initial condition, and perturbing the solution at various times with noise, whose mean is 0 and standard deviation is 10\% of the actual solution.
Synthetic model errors are introduced by adding a constant vector to the actual model; the imperfect model has the form \eqref{eqn:odeerr} with
$\Delta f(t) = 1$. 

We solve the inverse problem \eqref{eqn:ip} to obtain $\x_0^{\rm a}$ which minimizes the cost function $\eqref{eqn:cf}$. The solution of the inverse problem \eqref{eqn:ip} requires solving a constrained optimization problem. The optimization is performed using Poblano, a Matlab package for gradient based optimization \cite{SAND2010-1422}. The necessary gradients are computed using FATODE, a package for time integration and sensitivity analysis for ODEs \cite{Hong:2011}. 

The \qoi, i.e., the error functional, is the mean value of the optimal initial condition
\begin{equation}\label{eqn:errfunctional}
 \mathcal{E}\left(\x_0^{\rm a}\right) = \frac{1}{n} \displaystyle \sum \limits_{i=1}^{n} \left( \x_{0}^{\rm a}\right)_i \,.
\end{equation}
We denote the solution of the perturbed inverse problem \eqref{eqn:4dvar-opt-p} by $\widehat{\x}_0^{\rm a}$. The actual error in the mean of the solution \eqref{eqn:errfunc} is given by:
\begin{equation}
\label{eq:eactual}
\Delta \mathcal{E}_{\rm actual} =  \mathcal{E}\left(\widehat{\x}_0^{\rm a}\right) -\mathcal{E}\left(\x_0^{\rm a}\right)  = \frac{1}{n} \displaystyle \sum \limits_{i=1}^{n}\left(  \left( \widehat{\x}_{0}^{\rm a}\right)_i - \bigl( \x_{0}^{\rm a}\right)_i \bigr)\,.
\end{equation}
We follow the procedure outlined in Algorithm \ref{alg:sup} and Section \ref{sec:errest} to estimate the impact of the data and model 
errors on the mean of the optimal solution \eqref{eqn:errfunctional}. 
Solutions of the tangent linear, first order adjoint, and the second order adjoint models are shown in Figures \ref{fig:tlmsol}, \ref{fig:adjoint}, and \ref{fig:soasol} respectively. Table \ref{tab:error:he} compares the actual error(equation \eqref{eq:eactual}) in the \qoi and an estimate of \eqref{eq:eactual} $\left(\Delta \mathcal{E}_{\rm est}\right)$. We observe that the estimates are within acceptable bounds, when compared to the actual values.
Figure \ref{fig:dataerrors} shows the errors in the individual observations for the 1D heat equation; they are randomly distributed. Figure \ref{fig:datacontr} shows the contributions of different observation errors to the error in the quantity of interest \eqref{eqn:errfunctional}. We observe that certain grid points contribute to the error more than others. Since the physical process is diffusive, measurements errors occurring earlier in time contribute more to the a posteriori error estimate. The data error contributions indicate the sensitive areas, where measurements need to be very accurate. Gross inconsistencies in the data error contribution may also point towards faulty sensors. Figure \ref{fig:modelcontr} shows the contributions of model errors at different grid points to the error in the quantity of interest \eqref{eqn:errfunctional}.   We observe that the contributions of model errors follows the profile of the second order adjoint model evolution shown in Figure \ref{fig:soasol}. This is in 
agreement with the theory in Section \ref{sec:probdef}. Some grid points tend to be more sensitive than the others to the errors in the model. This indicates the need for better physical representation, e.g., obtained by increasing grid resolution in the sensitive regions.
\begin{table}
 \centering
 \begin{tabular}{|l|l|l|}
  \hline
  &$\Delta\mathcal{E}_{\rm actual}$&$\Delta\mathcal{E}_{\rm est}$ \\ \hline
  Data Errors & 1.945$\times 10^{-2}$& 2.395$\times 10^{-2}$ \\ \hline
  Model Errors& 2.561$\times 10^{-2}$ & 1.819$\times 10^{-2}$ \\ \hline
 \end{tabular}
 \caption{The comparison between actual error and the a posteriori error estimates for the heat equation.}
 \label{tab:error:he}
\end{table}
\begin{figure}[ht]
  \centering
  \subfigure[Forward model\label{fig:forward}]{\includegraphics[scale=0.30]{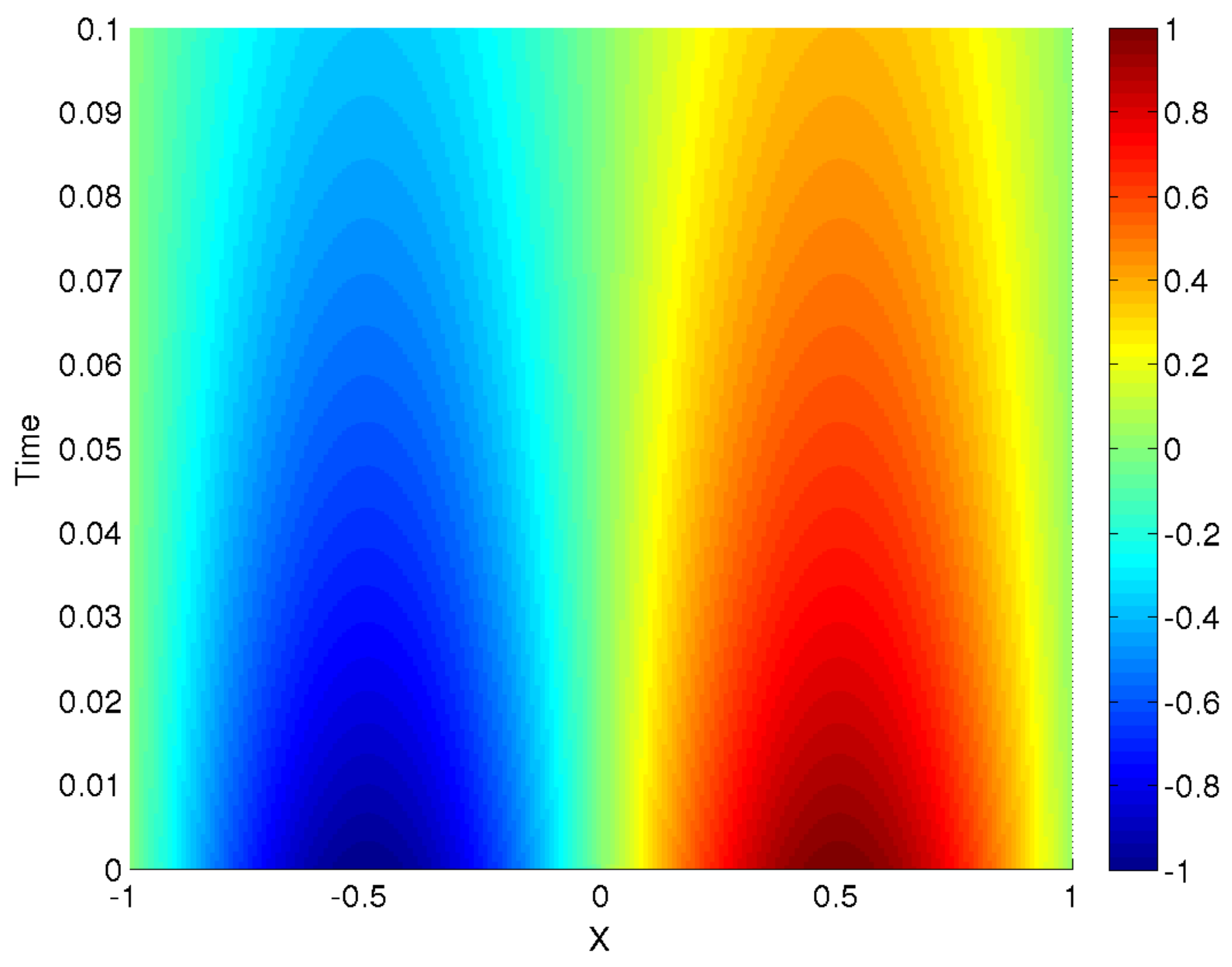}}
  \subfigure[Tangent linear solution\label{fig:tlmsol}]{\includegraphics[scale=0.30]{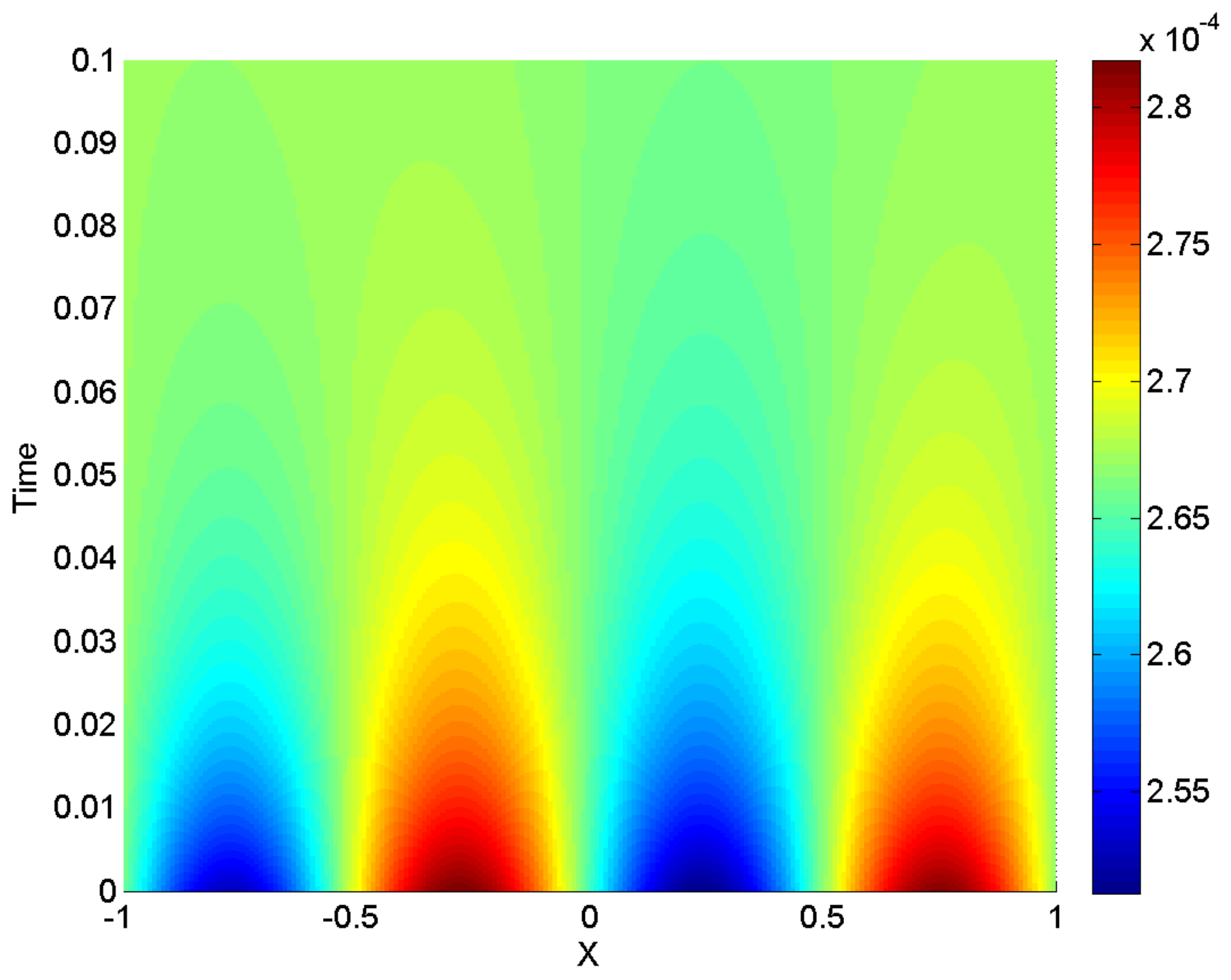}}
  \subfigure[Adjoint solution\label{fig:adjoint}]{\includegraphics[scale=0.30]{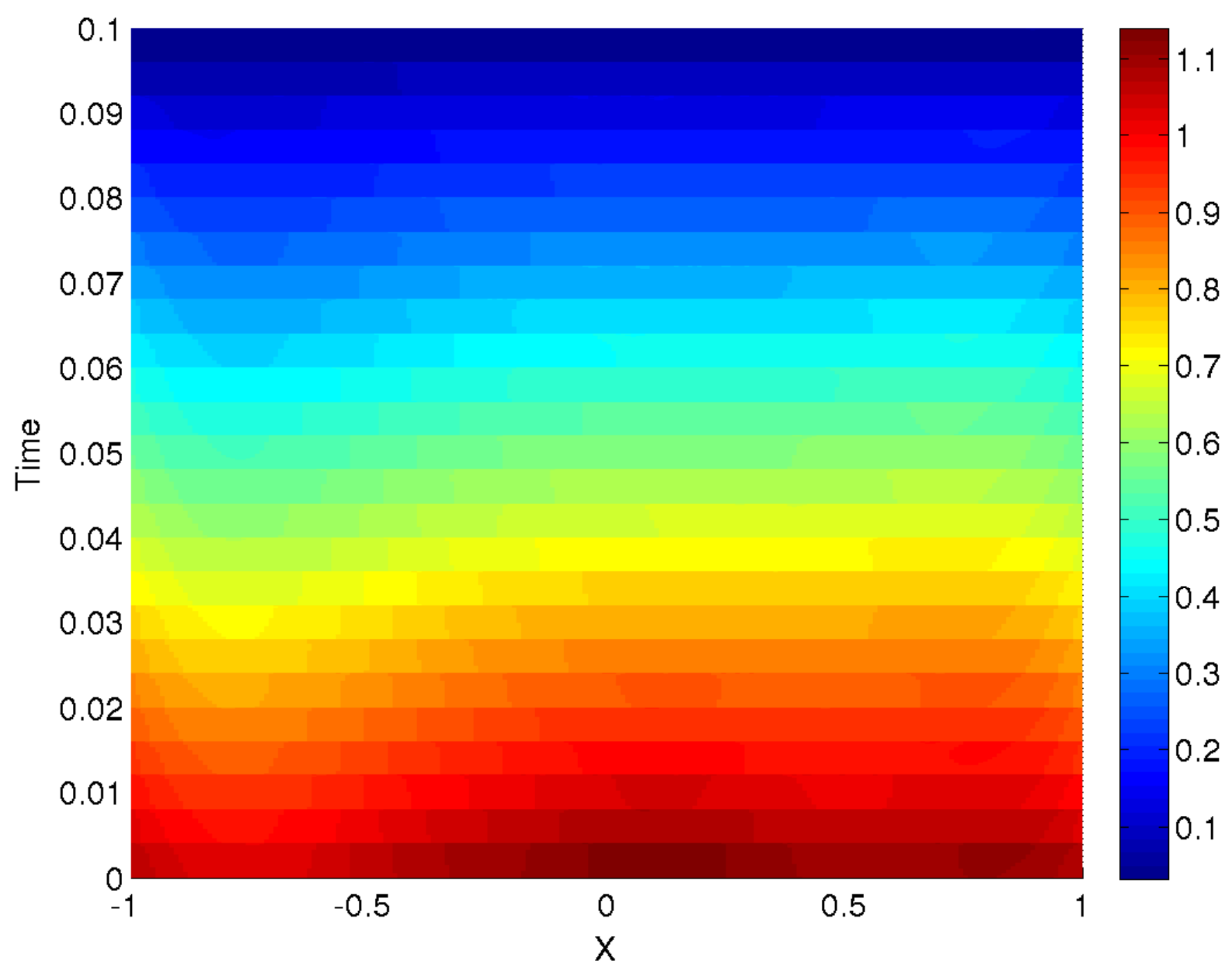}}
  \subfigure[Second order adjoint solution\label{fig:soasol}]{\includegraphics[scale=0.30]{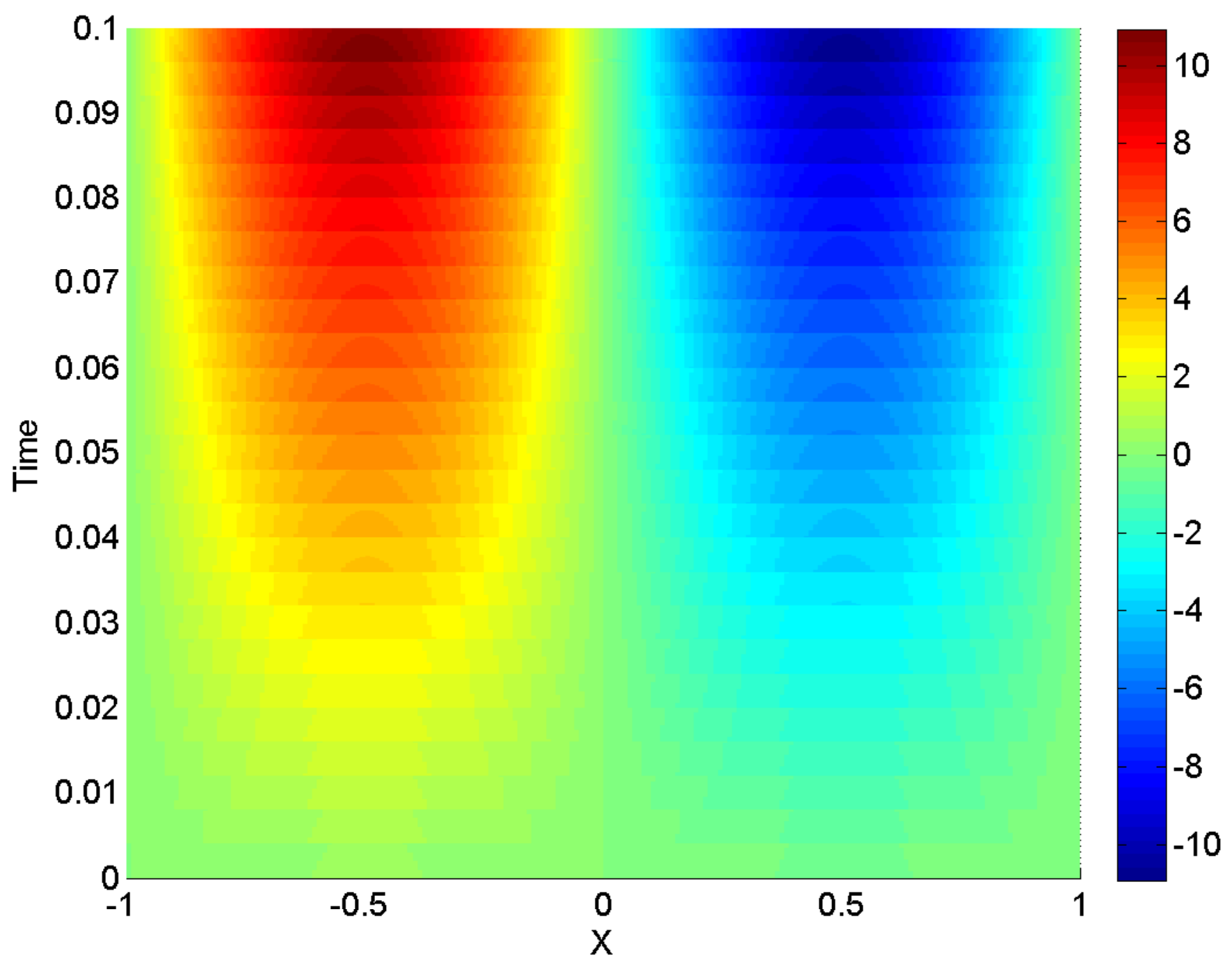}}
  \caption{The evolution of forward, tangent linear, and adjoint variables for the heat equation (equations \eqref{eqn:he} and \eqref{eqn:periodicpde}).}
\end{figure}
\begin{figure}[ht]
 \centering
 \subfigure[Data errors\label{fig:dataerrors}]{\includegraphics[scale=0.3]{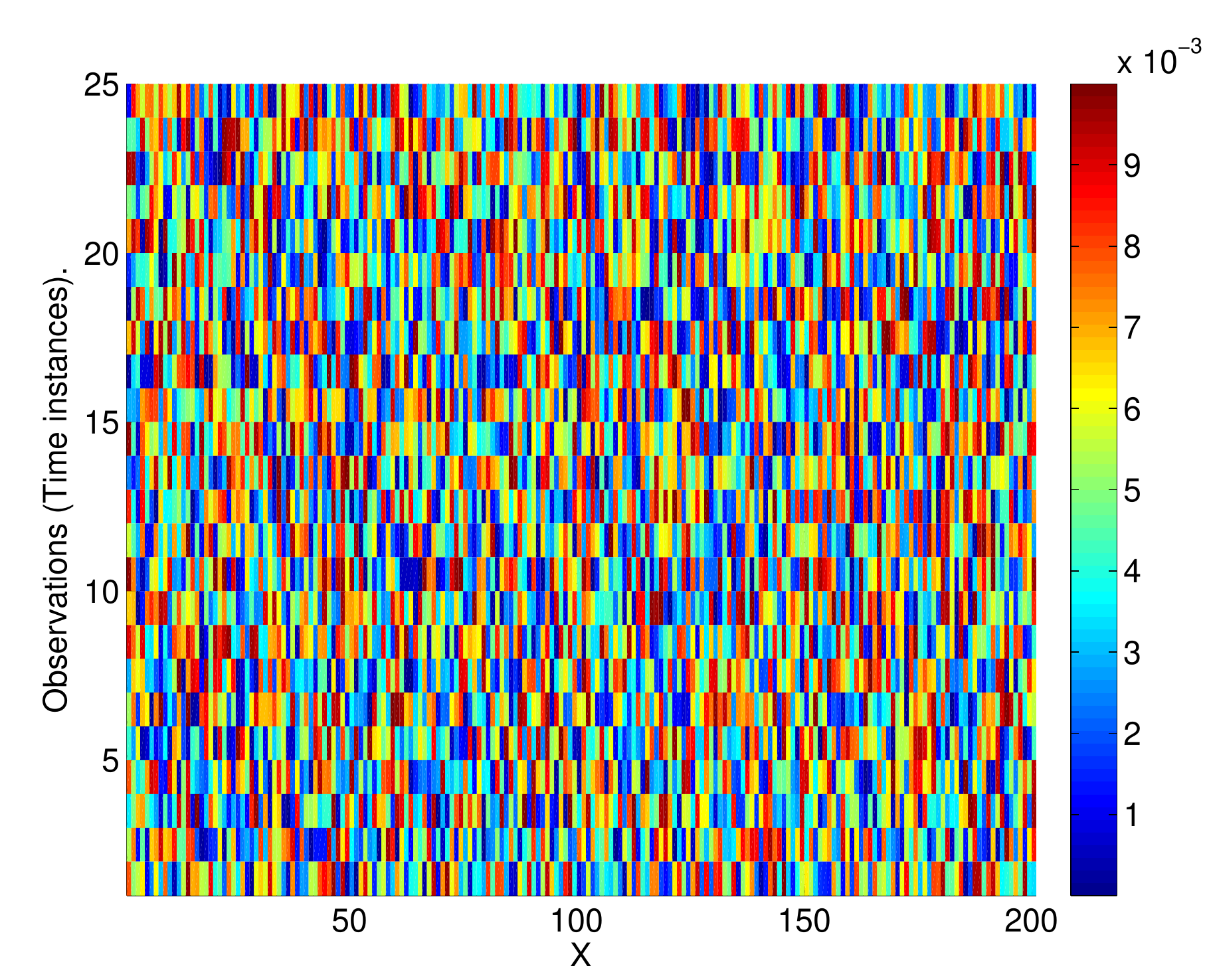}}
 \subfigure[Contributions of data errors to error in the quantity of interest \eqref{eq:eactual}\label{fig:datacontr}]{\includegraphics[scale=0.3]{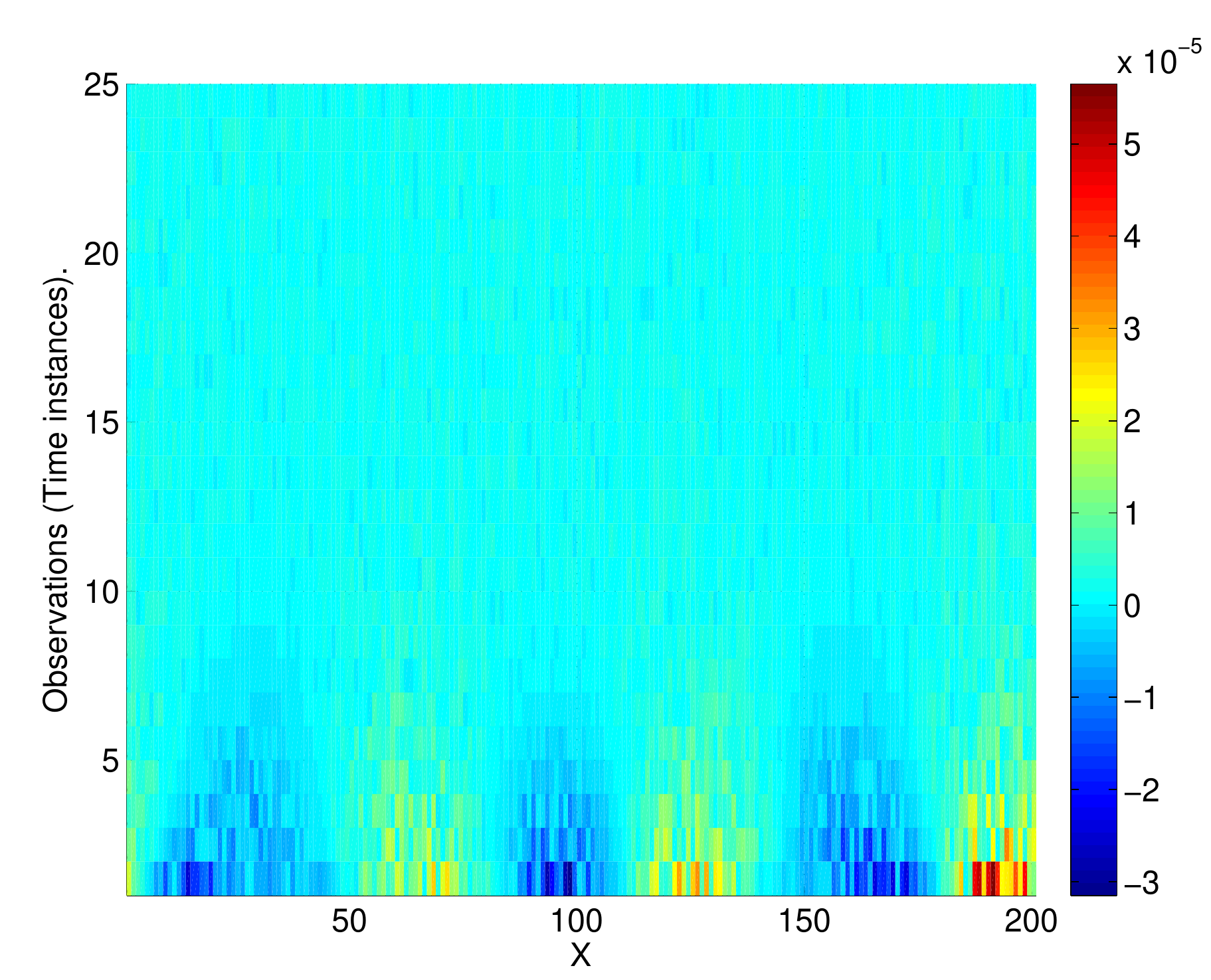}}
 \subfigure[Contributions of model errors to error in the quantity of interest \eqref{eq:eactual}\label{fig:modelcontr}]{\includegraphics[scale=0.3]{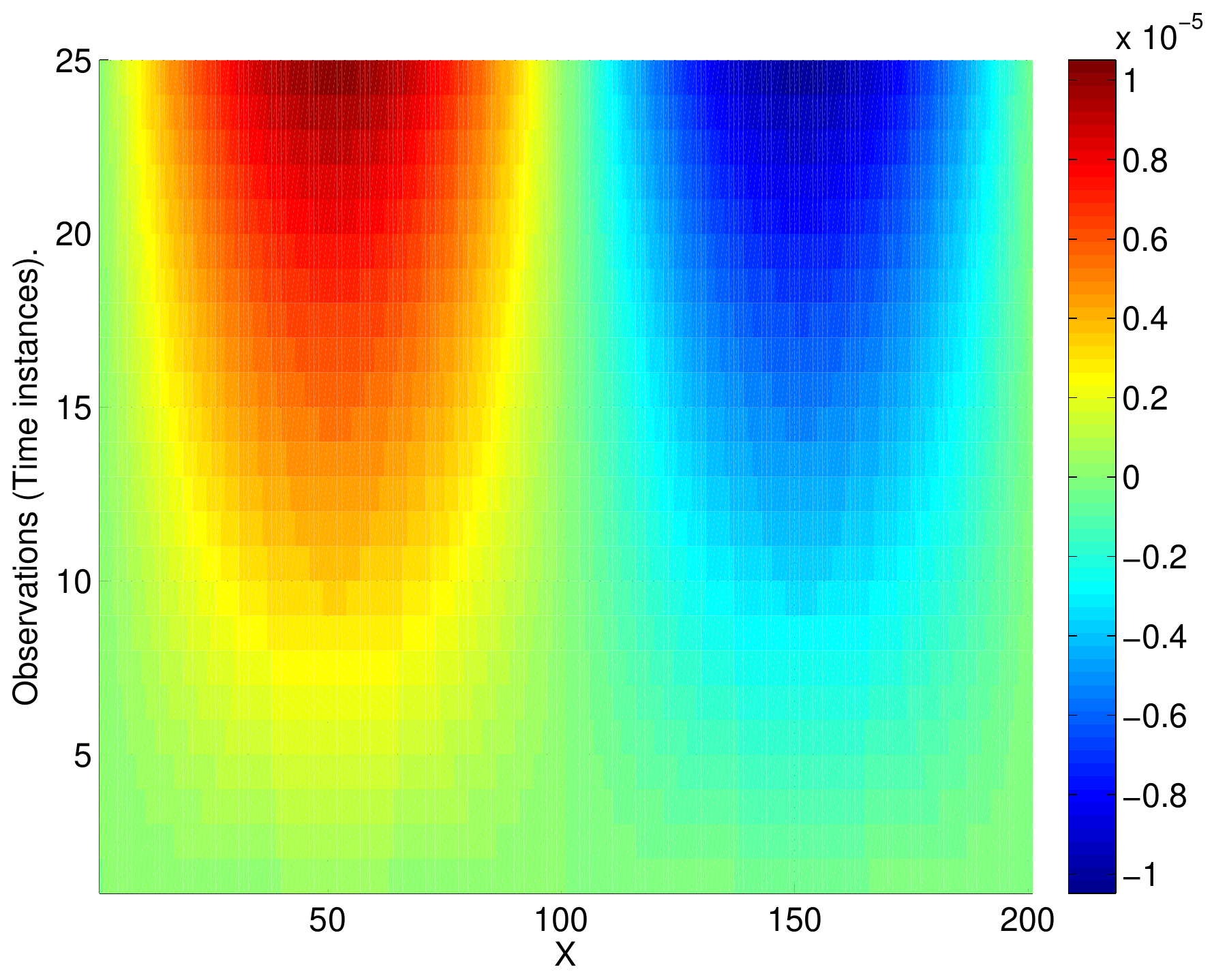}}
 \caption{Data errors at different grid points for the heat equation (equations \eqref{eqn:he} and \eqref{eqn:periodicpde}) and the contributions to the error functional resulting from data and model errors.}
%\label{fig:datacontr}
\end{figure}
% 
%%%%%%%%%%%%%%%%%%%%%%%%%%%%%%%%%%%%%%%%%%%%%%%%%%%%%%%%%%%
\subsection{Shallow water model on a sphere}\label{sec:swe}
%%%%%%%%%%%%%%%%%%%%%%%%%%%%%%%%%%%%%%%%%%%%%%%%%%%%%%%%%%%
The shallow water equations have been used to model the atmosphere for many years. They contain the essential wave propagation mechanisms found in general circulation models (GCMs)\cite{Amik:2007}. The shallow water equations in spherical coordinates are:
\begin{subequations}
\label{eqn:swe}
\begin{eqnarray}
 \frac{\partial u}{\partial t} + \frac{1}{a\cos \theta} \left( u \frac{\partial u}{\partial \lambda} + v \cos \theta \frac{\partial u}{\partial \theta} \right) - \left(f + \frac{u \tan \theta}{a} \right) v + \frac{g}{a \cos \theta} \frac{\partial h} {\partial \lambda} = 0, \\
 \frac{\partial v}{\partial t} + \frac{1}{a\cos \theta} \left( u \frac{\partial v}{\partial \lambda} + v \cos \theta \frac{\partial v}{\partial \theta} \right) + \left(f + \frac{u \tan \theta}{a} \right) u + \frac{g}{a} \frac{\partial h} {\partial \theta} = 0, \\
 \frac{\partial h}{\partial \theta} + \frac{1}{a \cos \theta} \left(\frac{\partial\left(hu\right)}{\partial \lambda} + \frac{\partial{\left(hv \cos \theta \right)}}{\partial \theta} \right) = 0.
\end{eqnarray}
\end{subequations}
Here, $f$ is the Coriolis parameter given by $f = 2 \Omega \sin \theta$, where  $\Omega$ is the angular speed of the rotation of the Earth, $h$ is the height of the homogeneous atmosphere, $u$ and $v$ are the zonal and meridional wind components, respectively, $\theta$ and $\lambda$ are the latitudinal and longitudinal directions, respectively, $a$ is the radius of the earth and $g$ is the gravitational constant. The space discretization is performed using the unstaggered Turkel-Zwas scheme \cite{neta:1997}. The discretization has nlon=72 nodes in longitudinal direction and  nlat=36 nodes in the latitudinal direction. The code we use for the forward model is a {\sc matlab} version of the {\sc fortran} code developed by Daescu and Navon and used in the paper \cite{NavonDaescu:2004}. The semi-discretization in space leads to the following discrete model: 
% %
% \begin{equation}\label{swe:ode}
%  \x^{\prime} = f(t,\x) \quad \quad \x(t_0) = \x_0, \quad \quad t=[0\,,\, 24] ,\textnormal{(hours)}.
% \end{equation}
%
%Equation \eqref{swe:ode} can also represented in the form of a discrete model as follows:
%
\begin{equation}\label{swe:ode}
 \x_{k+1} = \mathcal{M}\left(\x_k, \theta\right) \quad \quad \x_0 = \x_0\left(\theta\right), \quad \quad k = 0, \dots, N.
\end{equation}
  In \eqref{swe:ode}, the zonal wind, meridional wind and the height variables are combined into the vector $\x \in \mathbb{R}^n$ with $n=3\times{\rm nlat}\times{\rm nlon}$. We perform the time integration using an adaptive time-stepping algorithm. For a tolerance of $\displaystyle 10^{-8}$ the average time step size is 180 seconds. A reference initial condition is used to generate a reference trajectory.

Synthetic observation errors at various times $t_k$ are normally distributed  with mean zero and a diagonal observation error covariance matrix with entries equal to $(\mathbf{R_k})_{i,i}=1$ for $u$ and $v$ components and
 $(\mathbf{R_k})_{i,i}=10^6$ for $h$ components. The $\mathbf{R_k}$ values correspond to a standard deviation of $5\%$ for $u$ and $v$ components, and $2\%$ for $h$ component.
 Synthetic observations are obtained by adding the synthetic observation noise to the reference solution at times $t_k$.
 The background error covariance matrix is also diagonal with entries equal to $(\mathbf{B_0})_{i,i}=1$ for $u$ and $v$ components and $(\mathbf{B_0})_{i,i}=10^6$ for $h$ components.
 
Model errors are introduced in the form of random correlated noise. We build statistical models of model errors and consider different realizations in Section \ref{sec:modelerrors}. The cost function has the form \eqref{eqn:4D-Var:imperfect}. 

The \qoi is the mean of the height component of the analysis (the optimal initial condition)
\begin{equation} 
\label{eqn:errfunc}
\mathcal{E}\left(\x^{\rm a}_0\right) = \frac{1}{{\rm nlat} \times {\rm nlon}}\,\sum_{i=2\times {\rm nlat} \times {\rm nlon}+1}^{3\times {\rm nlat} \times {\rm nlon}} \left(\x^{\rm a}_0\right)_i.
\end{equation}
%
%It should be noted that we use the model \eqref{swe:ode} and its discrete form, \eqref{swe:odeDisc} for a posteriori error estimation with continuous  and discrete frameworks respectively.
%%%%%%%%%%%%%%%%%%%%%%%%%%%%%%%%%%%%%%%%%%%%%%%%%%%%%%%%%%%%%%%%%%%%%%%
\subsection{Statistical models for model errors}\label{sec:modelerrors}
%%%%%%%%%%%%%%%%%%%%%%%%%%%%%%%%%%%%%%%%%%%%%%%%%%%%%%%%%%%%%%%%%%%%%%%
To realistically simulate model errors we consider differences between the shallow water solutions obtain on a coarse and on a fine grid. The coarse grid was discussed in Section \ref{sec:swe}.  The fine grid has a spatial resolution of {\rm nlat $\times$ nlon = 108 $\times$ 72}, three times smaller than the coarse grid. The time integration is also performed at a finer temporal resolution realized by using the {\sc matlab}'s {\sc ode45} integrator. The {\sc{atol}} and {\sc{rtol}} are both set to $10^{-12}$. The solution fields obtained one the fine grid are perturbed to produce synthetic observations, which are then used for the coarse grid data assimilation.

The differences between model solutions on the fine grid (projected onto the coarse) and on coarse grid are used as proxies for the model errors. The procedure used to generate the ensemble of model errors is as follows. Integrate the model on the fine grid for the simulation window. Divide the simulation window into sub-intervals $[t_k,t_{k+1}]$ of length $t_{k+1}-t_k=400$ seconds. At the beginning of each sub-interval project the solution values from the fine grid onto the coarse grid. Use these values as coarse grid initial solutions, and run the coarse model on each sub-interval. The differences between the coarse and fine solutions at the end of each sub-interval $[t_k,t_{k+1}]$ (projected onto the coarse model space) represent the model error terms $\Delta\widehat{\x}_{k+1}$ in \eqref{eqn:model-p}. The procedure summarized above, is used to generate a total of 216 error vectors (an ensemble member is collected every 400 seconds for a period of 24 hours). We make the assumption that model errors are 
stationary and use this ensemble of differences to build statistical models of model errors. 

To find an appropriate description of model errors we consider a variety of distributions and fit the model errors using the Bayesian information criterion (BIC). The BIC is a criterion for model selection among a finite set of models that resolves the problem of overfitting by introducing a penalty term for the number of parameters in the model \cite{Schwarz:1978, allfitdist}.
\begin{figure}[ht]
 \centering
 \subfigure[The accuracy of fit for different distributions for a zonal velocity component \newline(Lat = $87^\circ 30^{\prime}\text{S } 7^{\prime \prime}$, Lon = $174^\circ\text{W } 10^\prime 3^{\prime \prime}$ ). ]{\includegraphics[width=0.475\textwidth]{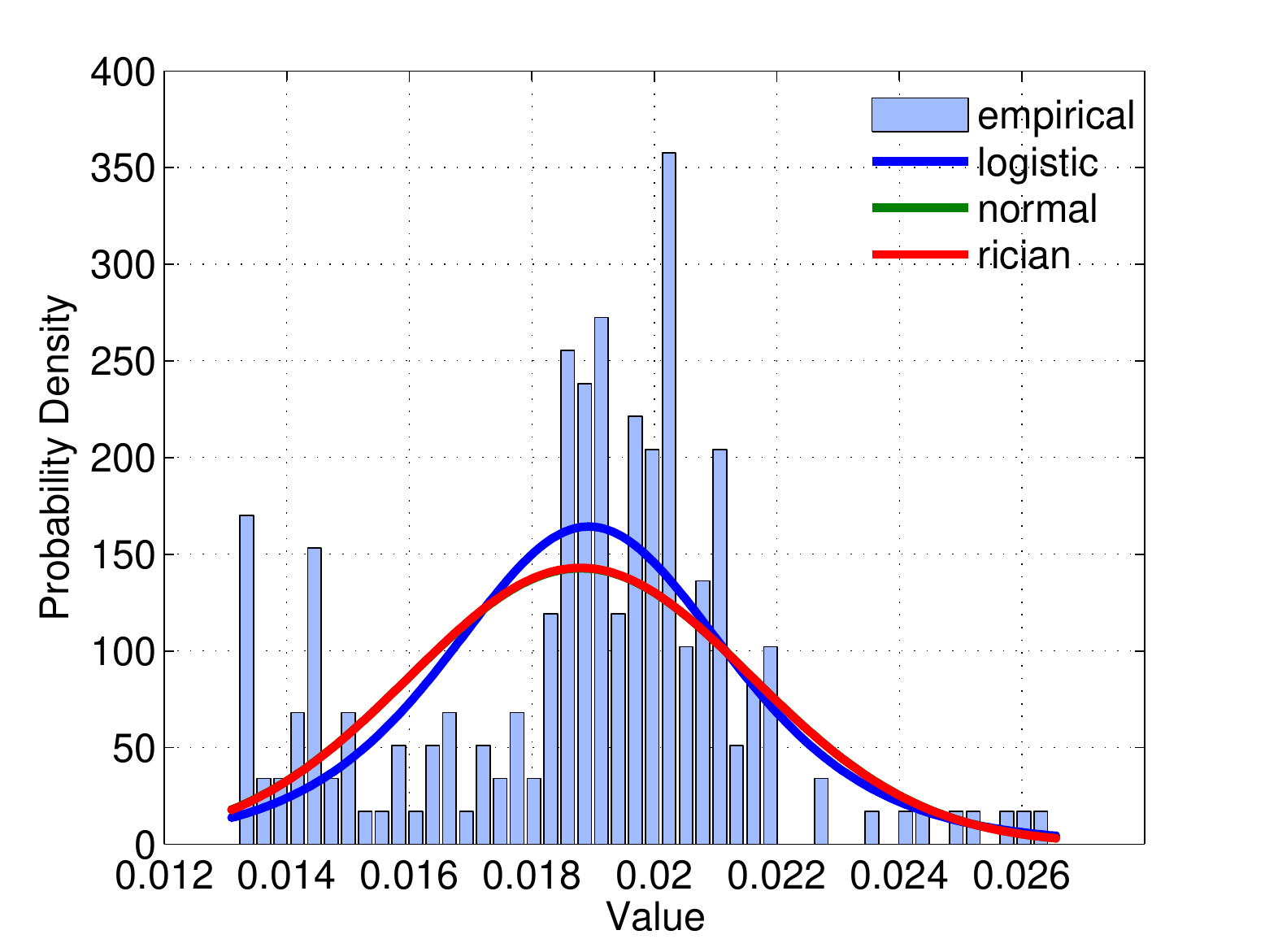}}
 \subfigure[The accuracy of fit for different distributions for a zonal velocity component \newline(Lat = $87^\circ\text{S } 30^{\prime} 7^{\prime \prime}$, Lon = $177^\circ\text{E } 30^\prime 9^{\prime \prime}$ ).]{\includegraphics[width=0.475\textwidth]{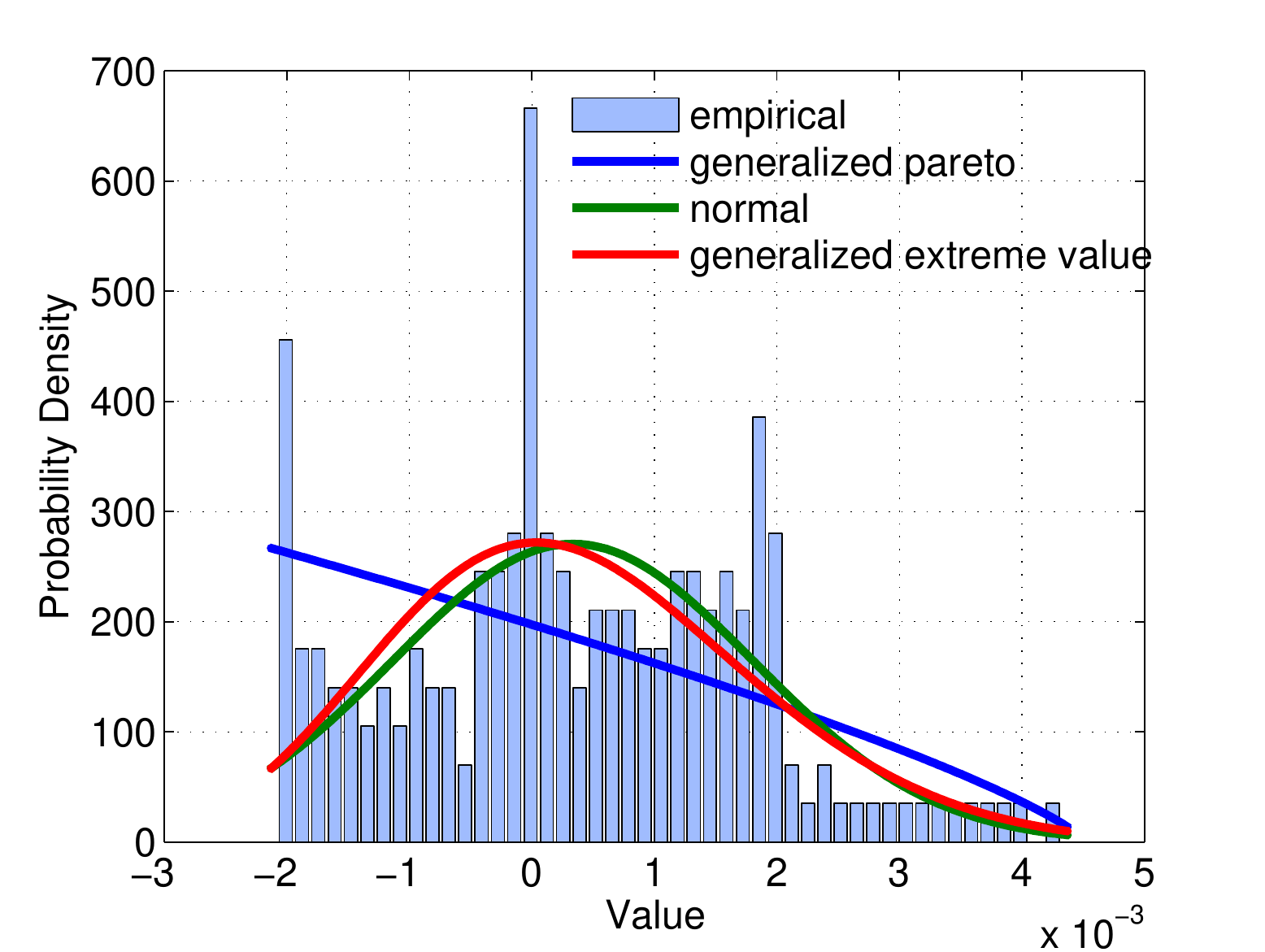}}
 \caption{The model errors for the shallow water equations \eqref{eqn:swe} are fit to different distributions based on the Bayesian information criterion for some samples. The plots show the top three best fits.}
 \label{fig:DistributionFits1}
\end{figure}

We first seek one probability distribution that can best describe the errors at each of the 7,776 individual grid points. Different distribution families are used to fit the ensembles of errors. As shown in Figures \ref{fig:DistributionFits1} and \ref{fig:DistributionFits2}, no distribution fits the error completely satisfactorily. The BIC criterion ranks the suitability of different distributions for each grid point, and Table \ref{tab:distfits} shows the number of grid points where the most successful fits appear in top three. Since the normal distribution consistently ranks in the top three we choose to model the model errors as a Gaussian process. There is a considerable inter-grid correlation of errors.  The scaled Bessel functions of the first kind \cite{Abramowitz:2012} are used to model inter-grid correlation functions and their parameters are obtained by fitting to the actual values. Figure \ref{fig:CorrelationModeling} shows the comparison between actual correlation values and the correlation 
modeled with Bessel functions. We construct the error correlation matrix using inter-grid correlations modeled by the Bessel functions of the first kind. We use the resulting covariance matrix and  the mean of the ensembles of real errors to generate different realizations of model errors. These realizations correspond to the terms $\Delta \widehat{\x}_{k+1}$ in equation \eqref{eqn:model-p}. The multiple instances of model errors help with the statistical validation of the a posteriori error estimates discussed in Section \ref{sec:statValEsts}.
\begin{table}[ht]
 \centering
 \begin{tabular}{|l|c|}
  \hline
  Distribution name & No of best fits in top three \\ \hline
  Generalized extreme value & 7,608 \\ \hline
  Normal & 7,247\\ \hline
  Tlocation scale & 7,052 \\ \hline
 \end{tabular}
 \caption{The number of best fits for selected distributions. This is the number of grid points where the distributions are ranked in top three by the Bayesian information criterion as the comparison metric.}
 \label{tab:distfits}
\end{table}
 \begin{figure}[ht]
 \centering
 \subfigure[The accuracy of fit for different distributions for a meridional velocity component \newline(Lat = $87^\circ\text{S } 30^{\prime} 7^{\prime \prime}$, Lon = $114^\circ\text{W } 10^\prime 1^{\prime \prime}$ ).]{\includegraphics[width=0.475\textwidth]{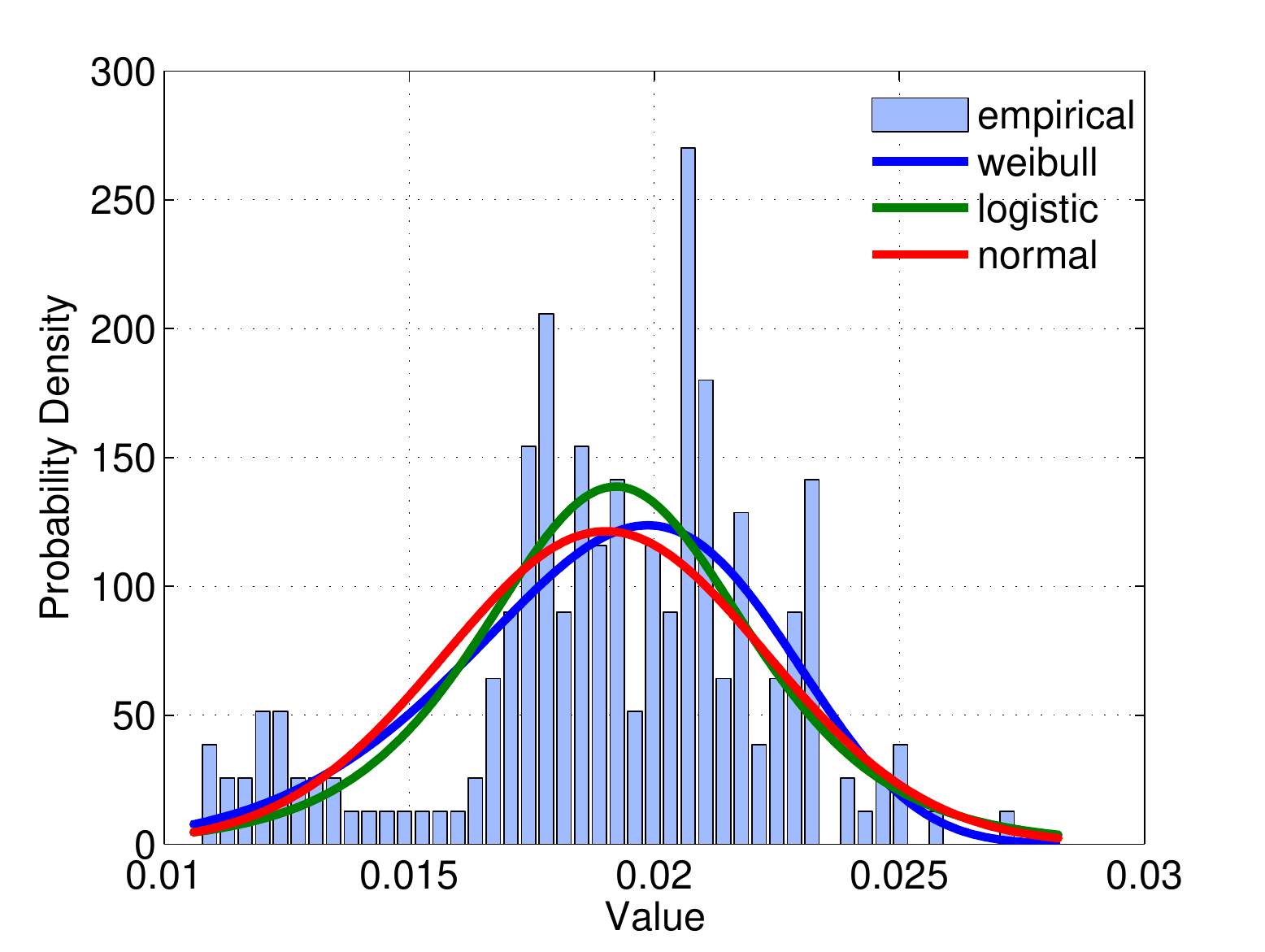}}
 \subfigure[The accuracy of fit for different distributions for a height component \newline(Lat = $87^\circ\text{N } 30^{\prime} 7^{\prime \prime}$, Lon = $52^\circ\text{E } 30^\prime 1^{\prime \prime}$ ).]{\includegraphics[width=0.475\textwidth]{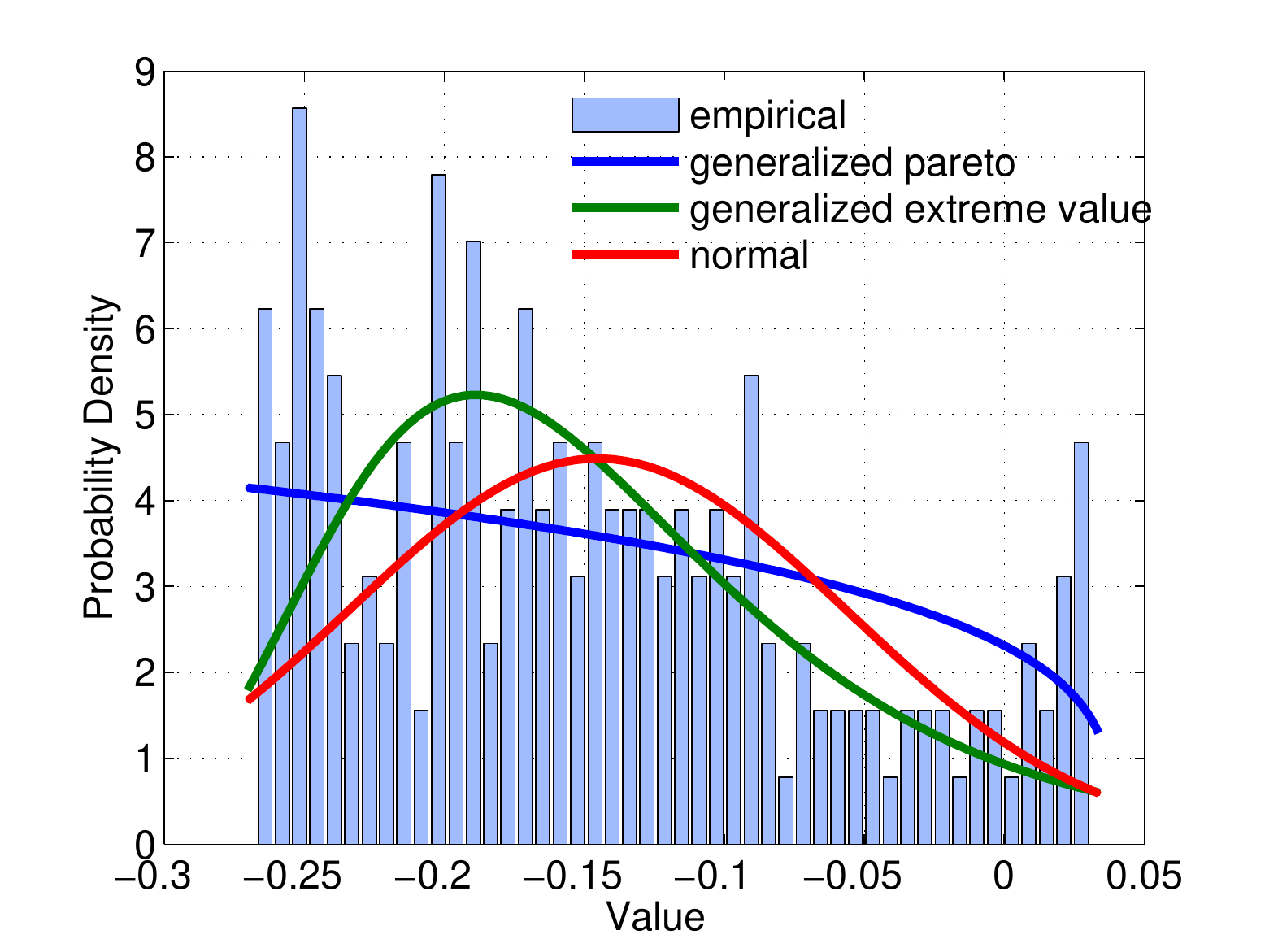}}
 \caption{The model errors are fit to different distributions for the shallow water model \eqref{eqn:swe} based on the Bayesian information criterion for some samples. The plots show the top three best fits.}
 \label{fig:DistributionFits2}
\end{figure}
\begin{figure}[ht]
 \centering
 \subfigure[East-West correlation of model errors]{\includegraphics[width=0.475\textwidth]{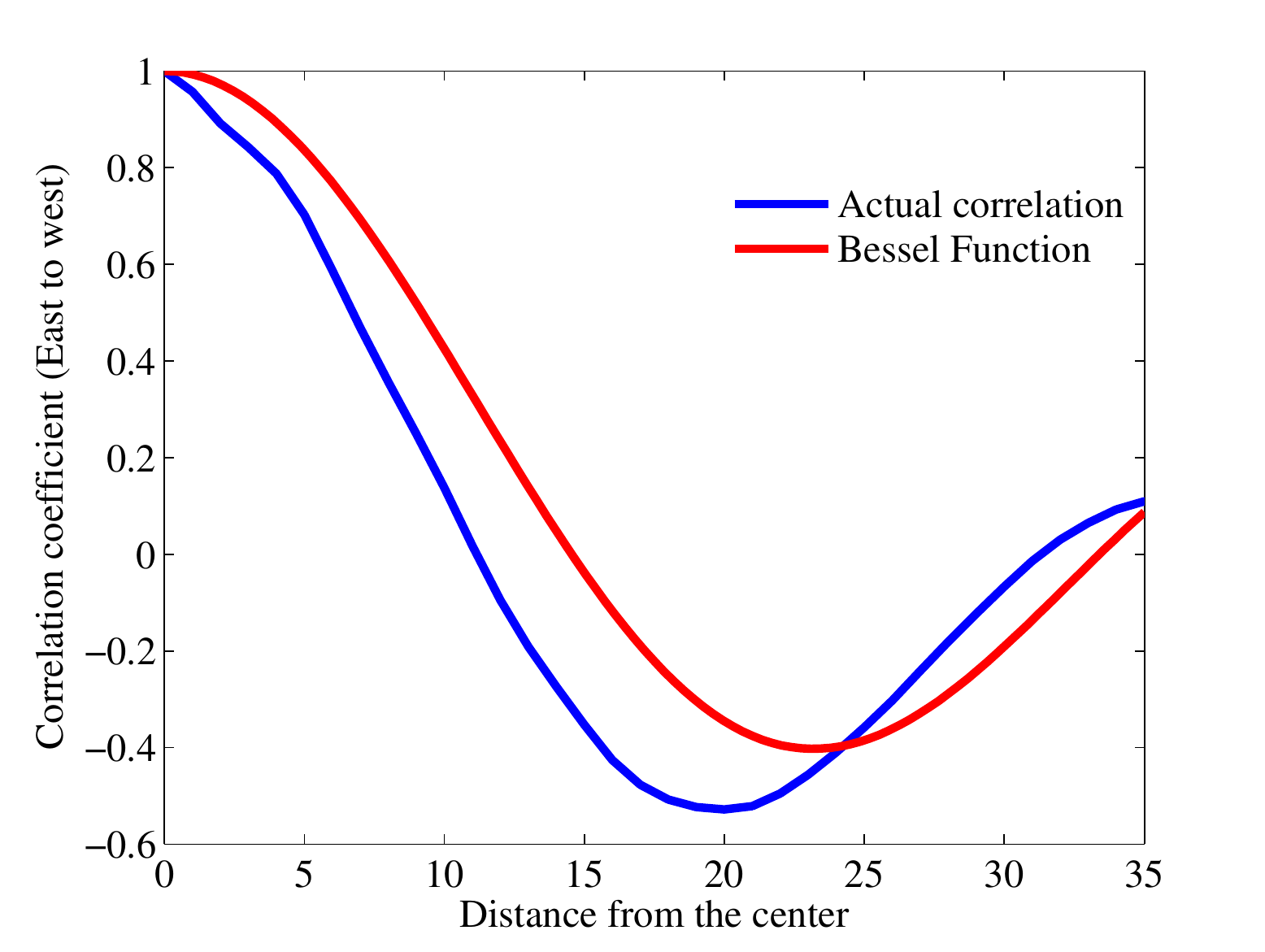}}
 \subfigure[North-South correlation of model errors]{\includegraphics[width=0.475\textwidth]{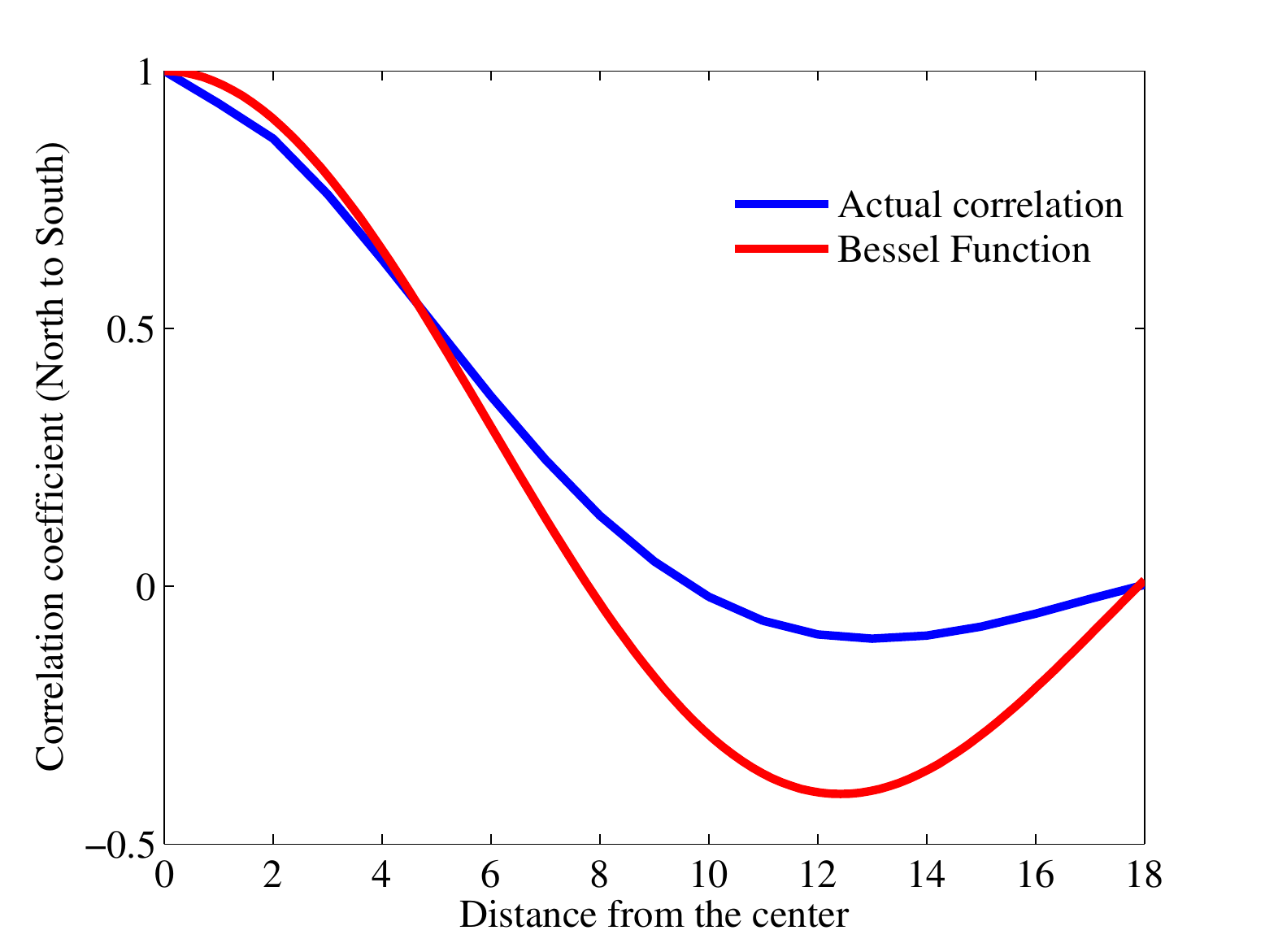}}
 \caption{Correlations between model errors at neighboring grid points. The actual values are obtained form the ensemble of runs. They are modeled by Bessel functions of the first kind.}
\label{fig:CorrelationModeling}
\end{figure}
%

%%%%%%%%%%%%%%%%%%%%%%%%%%%%%%%%%%%%%%%%%%%%%%%%%%%%
\subsection{Validation of a posteriori error estimates in deterministic setting}
%%%%%%%%%%%%%%%%%%%%%%%%%%%%%%%%%%%%%%%%%%%%%%%%%%%%
A posteriori estimates for the error in the \qoi \eqref{eqn:errfunc} due to data and model errors in 4D-Var data assimilation with the shallow water system
are computed using the methodology discussed in the Section \ref{sec:discretemodels}. Table \ref{tab:error:swe} compares the actual errors \eqref{eqn:errorE} and the estimated errors \eqref{eqn:discimpactFdvar}. We observe that the estimates are fairly accurate. Figure \ref{fig:obsErrContr:swe} shows the errors in the individual observations (which are independent and normally distributed) and the corresponding contributions of different observation errors to the error in the quantity of interest \eqref{eqn:errfunc}. We observe that certain grid points contribute to the error more than others. The data error contributions indicate the sensitive areas where measurements need to be more accurate in order to obtain a better analysis (as measured by the \qoi). Larger than expected data error contributions may also point to faulty sensors. Figure \ref{fig:ModelErrContr:swe} shows the model errors at different grid points and their contributions to the error in the \qoi \eqref{eqn:errfunc}. Some grid points are 
more sensitive than others to the errors in the model. This indicates the need for better physical representation, or for higher numerical accuracy (e.g., obtained by increasing grid resolution, or by using higher order time integration)  in the sensitive regions.
\begin{table}[ht]
 \centering
 \begin{tabular}{|l|c|c|}
  \hline
  &$\Delta\mathcal{E}^{\rm actual}$&$\Delta\mathcal{E}^{\rm est}$ \\ \hline
  Data Errors &  54.701 & 57.268 \\ \hline
  Model Errors&  1.9278 & 2.9683 \\ \hline
 \end{tabular}
 \caption{Comparison between actual errors in the \qoi and the a posteriori error estimates for the shallow water model in a deterministic setting.}
 \label{tab:error:swe}
\end{table}
 \begin{figure}[ht]
\begin{center}
\begin{tabular}{cc}
\subfigure[Data errors: Zonal wind velocity, {$u\,[m/s]$}]{\includegraphics[width=0.45\textwidth]{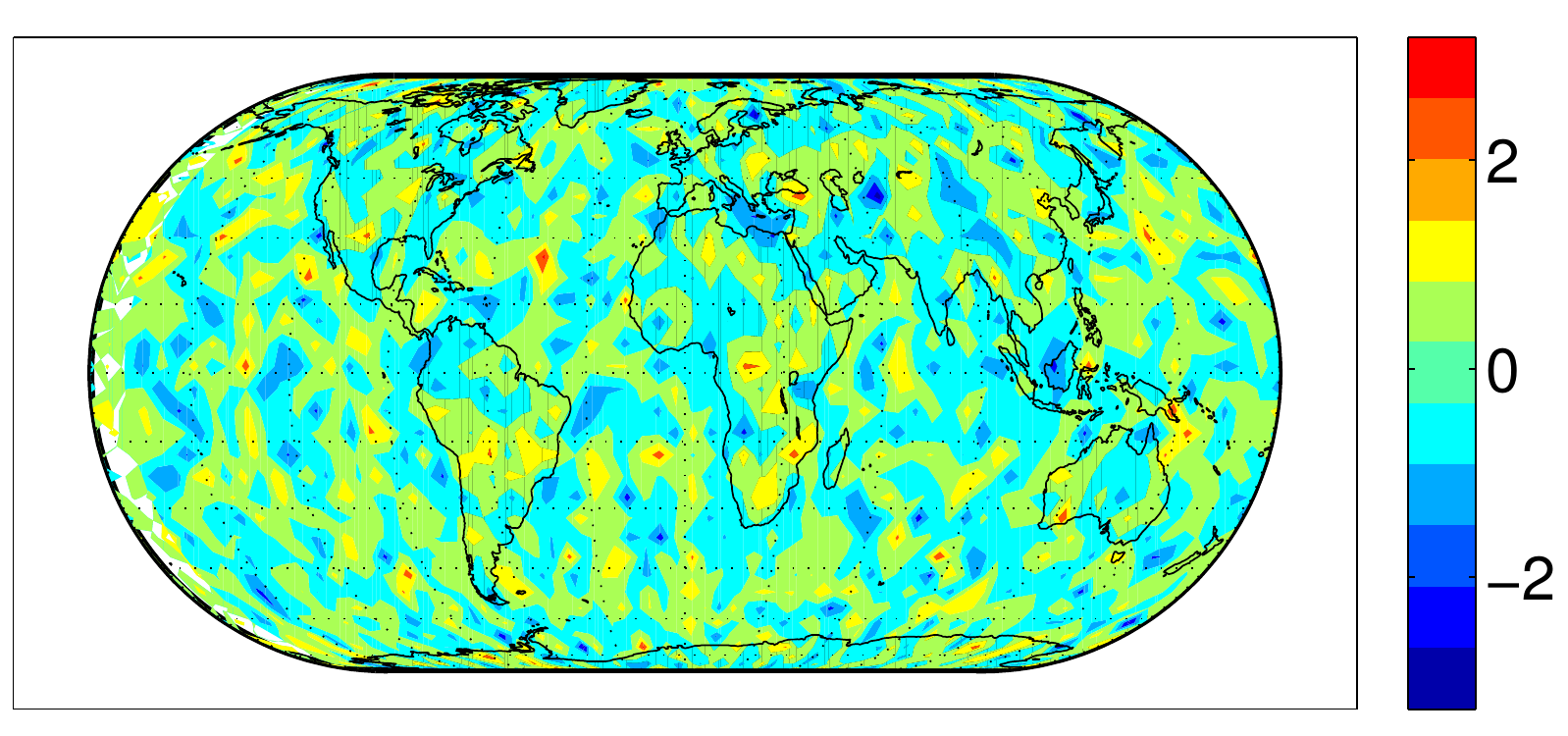}} &
\subfigure[Contributions of data errors: Zonal wind velocity]{\includegraphics[width=0.45\textwidth]{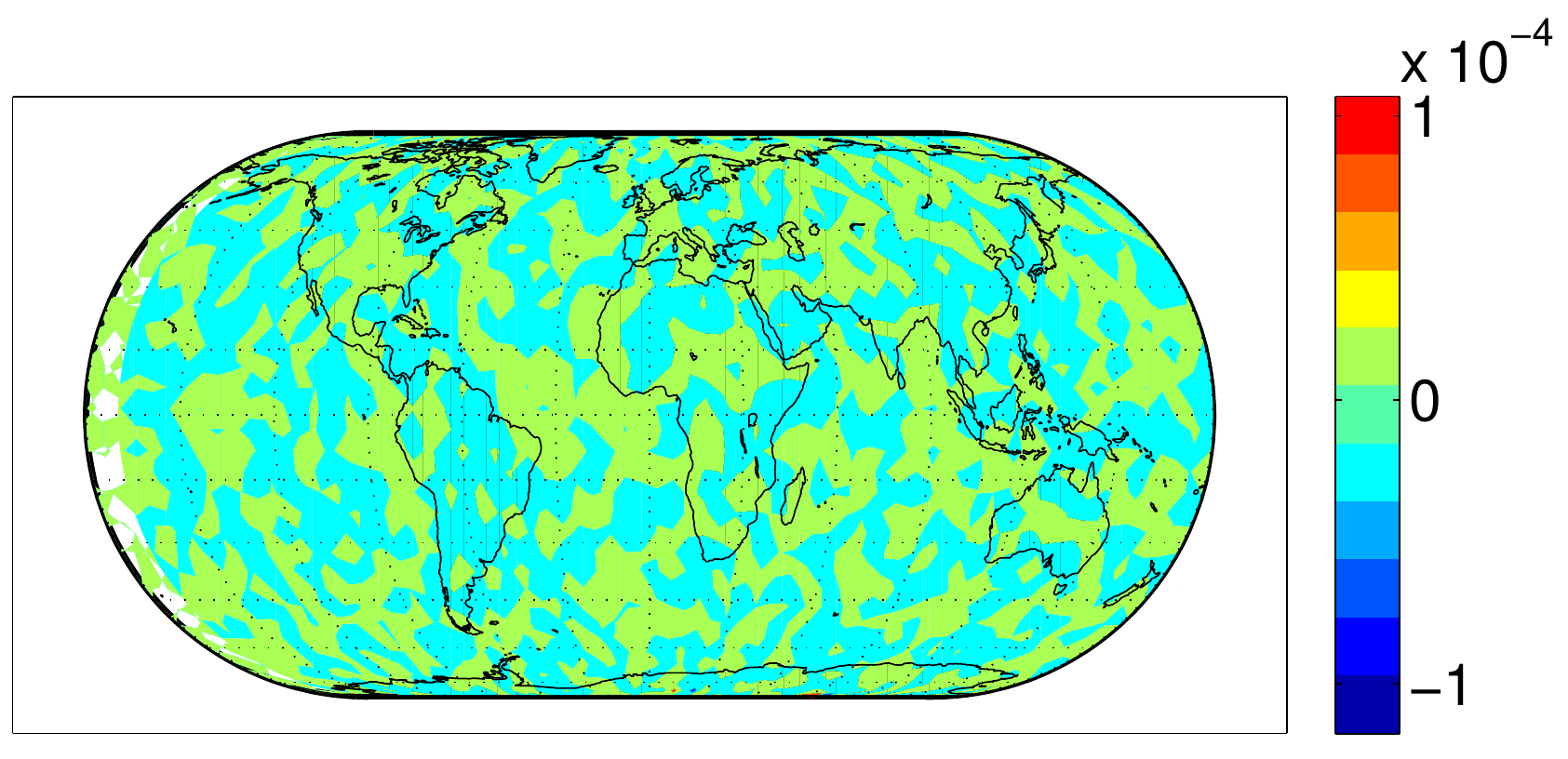}}\\
\subfigure[Data errors: Meridional wind velocity, {$v\,[m/s]$}]{\includegraphics[width=0.45\textwidth]{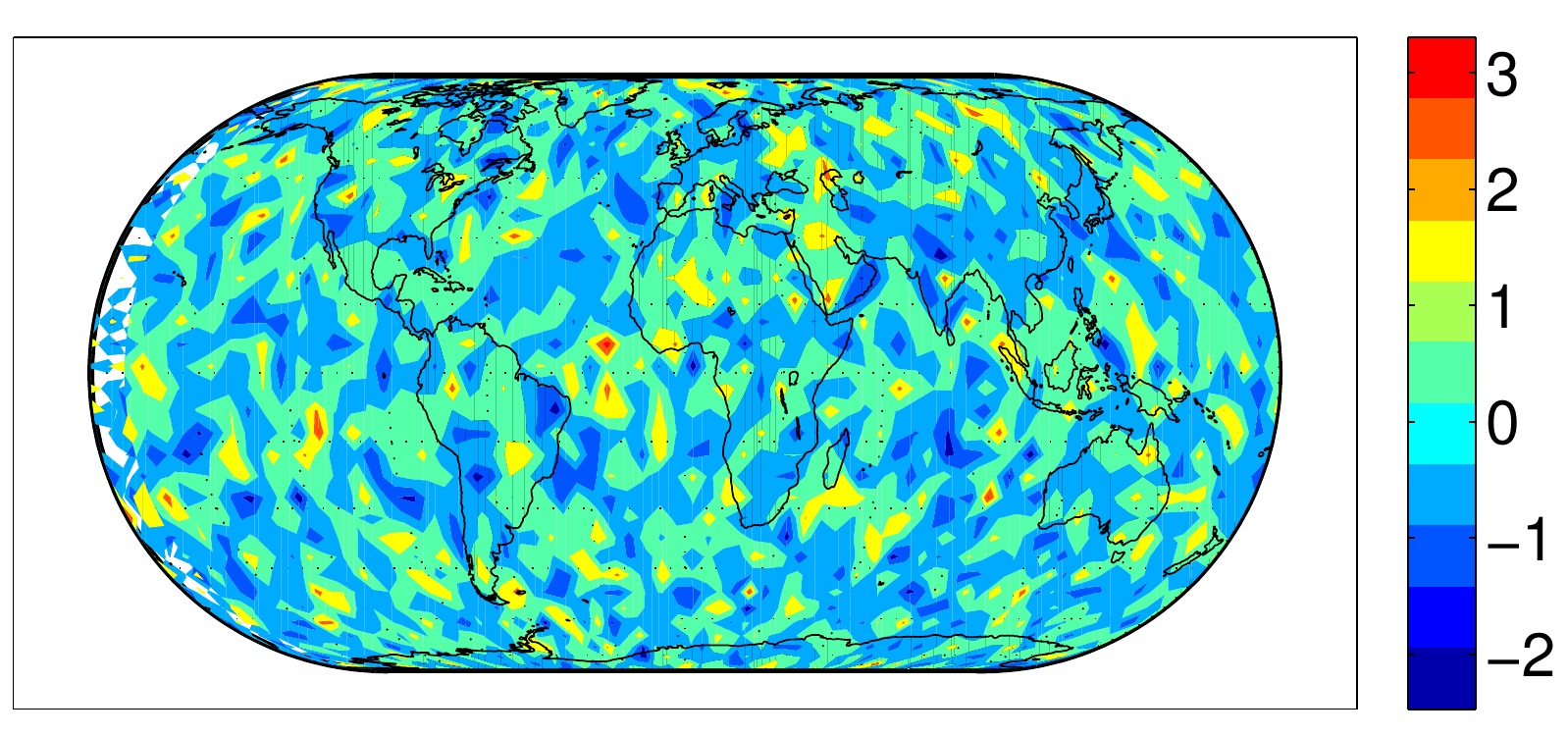}} &
\subfigure[Contributions of data errors: Meridional wind velocity]{\includegraphics[width=0.45\textwidth]{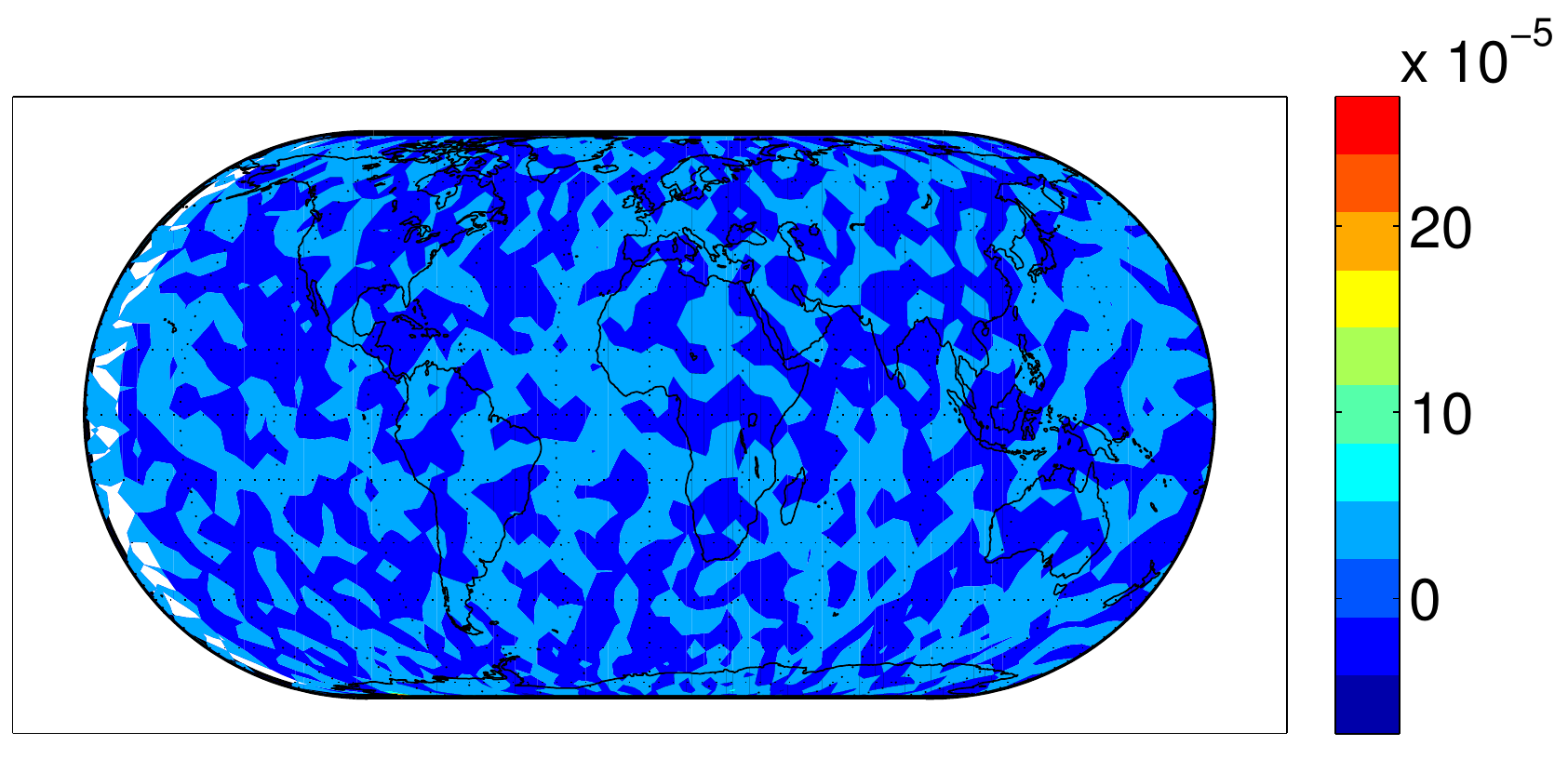}} \\
\subfigure[Data errors: Height, {$h\,[m]$}]{\includegraphics[width=0.45\textwidth]{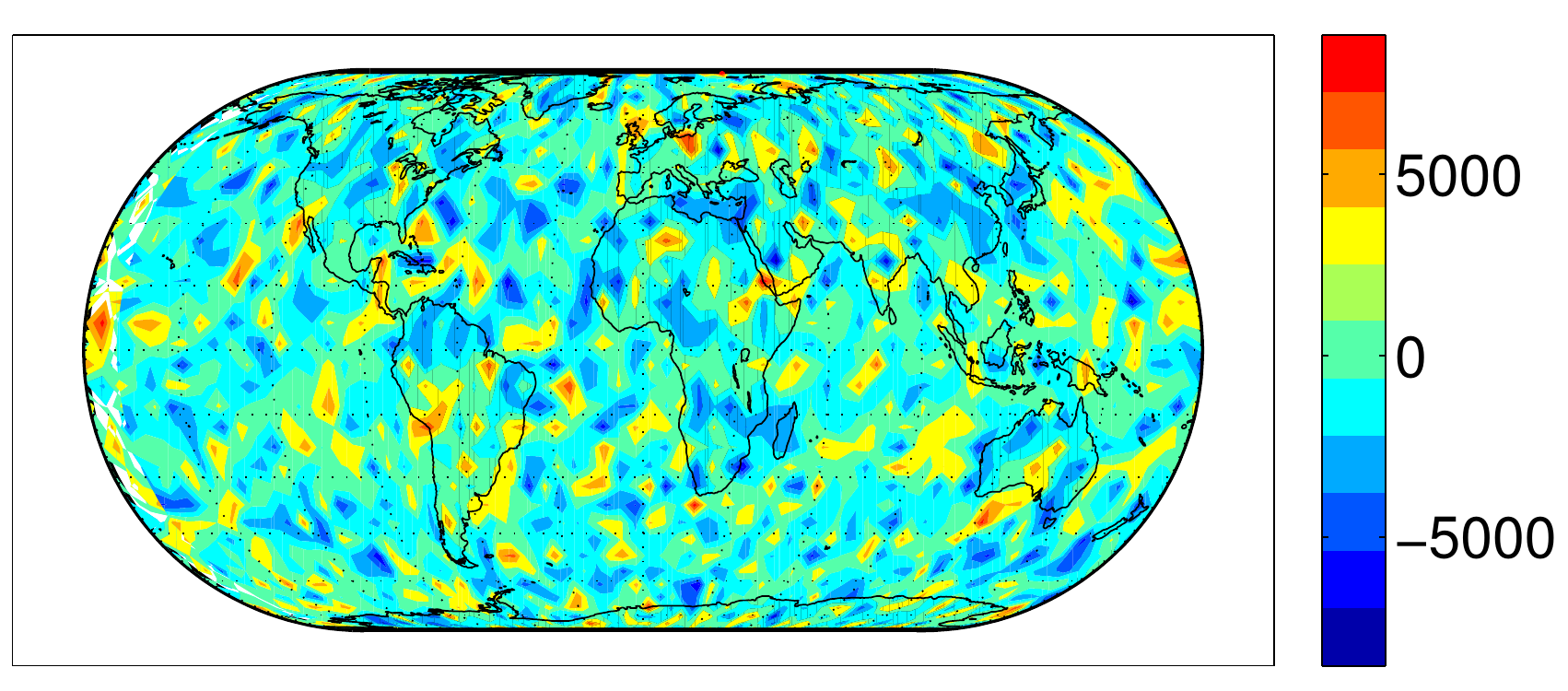}}&
\subfigure[Contributions of data errors: Height]{\includegraphics[width=0.45\textwidth]{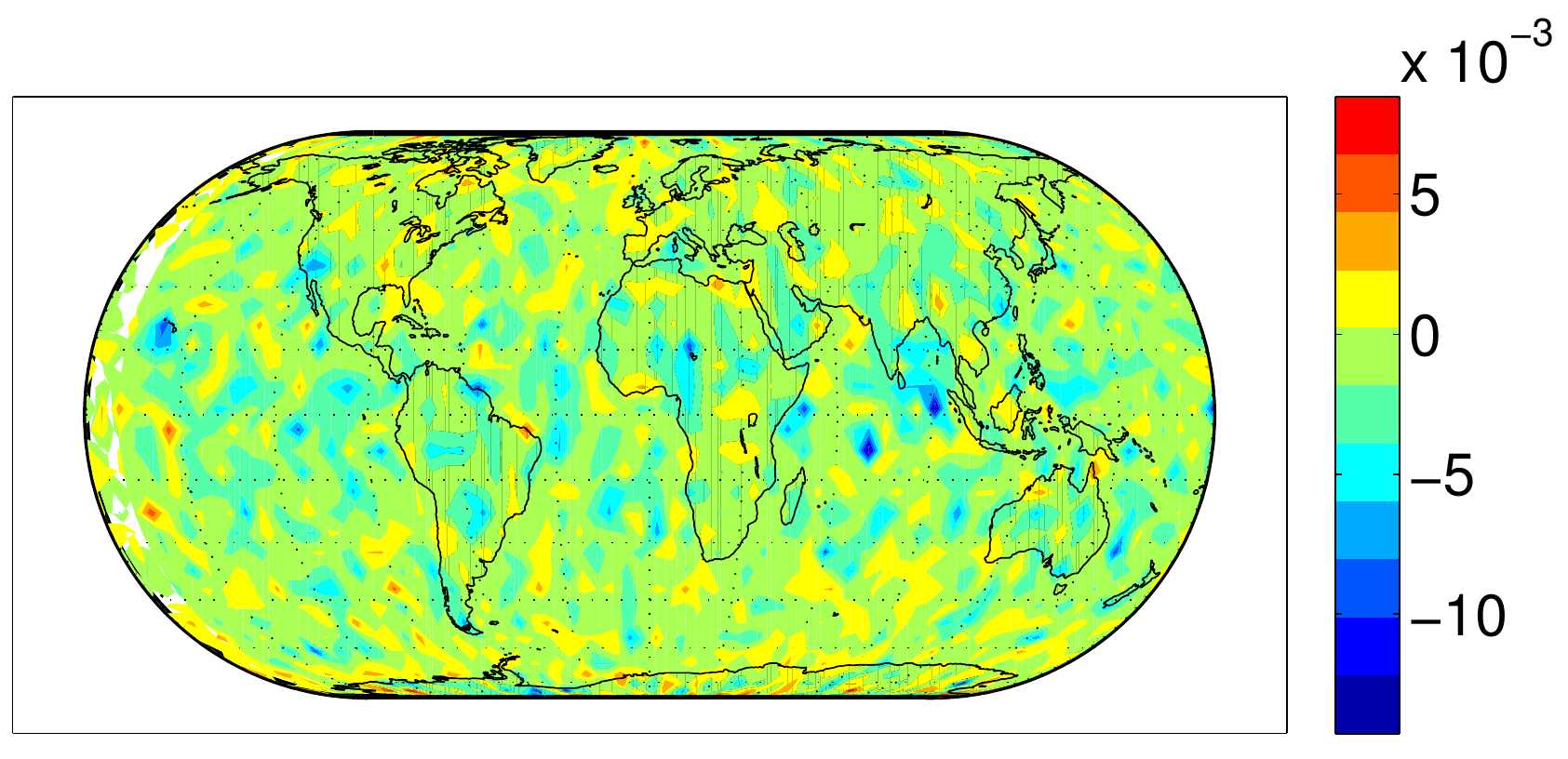}}
\end{tabular}
\end{center}
\caption{The figures on the left show the errors in the data collected for different variables at different grid points for the shallow water model \eqref{eqn:swe} at an observation time $t = 12h$. The figures on the right show the sum total of data error contributions at different grid points to the error functional  \eqref{eqn:errfunc} for hourly observations measured over a period of 24 hours.}
\label{fig:obsErrContr:swe}
\end{figure}
%
\iffalse
 \begin{figure}
  \centering
  \subfigure[Zonal wind velocity, {$u\,[m/s]$}]{\includegraphics[scale=0.50]{DataErrors_U.pdf}}
  \subfigure[Meridional wind velocity, {$v\,[m/s]$}]{\includegraphics[scale=0.50]{DataErrors_V.pdf}}
  \subfigure[Height, {$h\,[m]$}]{\includegraphics[scale=0.50]{DataErrors_H.pdf}}
  \caption{Errors in the data collected for different variables at different grid points for the shallow water model \eqref{eqn:swe}.
  \textcolor{red}{Can you combine Figures \ref{fig:obserrors:swe} and \ref{fig:DataImpact24:swe} in a $3 \times 2$ array of panels, with
  the errors on the left and the contributions on the right? Keep the current Figs under comment in case we want to revert.}
  }
 \label{fig:obserrors:swe}
 \end{figure}

 \begin{figure}
  \centering
  \subfigure[Zonal wind velocity]{\includegraphics[scale=0.50]{DataImpact24_U.pdf}}
  \subfigure[Meridional wind velocity]{\includegraphics[scale=0.50]{DataImpact24_V.pdf}}
  \subfigure[Height]{\includegraphics[scale=0.50]{DataImpact24_H.pdf}}
  \caption{Data error contributions at different grid points to the error functional (equation \eqref{eqn:errfunc}) for the shallow water model \eqref{eqn:swe} for hourly observations measured over a period of 24 hours.}
 \label{fig:DataImpact24:swe}
 \end{figure}
\fi
%
\begin{figure}[ht]
\begin{center}
\begin{tabular}{cc}
\subfigure[Model errors: Zonal wind velocity, {$u\,[m/s]$}]{\includegraphics[width=0.45\textwidth]{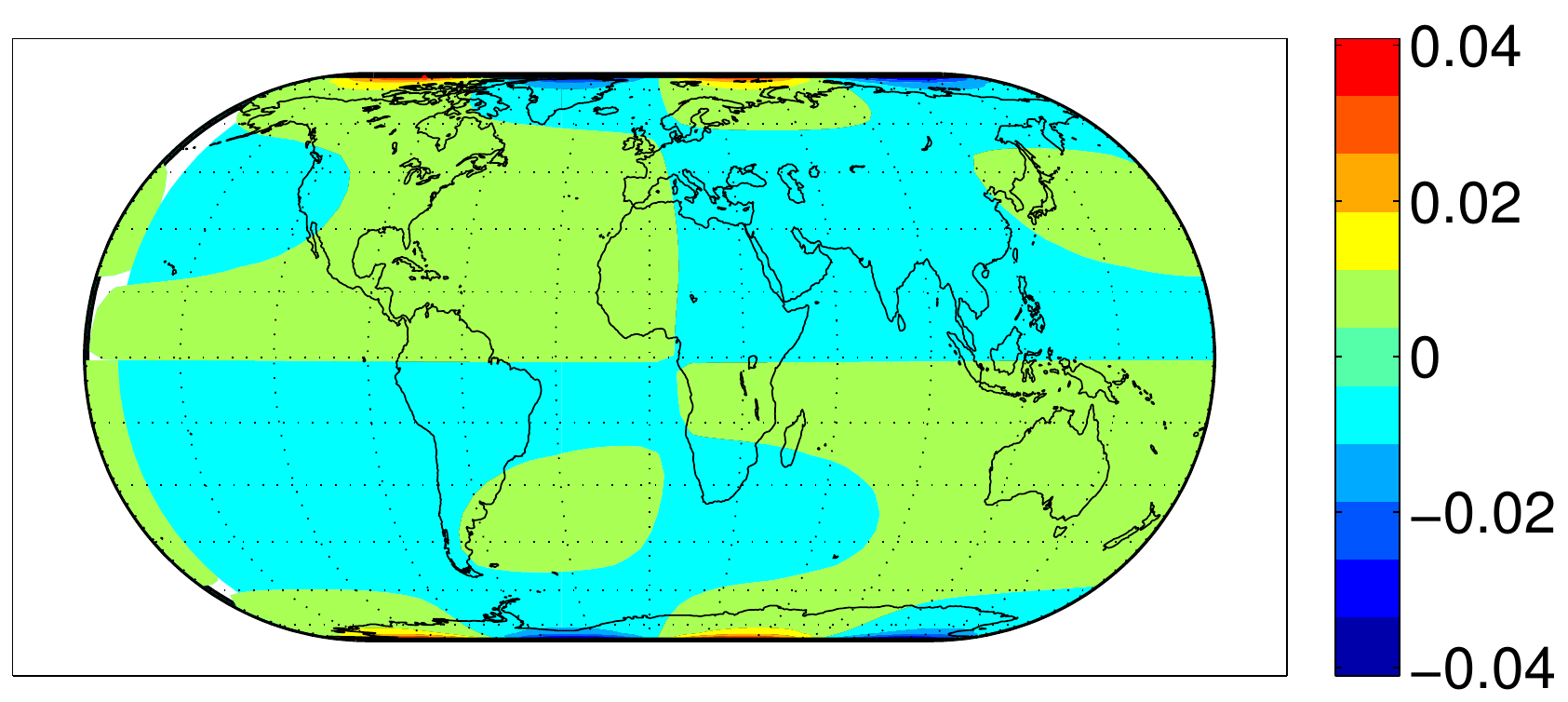}} &
\subfigure[Contributions of model errors: Zonal wind velocity]{\includegraphics[width=0.45\textwidth]{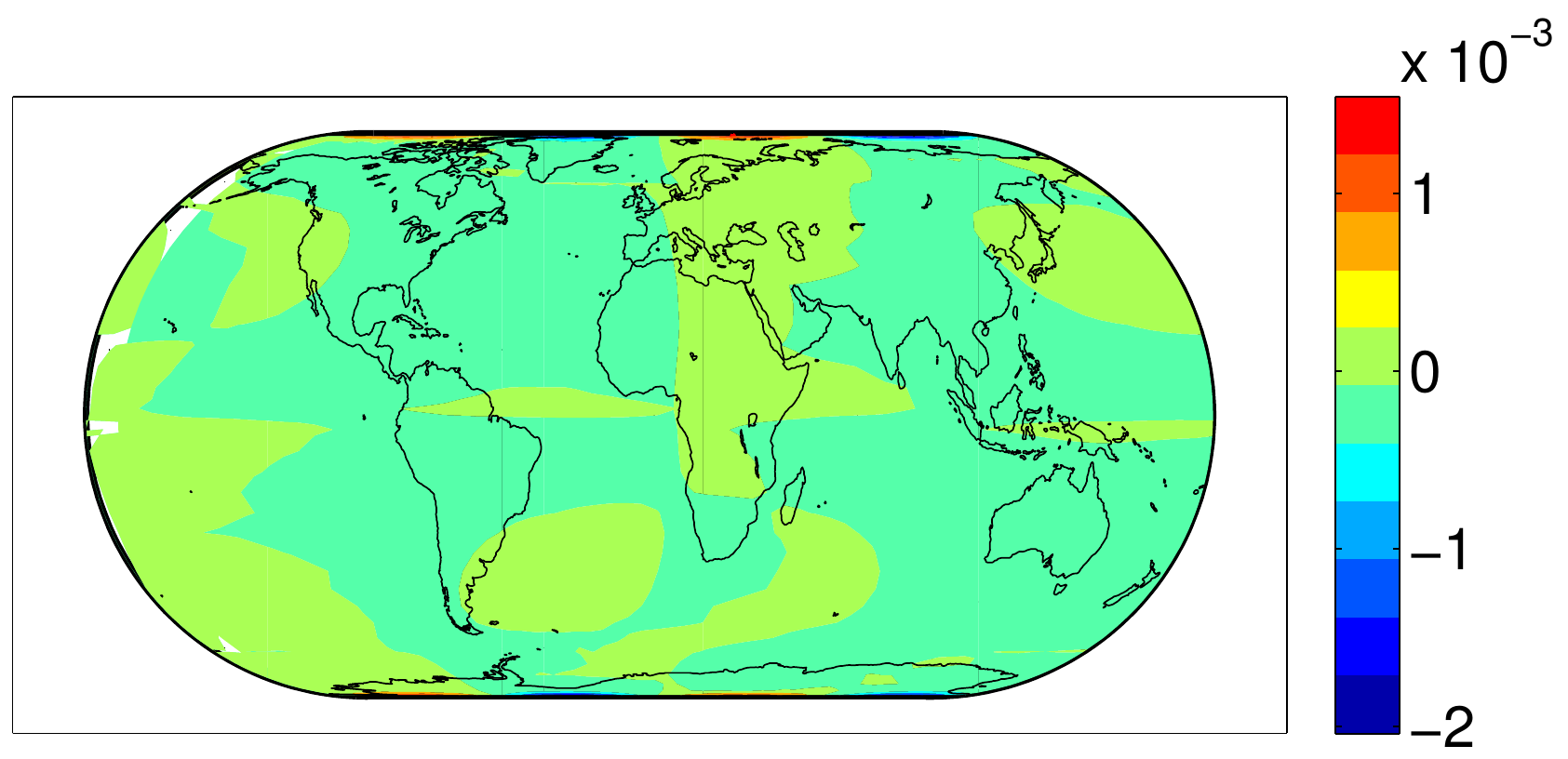}}\\
\subfigure[Model errors: Meridional wind velocity, {$v\,[m/s]$}]{\includegraphics[width=0.45\textwidth]{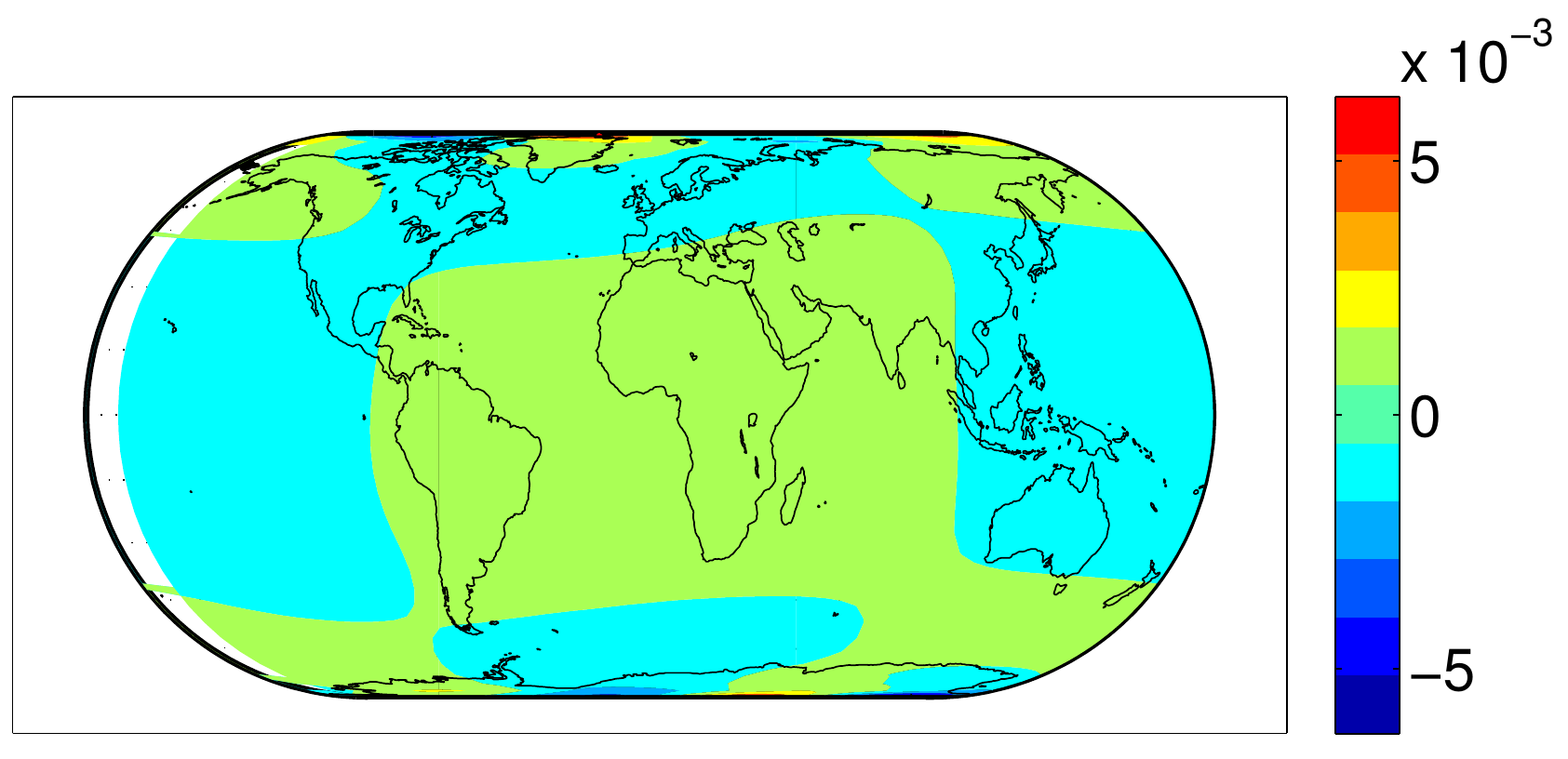}} &
\subfigure[Contributions of model errors: Meridional wind velocity]{\includegraphics[width=0.45\textwidth]{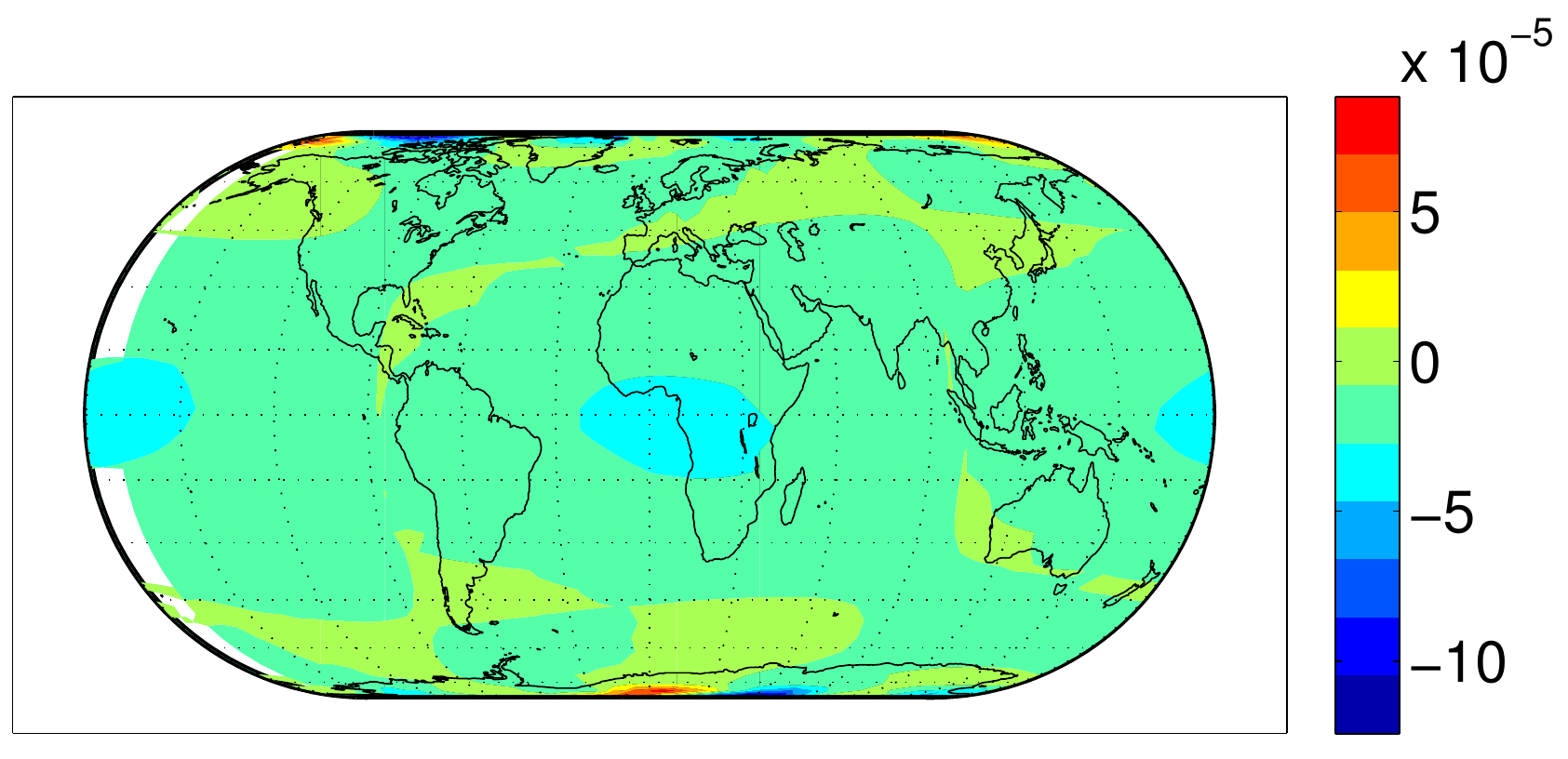}} \\
\subfigure[Data errors: Height, {$h\,[m]$}]{\includegraphics[width=0.45\textwidth]{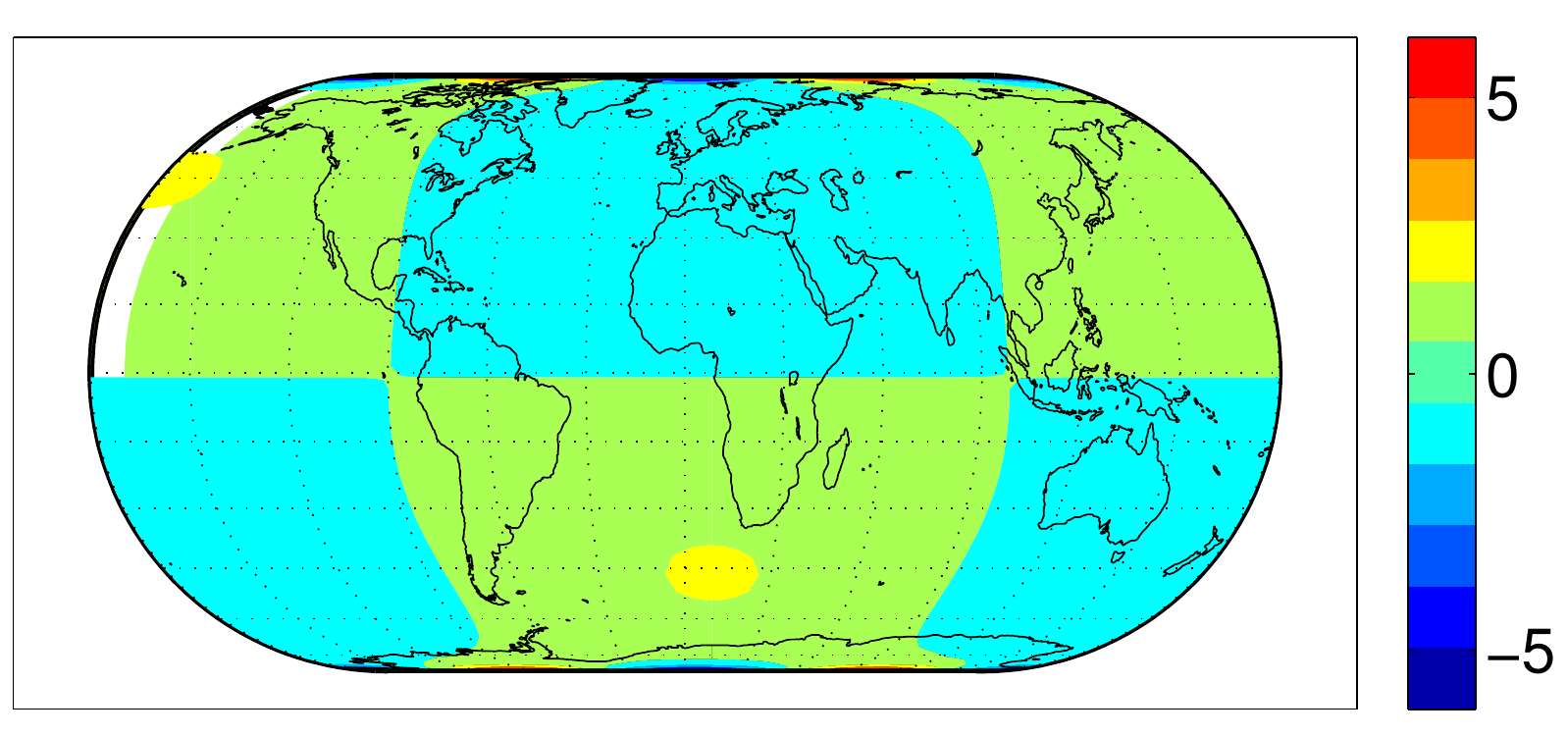}}&
\subfigure[Contributions of model errors: Height]{\includegraphics[width=0.45\textwidth]{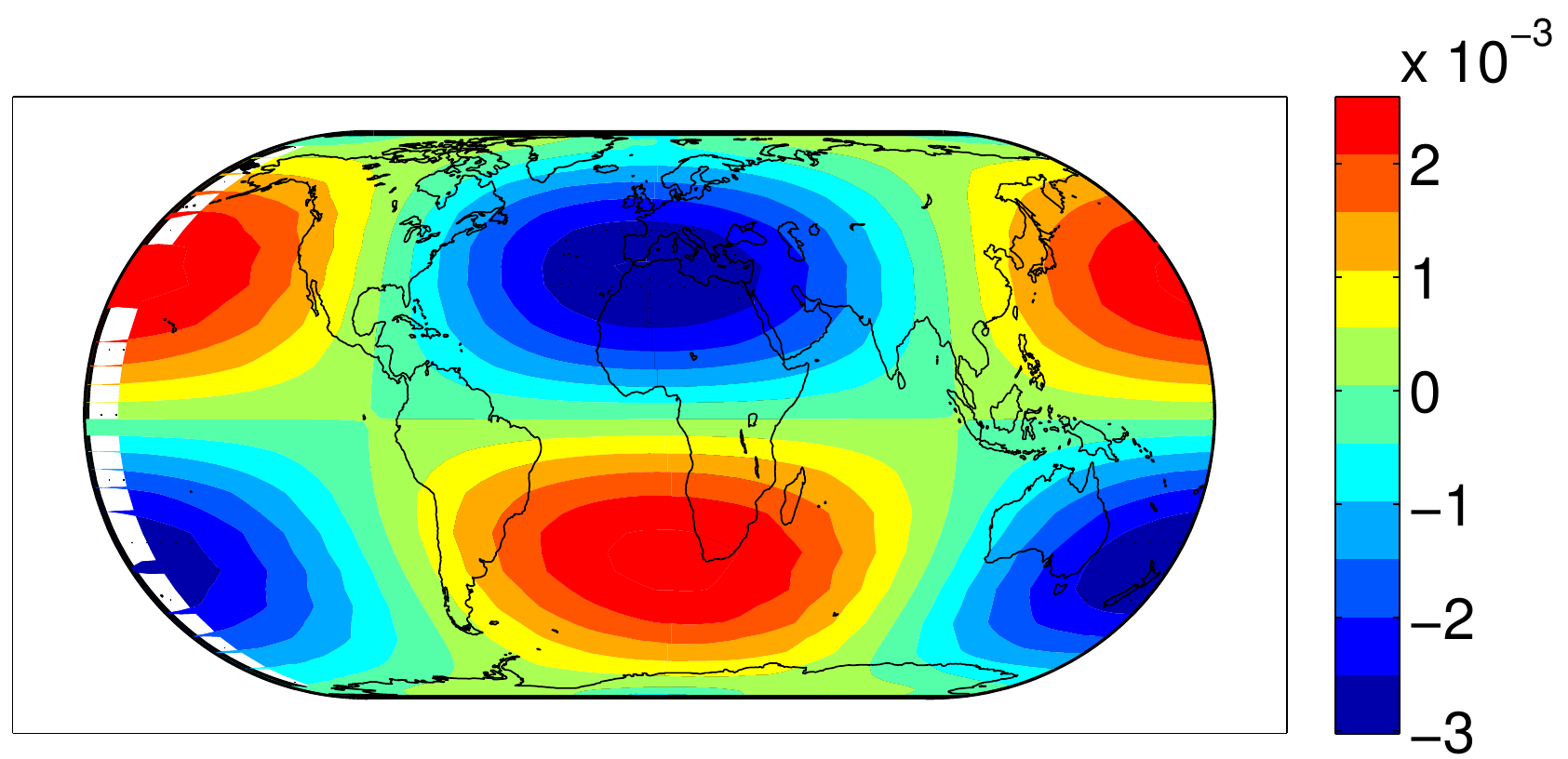}}
\end{tabular}
\end{center}
\caption{The figures on the left show samples of model errors for different variables at different grid points for the shallow water model \eqref{eqn:swe} at time t= 3600s. The figures on the right show the corresponding model error contributions at different grid points to the error functional \eqref{eqn:errfunc} for hourly observations measured over a period of 24 hours. The plot indicates the sum of the model error impact over all the observation instances.}
\label{fig:ModelErrContr:swe}
\end{figure}
%
\iffalse
 \begin{figure}
  \centering
  \subfigure[Zonal wind velocity]{\includegraphics[scale=0.50]{ModelImpact24_U.pdf}}
  \subfigure[Meridional wind velocity]{\includegraphics[scale=0.50]{ModelImpact24_V.pdf}}
  \subfigure[Height]{\includegraphics[scale=0.50]{ModelImpact24_H.pdf}}
  \caption{Model error contributions at different grid points to the error functional (equation \eqref{eqn:errfunc}) for the shallow water model \eqref{eqn:swe} for hourly observations measured over a period of 24 hours. The plot indicates the sum of the model error impact over all the observation instances.
  \textcolor{red}{Can you add plots of the model errors -- the particular realization of the random model  errors used? You can also do a $3 \times 2$ array of panels, with
  the errors on the left and the contributions on the right? Keep the current Figs under comment in case we want to revert.}
 }
 \label{fig:ModelImpact24:swe}
 \end{figure}
\fi

%%%%%%%%%%%%%%%%%%%%%%%%%%%%%%%%%%%%%%%%%%%%%%%%%%%%
\subsection{Validation of a posteriori error estimates in probabilistic setting}\label{sec:statValEsts}
%%%%%%%%%%%%%%%%%%%%%%%%%%%%%%%%%%%%%%%%%%%%%%%%%%%%

The statistics of the a posteriori error estimate \eqref{eqn:discimpactFdvar-noise} are validated by comparing them against the mean and variance of the \qoi for an ensemble of runs (\textit{ensemble mean and variance}). 

The validation procedure is as follows:
\begin{enumerate}
 \item Generate $N_{\rm ens}$ realizations of data errors taken from a Gaussian distribution $\Delta \y_k \sim \mathcal{N}(0,\mathbf{R}_k)$.
This distribution is consistent with \eqref{eq:data-error-statistics} for $\rho_k=0$. 
 \item Generate $N_{\rm ens}$ realizations of model errors. The procedure to obtain different realizations of model error is described in Section \ref{sec:modelerrors}. 
 \item  Solve $N_{\rm ens}$ different 4D-Var optimization problems \eqref{eqn:4dvar-opt-p} to obtain solutions 
  $(\widehat{\x}^{\rm a}_0)_e$, $e=1,\dots, N_{\rm ens}$. Each 4D-Var problem uses a different realization of model error and a different realization of
  the synthetic data (reference values plus the realization of data errors).
 \item Obtain an ensemble of errors in the \qoi $(\Delta \mathcal{E}^{\rm ens})_e=\mathcal{E}((\widehat{\x}^{\rm a}_0)_e)-\mathcal{E}(\x^{\rm a}_0)$, $e=1,\dots, N_{\rm ens}$.
 \item The ensemble mean of error impact is computed by:
 \begin{subequations}\label{eqn:dataImpactDetStats}
 \begin{equation}\label{eqn:dataImpactMeanDet}
  \texttt{E}\lbrack\Delta \mathcal{E}^{\rm ens}\rbrack =
   \frac{1}{N_{\rm ens}}\displaystyle \sum_{i=e}^{N_{\rm ens}}\left( \Delta \mathcal{E}^{\rm ens} \right)_e \,,
 \end{equation}
 and the ensemble variance  of error impact is computed by:
 \begin{equation}
 \label{eqn:dataImpactVarDet}
   \texttt{var}\lbrack\Delta \mathcal{E}^{\rm ens}\rbrack = \frac{1}{N_{\rm ens}-1}\displaystyle \sum_{e=1}^{N_{\rm ens}}\left( \left(\Delta \mathcal{E}^{\rm ens}\right)_e - \texttt{E}\lbrack\Delta \mathcal{E}^{\rm ens}\rbrack \right)^2 \,.
 \end{equation}
 \end{subequations}
\item Compare the variational estimates \eqref{eqn:discimpactFdvar-noise} of means and variances of the impact of data and model errors on the optimal solution
against the ensemble estimates \eqref{eqn:dataImpactDetStats}.

\end{enumerate}
%
\iffalse
The procedure to obtain ensemble mean and variance of the impact of the model errors on the optimal solution is as follows:
\begin{enumerate}
 \item Generate $N_{\rm ens}$ realizations of model errors. The procedure to obtain different realizations of model error is described in Section \ref{sec:modelerrors}. 
 \item With each ensemble member perform the a posteriori error estimate to obtain the impact of the model errors on the optimal solution. Now we have an ensemble of estimates of the impact of model errors on the optimal solution. \newline
 \item The ensemble mean (denoted by $\overline{\Delta \mathcal{E}}_{\rm mod}^{\rm ens}$) of the model error impact can be computed by :
 \begin{subequations}\label{eqn:modelImpactDetStats}
 %
 \begin{equation}\label{eqn:modelImpactMeanDet}
  \overline{\Delta \mathcal{E}}_{\rm mod}^{\rm ens} = \texttt{E}\lbrack\Delta \mathcal{E}_{\rm mod}^{\rm ens}\rbrack = \frac{1}{N_{\rm ens}}\displaystyle \sum_{i=1}^{N_{\rm ens}}\left( \Delta \mathcal{E}_{\rm mod} \right)_i \,.
 \end{equation}
 %
 \item The deterministic variance of the ensemble of the impact of model errors can be computed by :
 %
 \begin{equation}\label{eqn:modelImpactVarDet}
   \texttt{var}\left(\Delta \mathcal{E}_{\rm mod}^{\rm ens}\right) = \frac{1}{N_{\rm ens}}\displaystyle \sum_{i=1}^{N_{\rm ens}}\left( \left(\Delta \mathcal{E}_{\rm mod}\right)_i - \overline{\Delta \mathcal{E}}_{\rm mod}^{\rm ens}  \right)^2 \,.
 \end{equation}
 %
 \end{subequations}

\end{enumerate}
\fi
%
\begin{table}[ht]
 \centering
 \begin{tabular}{|c|c|c|c|c|}
  \hline
  &$\mathtt{E}\lbrack\Delta\mathcal{E}_{\rm obs}\rbrack$&$\texttt{var} \left(\Delta\mathcal{E}_{\rm obs}\right)$ &$\mathtt{E}\lbrack\Delta\mathcal{E}_{\rm mod}\rbrack$ & $\texttt{var} \left(\Delta\mathcal{E}_{\rm mod}\right)$\\ \hline
  Variational estimates \eqref{eqn:discimpactFdvar-noise} &  0.00  &  2.87 & 1.21 & 0.053\\ \hline
  Ensemble estimates \eqref{eqn:dataImpactDetStats} &  0.105 & 2.53 & 1.17 & 0.080 \\ \hline
 \end{tabular}
 \caption{Comparison between ensemble mean and variances of the impact of model and data errors on the 4D-Var optimal solution with the shallow water model \eqref{eqn:swe}.}
 \label{tab:validationMV}
\end{table}

Table \ref{tab:validationMV} shows the results for the shallow water equation. Two sets of experiments are performed. In the first set we consider data errors, but no model errors. In the second we consider model errors, but no data errors. This allows to validate separately the impact of data and the impact of model errors. In each case we use ensembles of $N_{\rm ens} = 15$ members. The variational estimates are fairly close to the ensemble means and variances.
\section{Conclusions and future work}\label{sec:conc}
%%%%%%%%%%%%%%%%%%%%%%%%%%%%%%%%%%%%%%%%%%%%%%%%%%%%%%
Practical inverse problems use imperfect models and noisy data. This work considers variational inverse problems
with time dependent models such as those arising from the discretization of evolutionary PDEs. 
An a posteriori error estimation methodology is developed to quantify the impact of  model and data errors on the inference result.
The approach considers a scalar quantity of interest that depends on the inference result, and which is formalized as an error functional. 
The errors in the quantity of interest due to errors in the model and data are estimated to first order using an algorithm that involves
tangent linear, first, and second order adjoint models. 
We consider generic continuous-time and discrete-time models, and generic cost functionals for the inverse problem. We also derive estimations
in the particular case of 4D-Var data assimilation.

We illustrate the proposed approach using a 4D-Var data assimilation tests with a one dimensional heat equation and with the shallow water model on a sphere.
The error estimates are very close to the actual errors in the quantity of interest due to both the data as well as the model inaccuracies. The statistics (mean and variance) of the estimates are cross-validated using an ensemble of estimates.

The proposed methodology can prove useful in a general context to quantify and reduce uncertainties in a real-time system with feedback. The error estimates can be used to locate faulty sensors. Moreover, the areas of maximum sensitivity  highlighted via the error estimates indicate the locations where greater accuracy in measurements is required (adaptive observations), or where it is beneficial to increase the model resolution (adaptive modeling). 
In future work we plan to apply this methodology to estimate errors in real scenarios using models like the Weather Research and Forecast Model (WRF).

% \bibitem[Lorenz(1996)]{Lorenz_1996}
%  E.~N. Lorenz.
%  \newblock Predictabilty:\penalty0  {A} problem partly solved.
%  \newblock In \emph{Seminar on Predictability, ECMWF}, Shinfield Park,
%  Reading
%    UK., 1996. European Centre for Medium-Range Weather Forecasting.
\clearpage
\section*{Acknowledgements}
This work was supported by AFOSR DDDAS program through the award AFOSR FA9550--12--1--0293--DEF managed by Dr. Frederica Darema.
\bibliographystyle{siam}
\bibliography{fdvar_errors_Journal}

 \begin{appendices}
 %%%%%%%%%%%%%%%%%%%%%%%%%%%%%%%%%%%%%%%%%%%%%%%%%%%%%%%%%%%%%%%%%%%%%%%%%%%%%%%%%%%%%%%%%%%%%%%%%%%
 \section{Derivation of first order optimality conditions for continuous-time models}\label{app:kkt}
 %%%%%%%%%%%%%%%%%%%%%%%%%%%%%%%%%%%%%%%%%%%%%%%%%%%%%%%%%%%%%%%%%%%%%%%%%%%%%%%%%%%%%%%%%%%%%%%%%%%
  The Lagrangian function associated with the cost function in \eqref{eqn:cf} and the constraint in \eqref{eqn:ode} is
\begin{equation}\label{eqn:LagApp}
 \mathcal{L} = \displaystyle\int\limits_{t_0}^{t_F}\,r\left(\x(t),\theta\right)\, \mathrm{d}t + w\left(\x(t_F),\theta\right) - \displaystyle\int\limits_{t_0}^{t_F}\,\lambda^{\rm T}(t)\cdot\left(\x' - f(t,\x,\theta)\right)\,\mathrm{d}t
\end{equation}
Taking variations of \eqref{eqn:LagApp} we obtain:                        %
%                                                 %
\begin{equation*}\label{eqn:var-LApp}                  %
 \begin{split}
 \delta \mathcal{L} & = \displaystyle\int\limits_{t_0}^{t_F}\,\left(\left(r^{\rm T}_{\theta}\left(\x(t),\theta \right) + f_{\theta}^{\rm T} \left(t,\x,\theta \right) \cdot \lambda \right)^{\rm T} \cdot \delta \theta   \right) \mathrm{d}t \\
 & + \displaystyle\int\limits_{t_0}^{t_F}\, \left(\left(r^{\rm T}_{\x}\left(\x(t),\theta \right) + f_{\x}^{\rm T} \left(t,\x,\theta \right) \cdot \lambda \right)^{\rm T} \cdot \delta \x \,\right)\mathrm{d}t \\
 & - \displaystyle\int\limits_{t_0}^{t_F}\, \delta \lambda^{\rm T} \cdot \left( \x' - f(t,\x,\theta) \right) \mathrm{d}t - \displaystyle\int\limits_{t_0}^{t_F}\, \lambda^{\rm T} \cdot \left(\delta\x' \right) \mathrm{d}t\\
 &+ w_\x\left(\x\left( t_F \right), \theta \right) \cdot \delta \x\left(t_F \right) + w_\theta\left(\x\left( t_F \right), \theta \right) \cdot \delta \theta
 \end{split}                                      %
 \end{equation*}          
 Further, by performing integration by parts we obtain:
\begin{equation} \label{eqn:int-parts1App}
\begin{split}
- \displaystyle\int\limits_{t_0}^{t_F}\, \lambda^{\rm T} \cdot \left(\delta\x' \right) \mathrm{d}t = & -\lambda^{\rm T} \left(t \right) \cdot \delta \x \left( t \right) \bigg|_{t_0}^{t_F} + \displaystyle\int\limits_{t_0}^{t_F}\, \left(\lambda' \right)^{\rm T} \cdot \delta \x \left(t \right)\, \mathrm{d}t \\
=& -\lambda^{\rm T} \left(t_F\right)\cdot \delta \x \left(t_F \right) + \lambda^{\rm T} \left(t_0\right)\cdot \delta \x \left(t_0 \right)\\
&+ \displaystyle\int\limits_{t_0}^{t_F}\, \left(\lambda' \right)^{\rm T} \cdot \delta \x \left(t \right)\, \mathrm{d}t \\
=& -\lambda^{\rm T} \left(t_F\right)\cdot \delta \x \left(t_F \right) + \lambda^{\rm T} \left(t_0\right)\cdot \left( \x_\theta(t_0)\cdot \delta \theta\right)\\
& + \displaystyle\int\limits_{t_0}^{t_F}\, \left(\lambda' \right)^{\rm T} \cdot \delta \x \left(t \right)\, \mathrm{d}t
\end{split}
\end{equation}
The KKT conditions or the first order optimality conditions are obtained by setting $\mathcal{L}_\lambda, \mathcal{L}_\theta, \text{ and } \mathcal{L}_\x = \mathbf{0} $.\\
 Setting $\langle\mathcal{L}_\x, \delta \x \rangle = 0$,   $\left(\text{where, }\langle.\rangle\text{ denotes the inner product.}\right)$ gives us the following adjoint ODE :
\begin{subequations}
\label{eqn:kktApp}
\begin{equation}\label{eqn:adjointApp}
\lambda' = -  r_\x^{\rm T}\left(\x(t), \theta \right) - f_\x^{\rm T} \left(t,\x,\theta \right) \cdot \lambda, \qquad \lambda\left( t_F \right) = w_{\x}^{\rm T} \left( \x \left(t_F\right), \theta \right)\,.
\end{equation}
 Setting $\left \langle \mathcal{L}_\lambda, \delta \lambda \right \rangle =0 $, we obtain the constraint ODE
\begin{equation}\label{eqn:odekktApp}
 -\x' + f(t,\x,\theta) = 0, \qquad \x(t_0) = \x_0\,.
\end{equation}
 Setting $\left \langle \mathcal{L}_\theta, \delta \theta \right \rangle =0 $, we obtain the following optimality condition:
\begin{equation}\label{eqn:optApp}
 \xi\left(t_0 \right) + \x_\theta^{\rm T}(t_0)\cdot \lambda (t_0) = 0\,.
\end{equation}
 \end{subequations}
%
% \begin{eqnarray}\label{gradtheta}
%  \left \langle \mathcal{L}_\theta, \delta \theta \right \rangle = \left \langle r_\theta^T \left(\x(t),\theta \right) + f_{\theta}^T (t,\x,\theta)\cdot \lambda ,\, \delta \theta \right \rangle + \left \langle w_{\theta} \left( \x \left(t_F \right), \theta \right), \, \delta \theta \right \rangle \\
%  \nonumber
%  + \left \langle \x_\theta^T(t_0)\cdot \lambda (t_0), \, \delta \theta \right \rangle = 0
% \end{eqnarray}
%
 The value of $\xi\left(t_0\right)$ can be obtained by solving the following ODE:
% % Equation \eqref{gradtheta} gives us the following optimality condition:
%
%  \begin{equation}\label{eqn:kktopt}
%    \displaystyle \int \limits_{t_0}^{t_F}\,\left(r_\theta^T \left(\x(t),\theta \right) + f_{\theta}^T (t,\x,\theta)\cdot \lambda\right)\mathrm{d}t + w_{\theta} \left( \x \left(t_F \right), \theta \right) 
%    + \x_\theta^T(t_0)\cdot \lambda (t_0) = 0
%  \end{equation}
% %
% %Treating the integral as a quadrature variable $\xi$ we get the following expression
% %
\begin{equation}\label{eqn:quad1App}
\begin{split}
 \xi' = -\left(r_\theta^{\rm T} \left(\x(t),\theta \right) + f_{\theta}^{\rm T} (t,\x,\theta)\cdot \lambda\right), \\
 t_F \leq t \leq t_0, \qquad \xi\left(t_F\right) = w^{\rm T}_{\theta}\left(\x\left(t_F \right), \theta \right)\,.
 \end{split}
\end{equation}
The group of equations in \eqref{eqn:kktApp} represent the first order optimality conditions and are same as the group of equations elaborated in \eqref{eqn:kkt}.
 \section{Derivation of super-Lagrange parameters}\label{app:supLag}
%%%%%%%%%%%%%%%%%%%%%%%%%%%%%%%%%%%%%%%%%%%%%%%%%%%%%%%%%%%%%%%%%%%
The Lagrangian associated with the error functional of the form \eqref{eqn:error_functionalCont} and the constraints posed by the first order optimality conditions \eqref{eqn:kkt} is:
\begin{eqnarray}\label{eqn:supLagApp}
 \mathcal{L^E} &= & \,\mathcal{E}\left(\theta^{\rm a}\right) - \displaystyle\int\limits_{t_F}^{t_0}\,\mu^{\rm T} \cdot \left(\lambda'+r_{\x}^{\rm T}  +  f_{\x}^{\rm T} \cdot \lambda \right) \, \mathrm{d}t \\
\nonumber
&& - \mu^{\rm T}\left( t_F \right)  \cdot \left( \lambda\left( t_F \right) - w_{\x}^{\rm T} \left( \x\left(t_F \right), \theta \right) \right)\\
\nonumber
&& - \displaystyle\int\limits_{t_F}^{t_0}\,\zeta^{\rm T} \cdot \left(\xi' + r_\theta^{\rm T} + f_{\theta}^{\rm T} \cdot \lambda \right) \, \mathrm{d}t \\
\nonumber
&&-\zeta^{\rm T} \cdot \left(\xi\left(t_F \right) - w^{\rm T}_{\theta} \left( \x \left(t_F\right), \theta \right)  \right) \\
\nonumber
&&-\zeta^{\rm T} \cdot \left(\xi\left(t_0 \right) + \x_\theta^{\rm T}(t_0)\cdot \lambda (t_0) \right) \\
\nonumber
&&  -\displaystyle\int\limits_{t_0}^{t_F}\,\nu^{\rm T} \cdot \left(\x' - f \right)\,\mathrm{d}t    
 - \nu^{\rm T}\left(t_0 \right) \cdot \left( \x(t_0) - \x_0\right)\,. \nonumber
\end{eqnarray}
Taking the variations of \eqref{eqn:supLagApp} we obtain:
\begin{eqnarray*}
\left \langle \mathcal{L^E_\lambda}, \delta \lambda \right \rangle =  \left \langle \mathcal{E}_{\lambda}, \delta \lambda \right \rangle &=&\displaystyle\int\limits_{t_F}^{t_0}\, \left(-\left( \mu'\right)^{\rm T} \cdot \delta \lambda + \mu^{\rm T} \cdot \left(f_{\x}^{\rm T} \cdot \delta \lambda\right) + \zeta^{\rm T} \cdot \left(f_{\theta}^{\rm T}  \cdot \delta \lambda \right) \right)\mathrm{d}t\\
 &&+ \left.\mu^{\rm T}\cdot \delta \lambda  \right|_{t_F}^{t_0} + \mu^{\rm T}\left( t_F \right)  \cdot \delta\lambda\left( t_F \right) +\zeta^{\rm T} \cdot  \x_\theta^{\rm T}(t_0)\cdot \delta\lambda (t_0) 
\end{eqnarray*}
Imposing the stationary condition $\nabla_{\lambda}\mathcal{L^E} = 0$ leads to the following {\it tangent linear model} (TLM):
\begin{eqnarray}\label{eqn:tlmsupApp}
&& -  \mu' + f_{\x} \cdot \mu   + f_{\theta}  \cdot \zeta = 0, \quad t_0 \le t \le t_F; \\
\nonumber
&& \mu\left( t_0 \right) =   -\x_\theta(t_0)\cdot \zeta\,.
\end{eqnarray}
\begin{eqnarray*}
\left \langle \mathcal{L^E_\x}, \delta \x \right \rangle =  \left \langle \mathcal{E}_{\x}, \delta \x \right \rangle & = &\displaystyle\int\limits_{t_0}^{t_F}\,\left(-\left(\nu'\right)^{\rm T} - \nu^{\rm T} \cdot f_\x -  \left(r_{\x,\x} \cdot \mu \right)^{\rm T}  - \left(\left(f_{\x,\x} \cdot \mu \right)^{\rm T}\cdot \lambda\right)^{\rm T} \right) \cdot \delta \x \, \mathrm{d}t \\
&& + \displaystyle\int\limits_{t_0}^{t_F}\, \left( -\left(  r_{\theta,\x}^{\rm T}  \cdot\zeta  \right)^{\rm T}- \left(\left( f_{\theta,\x}^{\rm T} \cdot \zeta \right)^{\rm T}\cdot \lambda\right)^{\rm T} \right) \cdot \delta \x \, \mathrm{d}t \\
&& + \nu^{\rm T} \cdot \delta \x \big|_{t_0}^{t_F} + \nu^{\rm T}(t_0) \cdot \left(\delta \x \left( t_0 \right) - \delta \x_0\right) \\
&&+ \zeta^{\rm T}  \cdot \left(-w^{\rm T}_{\theta,\x} \left(\x\left(t_F\right), \theta\right) \cdot \delta \x\left(t_F\right) \right)\\
&&-\mu\left(t_F\right)^{\rm T} \cdot \left(w_{\x,\x} \left(\x\left(t_F\right),\theta\right)\right) \cdot \delta \x\left(t_F \right) \,.
\end{eqnarray*}
The stationarity condition $\nabla_{\x}\mathcal{L^E} = 0$ leads to the following {\it second order adjoint} ODE (SOA)
\begin{eqnarray}\label{eqn:soasupApp}
 && \left(\nu'\right) +   f_{\x}^{\rm T} \cdot \nu +  r_{\x,\x} \cdot \mu  + \left(f_{\x,\x} \cdot \mu \right)^{\rm T} \cdot \lambda  \\ 
 \nonumber
&& \qquad +  r_{\theta,\x}^{\rm T} \cdot \zeta +  \left(f_{\theta,\x}^{\rm T} \cdot \zeta \right)^{\rm T} \cdot \lambda   =  0, \qquad t_F \ge t \ge t_0; \\
\nonumber
&& \nu \left(t_F \right) = w_{\theta,\x} \left(\x\left(t_F \right), \theta\right) \cdot \zeta  +  w_{\x,\x} \left(\x\left(t_F \right), \theta\right) \cdot \mu \left(t_F \right)\,.
\end{eqnarray}
We group the remaining terms to obtain
\begin{eqnarray}\label{eqn:ethetaApp} 
  \left \langle \mathcal{E}_{\theta}, \delta \theta \right \rangle &= &\displaystyle\int\limits_{t_0}^{t_F}\,-\left(\mu^{\rm T} \cdot \left(r_{\x,\theta}^{\rm T}+f_{\x,\theta} \cdot \lambda \right)   + \zeta^{\rm T} \cdot \left(r_{\theta,\theta}^{\rm T}  + \left( f_{\theta,\theta} \cdot \lambda\right)^{\rm T}  \right) \right. \\
  \nonumber
  && \qquad \bigl. + \nu^{\rm T} \cdot f_\theta\bigr)\cdot \delta \theta \, \mathrm{d}t \\ \nonumber
  && - \mu^{\rm T} \left(t_F \right)\cdot \left(w_{\theta,\x}\cdot \delta \theta \left(t_F \right)\right) - \zeta^{\rm T}\cdot\left(w_{\theta,\theta}\cdot \delta \theta \left(t_F \right) \right) \\ \nonumber
  &&- \zeta^{\rm T} \cdot\left(\left(\x_{\theta,\theta}\left(t_0\right) \cdot \delta \theta \right)^{\rm T} \cdot \lambda (t_0) \right) - \nu\left(t_0\right)^{\rm T} \cdot \left(\left(\x_0\right)_{\theta} \delta \theta\right)\,.
\end{eqnarray}
Let us take the variation of the first order adjoint equation in equation \eqref{eqn:adjointApp} in the direction of $\delta \theta$, we obtain (we denote $\displaystyle \sigma(t) = \frac{d \lambda(t)}{d \theta} \cdot \delta \theta$)
\begin{eqnarray}\label{eqn:AppInterSOA}
 \sigma'  &=& - f_{\x}^T \cdot \sigma -\left(f_{\x,\x}\cdot \delta \x\right)^T \cdot \lambda - \left(f_{\x,\theta}\cdot\delta \theta\right)^T \cdot \lambda \\
 \nonumber && -r_{\x,\x} \cdot \delta \x -r^{\rm T}_{\theta, \x} \cdot \delta \theta\,, \quad t_F\geq t \geq t_0\,, \\
 \nonumber \sigma\left(t_F\right) &=& \left.w_{\x,\x}\cdot  \delta \x \right\vert_{t_F} + \left.w_{\theta,\x}\cdot\delta \theta\right\vert_{t_F}\,.
\end{eqnarray}
The gradient of the cost function with respect to $\theta$ is given by
\begin{eqnarray} \label{eqn:appGradInter}
 \nabla_{\theta} \J &= &w_{\theta}^T + \x_{\theta}^T \cdot \lambda\left(t_0\right)\\
 \nonumber &&+ \displaystyle \int_{t_0}^{t_F} \,  \left(f_{\theta}\cdot \lambda + r_{\theta}^T \right) \, \mathrm{d}t\,.
\end{eqnarray}
Now taking the derivative of the gradient in the direction of $\delta \theta$,  we have in the direction of $\delta \theta$, the following Hessian-vector product:
\begin{equation}\label{eqn:hvprodApp}
 \begin{split}
    \nabla^2_{\theta, \theta} \J \cdot \delta \theta & = w_{\x, \theta} \left(\x(t_F), \theta \right) \cdot \delta \x \left(t_F\right) + w_{\theta,\theta} \left(\x\left(t_F \right), \theta \right) \cdot \delta \theta \\
    & + \left(\frac{d\x_0}{d\theta}\right)^T\cdot \sigma(t_0) + \left( \frac{d^2\x_0}{d\theta^2} \delta \theta \right)\cdot \lambda(t_0)\\
    &+ \displaystyle \int \limits_{t_0}^{t_F}\, \left(f_\theta^T \cdot \sigma + \left(f_{\theta, \x} \cdot \delta \x\right)^T \cdot \lambda + \left(f_{\theta,\theta}  \cdot \delta \theta \right)^T \cdot \lambda \right) \, \mathrm{d}t \\
    &+\displaystyle \int \limits_{t_0}^{t_F}\,\left(r_{\x,\theta} \cdot \delta \x + r_{\theta,\theta} \cdot \delta \theta \right) \mathrm{d}t\,.
 \end{split}
\end{equation}
Comparing \eqref{eqn:AppInterSOA} and \eqref{eqn:soasupApp} we see the following relationsips
\begin{eqnarray} \label{eqn:SettingDeltas}
 \delta \theta &=& \zeta \\
 \nonumber \delta \x &=& \mu \\
 \nonumber \sigma & = \nu 
\end{eqnarray}
Substituting \eqref{eqn:SettingDeltas} in \eqref{eqn:ethetaApp} we obtain \eqref{jetheta}.
%%%%%%%%%%%%%%%%%%%%%%%%%%%%%%%%%%%%%%%%%%%%%%%%%%%%%%%%%%%%%%%%%%%%%%%%%%%%%%%%%%%%%%%%%%%%%%%%%%%%%%
 \section{Derivation of first order optimality conditions for discrete-time models}\label{app:kktDisc}
%%%%%%%%%%%%%%%%%%%%%%%%%%%%%%%%%%%%%%%%%%%%%%%%%%%%%%%%%%%%%%%%%%%%%%%%%%%%%%%%%%%%%%%%%%%%%%%%%%%%%%
The Lagrangian function associated with the cost function in \eqref{eqn:cfd} and the constraints in \eqref{eqn:model} is 
\begin{eqnarray}
 \label{eqn:LagDiscApp}
 \mathcal{L} &=& \displaystyle \sum_{k=0}^{N-1}\, \left(r_k\left(\x_k,\theta\right) - \lambda^{\rm T}_{k+1}  \cdot \left(\x _{k+1} - \Model_{k,k+1}(\x_k, \theta)\right)\right) + r_N\left(\x_N, \theta\right)\, \\
 \nonumber&& - \lambda_0^{\rm T} \cdot \left(\x_0 - \x_0 \left(\theta\right)\right)
\end{eqnarray}
Taking the variations we get
\begin{eqnarray*} 
 \delta \mathcal{L}&=& \displaystyle \sum_{k=0}^{N}\, \left(r_k\left(\x_k, \theta \right) \right)_{\x_k} \cdot\delta \x_k + \left(r_k\left(\x_k, \theta \right) \right)_{\theta} \cdot\delta \theta  \\ 
 &&- \displaystyle \sum_{k=0}^{N-1}\,\lambda_{k+1}^{\rm T} \cdot \left(\delta \x_{k+1} - \M_{k,k+1} \delta \x_k - \mathfrak{M}_{k,k+1} \delta \theta\right) \\
 && - \displaystyle \sum_{k=0}^{N-1}\,\delta \lambda^{\rm T}_{k+1}  \cdot \left( \x _{k+1} - \Model_{k,k+1}\left(\x_k,\theta\right)\right)\\
 &&- \lambda_0^{\rm T} \cdot \left(\delta \x_0 - \left(\x_{0}\right)_{\theta} \delta \theta\right) - \delta\lambda_0^{\rm T} \cdot \left(\x_0 - \x_0 \left(\theta\right)\right)  \,.
\end{eqnarray*}
Setting the independent variations with respect to $\delta \theta,\, \delta \x_k \,, \text{and } \delta \lambda_k = 0$ we get 
\begin{subequations}\label{eqn:kktDiscApp}
\begin{eqnarray} 
 \lambda_N &=& \left(r_N\left(\x_N, \theta \right)\right)^{\rm T}_{\x_N}\,, \\
 0 &=&\lambda_{k} - \M_{k,k+1}^{\rm T} \lambda_{k+1} - \left(r_k\left(\x_k, \theta\right)\right)^{\rm T}_{\x_k}\,, k =N-1, \dots, 0\,, \label{eqn:adjointDiscApp} \\
 0&=& \x_{k+1} - \Model_{k,k+1}\left(\x_k, \theta\right)\,, k=0,\dots, N-1\,, \label{eqn:forwardDiscApp} \\
 0&=& \sum_{k=0}^N \, \left(r_k\left(\x_k, \theta \right)\right)^{\rm T}_{\theta} + \left(\x_{0}\right)_{\theta}^{\rm T}\cdot \lambda_0 + \displaystyle \sum_{k=0}^{N-1} \mathfrak{M}^{\rm T}_{k,k+1} \lambda_{k+1} \,.\label{eqn:optDiscApp}
\end{eqnarray}
\end{subequations}
The set of equations in \eqref{eqn:kktDiscApp} represent the first order optimality conditions for the inverse problem \eqref{eqn:ipdiscrete} with discrete time models.
\section{Finite dimensional methodology}\label{App:fdm}
%%%%%%%%%%%%%%%%%%%%%%%%%%%%%%%%%%%%%%%%%%%%%%%%%%%%%%%

%%%%%%%%%%%%%%%%%%%%%%%%%%%%%%%%%%%%%%%%%%%%%%%%%%%%%%%
\subsection{The exact inverse problem}
%%%%%%%%%%%%%%%%%%%%%%%%%%%%%%%%%%%%%%%%%%%%%%%%%%%%%%%

Consider the exact (``reference'') inverse problem
\begin{equation}\label{eqn:ipFiniteApp}
 \begin{aligned}
 \theta^{\rm a}  =  &\underset{\theta} {\text{arg\, min}}\,
&& \J\left(\x, \theta\right) \\
  & \text{subject to}
  & & \cbf\left(\x,\theta\right) = 0\,. \\
 \end{aligned}
\end{equation}
The Lagrangian is given by
\begin{equation} \label{eqn:lagFiniteApp}
 \mathcal{L} = \J - \lambda^{\rm T} \cdot \cbf \,.
\end{equation}
The KKT conditions for equation \eqref{eqn:lagFiniteApp} is given by
\begin{subequations} 
\label{eqn:kktFiniteApp}
 \begin{eqnarray}
\label{kkt:fwdFiniteApp}
\textnormal{forward model:}~~  0 &=& \cbf\left(\x, \theta\right) \,, \\
\label{kkt:adjFiniteApp}
\textnormal{adjoint model:}~~ 0 &=& \J_{\x} - \lambda^{\rm T} \cdot \cbf_{\x}\,, \\
\label{kkt:optFiniteApp}
\textnormal{optimality:} ~~0 &=& \J_{\theta} - \lambda^{\rm T} \cdot \cbf_{\theta}\,.
\end{eqnarray}
\end{subequations}
It should be noted that the gradient of $\mathcal{L}$ with respect to $\theta$ is given by
\begin{equation}
 \nabla_{\theta} \mathcal{L} = \J_{\theta} - \lambda^{\rm T} \cdot \cbf_{\theta}\,.
\end{equation}
We seek to minimize the function $\mathcal{E}\left(\theta\right)$ with the KKT conditions in equation \eqref{eqn:kktFiniteApp} as the constraints. Hence we consider the following super-Lagrangian
\begin{eqnarray} \label{eqn:supLagFiniteApp}
\mathcal{L^E} = \mathcal{E} - \nu^{\rm T} \cdot \cbf  - \left(\J_{\x} - \lambda^{\rm T} \cdot \cbf_{\x} \right) \mu - \left(\J_{\theta} - \lambda^{\rm T} \cdot \cbf_{\theta} \right) \zeta \,.
\end{eqnarray}
Taking the derivative of $\mathcal{L^E}$ with respect to $\x, \lambda, \textnormal{ and } \theta$ we obtain the following:
\begin{subequations}
\begin{align}
 \label{kkt:soaFiniteApp}
 \left(\mathcal{L^E}\right)_{\x} &= \mathcal{E}_{\x} - \nu^{\rm T} \cdot \cbf_{\x} - \mu^{\rm T}\left(\J_{\x,\x} - \lambda^{\rm T} \cdot \cbf_{\x,\x} \right) \\
 &\nonumber -\zeta^{\rm T} \cdot \left(\J_{\theta, \x} - \lambda^{\rm T} \cdot \cbf_{\theta, \x} \right)\,, \\
 \label{kkt:TLMFiniteApp}
 \left(\mathcal{L^E}\right)_{\lambda} &= \mathcal{E}_{\lambda} + \mu^{\rm T} \cdot \cbf_{\x}^{\rm T} + \zeta^{\rm T} \cdot \cbf_{\theta}^{\rm T} \,,\\
 \label{kkt:OPTFiniteApp}
 \left(\mathcal{L^E}\right)_{\theta} & = \mathcal{E}_{\theta} - \nu^{\rm T} \cdot \cbf_{\theta} - \mu^{\rm T} \cdot \left( \J_{\x, \theta} - \lambda^{\rm T} \cdot \cbf_{\x, \theta} \right)  \\
 \nonumber & - \zeta^{\rm T} \cdot \left(\J_{\theta, \theta} - \lambda^{\rm T}  \cdot \cbf_{\theta, \theta} \right)\,.
\end{align}
\end{subequations}
Setting $\left(\mathcal{L^E}\right)_{\lambda} = 0$, we obtain ($\mathcal{E}_{\lambda} = 0$)
\begin{equation}\label{eqn:muZetaFiniteApp}
 \mu^{\rm T} = -\zeta^{\rm T} \cdot \left(\cbf_{\theta}^{\rm T} \cbf_{\x}^{-\rm T}\right)\,.
\end{equation}
From equations \eqref{kkt:soaFiniteApp} and \eqref{eqn:muZetaFiniteApp} and setting $\left(\mathcal{L^E}\right)_{\x}  = 0$, we obtain ($\mathcal{E}_{\x} = 0$)
\begin{eqnarray}\label{eqn:nuCThetaFiniteApp}
 \nu^{\rm T} \cdot \cbf_{\theta}&=& \zeta^{\rm T} \left(\cbf_{\theta} \cbf_{\x}^{-1} \left(\J_{\x,\x} - \lambda^{\rm T} \cdot \cbf_{\x,\x}\right)  - \left( \J_{\theta,\x}   -\lambda^{\rm T} \cdot \cbf_{\theta,\x}\right)\right)\cbf_{\x}^{-1} \cbf_{\theta}.
 \end{eqnarray}
Substituting equations \eqref{eqn:nuCThetaFiniteApp} and \eqref{eqn:muZetaFiniteApp} in \eqref{kkt:OPTFiniteApp} we obtain
\begin{eqnarray}\label{eqn:ethetaFiniteApp}
 \mathcal{E}_{\theta} &=& \zeta^{\rm T} \left(\cbf^{\rm T}_{\theta} \cbf_{\x}^{-\rm T} \left(\J_{\x,\x} - \lambda^{\rm T} \cdot \cbf_{\x,\x}\right)  - \left( \J_{\theta,\x}   -\lambda^{\rm T} \cdot \cbf_{\theta,\x}\right)\right)\cbf_{\x}^{-1} \cbf_{\theta} \\
 \nonumber&&- \zeta^{\rm T} \left( \cbf^{\rm T}_{\theta} \cbf_{\x}^{-\rm T} \left( \J_{\x, \theta} - \lambda^{\rm T} \cdot \cbf_{\x, \theta} \right) \right) + \zeta^{\rm T} \cdot \left(\J_{\theta, \theta} - \lambda^{\rm T}  \cdot \cbf_{\theta, \theta} \right)\,.
\end{eqnarray}
Consider the Lagrangian of the reduced cost function 
\begin{eqnarray}
\label{eqn:lagRedFiniteApp}
 \ell(\theta) =  \J\left(\x\left(\theta\right), \theta \right) - \lambda(\theta)^{\rm T} \cdot \cbf\left(\x\left(\theta\right), \theta \right)\,.
\end{eqnarray}
The reduced gradient reads
\begin{eqnarray}
\label{eqn:lagRedGradient}
 \ell_\theta^T &=&  \J_\theta^T - \cbf_\theta^T\, \lambda + \x_\theta^T\, \left( \J_\x^T -\cbf_\x^T\, \lambda \right) - \left( \lambda_\theta^{\rm T} + \x_\theta^T\, \lambda_\x^T \right)\, \cbf\,.
\end{eqnarray}
The reduced Hessian reads
\begin{eqnarray}
\label{eqn:lagRedHessianComplete}
 \ell_{\theta,\theta} &=&  \J_{\theta,\theta} - \lambda^{\rm T} \, \cbf_{\theta,\theta} + \left( \J_{\theta,\x} - \lambda^{\rm T} \, \cbf_{\theta,\x} \right)\, \x_\theta -  \cbf_\theta^T\, \left( \lambda_\theta +  \lambda_\x\, \x_\theta \right)\, \\
 \nonumber
 && +  \x_\theta^T\, \left( \J_{\x,\theta} - \lambda^{\rm T} \, \cbf_{\x,\theta} \right) \, + \x_\theta^T\,  \left( \J_{\x,\x} - \lambda^{\rm T} \, \cbf_{\x,\x} \right) \, \x_\theta\\
 \nonumber
&& - \x_\theta^T\, \cbf_\x^T\,\left( \lambda_\theta +  \lambda_\x\, \x_\theta \right) \\
 \nonumber
 &&  - \left( \lambda_\theta^{\rm T} + \x_\theta^T\, \lambda_\x^T \right)\, \left(\cbf_\theta + \cbf_\x\, \x_\theta \right)
        - \frac{d}{d\theta}\left( \lambda_\theta^{\rm T} + \x_\theta^T\, \lambda_\x^T \right)\, \cbf \\
\nonumber        
         &=&  \J_{\theta,\theta} - \lambda^{\rm T} \, \cbf_{\theta,\theta} + \left( \J_{\theta,\x} - \lambda^{\rm T} \, \cbf_{\theta,\x} \right)\, \x_\theta \, \\
 \nonumber
 && +  \x_\theta^T\, \left( \J_{\x,\theta} - \lambda^{\rm T} \, \cbf_{\x,\theta} \right) \, + \x_\theta^T\,  \left( \J_{\x,\x} - \lambda^{\rm T} \, \cbf_{\x,\x} \right) \, \x_\theta\\
 \nonumber
&& - \left( \cbf_\theta^T + \x_\theta^T\, \cbf_\x^T\right)\,\left( \lambda_\theta +  \lambda_\x\, \x_\theta \right)\\
 \nonumber
 &&  - \left( \lambda_\theta^{\rm T} + \x_\theta^T\, \lambda_\x^T \right)\, \left(\cbf_\theta + \cbf_\x\, \x_\theta \right)
        - \frac{d}{d\theta}\left( \lambda_\theta^{\rm T} + \x_\theta^T\, \lambda_\x^T \right)\, \cbf
 \nonumber\,.
\end{eqnarray}
When the optimality conditions are satisfied we have that
\[
\cbf = 0\,, \quad \cbf_\theta + \cbf_\x\, \x_\theta = 0 ~~ \Rightarrow ~~ \x_\theta = - \cbf_\x^{-1}\, \cbf_\theta.
\]
Consequently the reduced Hessian \eqref{eqn:lagRedHessianComplete} evaluated at the optimal solution reads
\begin{eqnarray}
\label{eqn:lagRedHessian}
 \ell_{\theta,\theta}
&=&  \J_{\theta,\theta} - \lambda^{\rm T} \, \cbf_{\theta,\theta} + \left( \J_{\theta,\x} - \lambda^{\rm T} \, \cbf_{\theta,\x} \right)\, \x_\theta \, \\
 \nonumber
 && +  \x_\theta^T\, \left( \J_{\x,\theta} - \lambda^{\rm T} \, \cbf_{\x,\theta} \right) \, + \x_\theta^T\,  \left( \J_{\x,\x} - \lambda^{\rm T} \, \cbf_{\x,\x} \right) \, \x_\theta  \\
 \nonumber
 &=& \J_{\theta,\theta} - \lambda^{\rm T} \, \cbf_{\theta,\theta} - \left( \J_{\theta,\x} - \lambda^{\rm T} \, \cbf_{\theta,\x} \right)\, \cbf_\x^{-1}\, \cbf_\theta \, \\
 \nonumber
 && -  \cbf_\theta^T\, \cbf_\x^{-T}\, \left( \J_{\x,\theta} - \lambda^{\rm T} \, \cbf_{\x,\theta} \right) \, + \cbf_\theta^T\, \cbf_\x^{-T}\,  \left( \J_{\x,\x} - \lambda^{\rm T} \, \cbf_{\x,\x} \right) \, \cbf_\x^{-1}\, \cbf_\theta.
\end{eqnarray}

Equation \eqref{eqn:ethetaFiniteApp} can be written as the ``Hessian linear system''
\[
\ell_{\theta,\theta}\cdot \zeta = \mathcal{E}_\theta^T\,.
\]
Equation \eqref{eqn:muZetaFiniteApp} is the tangent linear model
\[
 \mu = -  \cbf_{\x}^{-\rm 1}\, \cbf_{\theta}\, \cdot \zeta \quad \Leftrightarrow \quad  \cbf_{\x} \cdot \mu = - \cbf_{\theta}\, \cdot \zeta.
\]
Finally from \eqref{kkt:soaFiniteApp} we have the second order adjoint model
\[
 \cbf_{\x}^T \, \nu = - \left(\J_{\x,\x} - \lambda^{\rm T} \cdot \cbf_{\x,\x} \right)^T\, \mu
 - \left(\J_{\theta, \x} - \lambda^{\rm T} \cdot \cbf_{\theta, \x} \right)^T\, \zeta
\]
or
\[
 \nu =  \cbf_{\x}^{-T}\, \left(\J_{\x,\x} - \lambda^{\rm T} \cdot \cbf_{\x,\x} \right)^T\, \cbf_{\x}^{-\rm 1}\, \cbf_{\theta}\, \cdot \zeta
 - \cbf_{\x}^{-T}\,\left(\J_{\theta, \x} - \lambda^{\rm T} \cdot \cbf_{\theta, \x} \right)^T\, \zeta. 
\]
Consider now the perturbed inverse problem
\begin{subequations} 
\label{eqn:kktFiniteApp_pg}
 \begin{eqnarray}
\label{kkt:fwdFiniteApp_pg}
\textnormal{perturbed forward model:}~~  \Delta\mathcal{F} &=& \cbf\left(\x, \theta\right) \,, \\
\label{kkt:adjFiniteApp_pg}
\textnormal{perturbed adjoint model:}~~ \Delta\mathcal{A} &=& \J_{\x} - \lambda^{\rm T} \cdot \cbf_{\x}\,, \\
\label{kkt:optFiniteApp_pg}
\textnormal{perturbed optimality:} ~~\Delta\mathcal{O} &=& \J_{\theta} - \lambda^{\rm T} \cdot \cbf_{\theta}\,.
\end{eqnarray}
\end{subequations}
where $\Delta\mathcal{F}$, $\Delta\mathcal{A}$, and $\Delta\mathcal{O}$ are the residuals in the forward, adjoint, and optimality conditions, respectively.
From \eqref{eqn:supLagFiniteApp} we have the following error estimate:
\begin{eqnarray*} 
\Delta \mathcal{E} &\approx& \nu^{\rm T} \cdot \Delta\mathcal{F} + \mu^T \cdot \Delta\mathcal{A} + \zeta^T\cdot \Delta\mathcal{O} \\
&=& \zeta^T \, \left( \cbf_{\theta}^T\, \cbf_{\x}^{-\rm T}\, \left(\J_{\x,\x} - \lambda^{\rm T} \cdot \cbf_{\x,\x} \right)  - \left(\J_{\theta, \x} - \lambda^{\rm T} \cdot \cbf_{\theta, \x} \right) \right)\, \cbf_{\x}^{-1} \, \Delta\mathcal{F} \\
&& + \zeta^T \, \left( -\cbf_{\theta}^T \,\cbf_{\x}^{-\rm T}\, \Delta\mathcal{A} + \Delta\mathcal{O} \right)
\end{eqnarray*}

%%%%%%%%%%%%%%%%%%%%%%%%%%%%%%%%%%%%%%%%%%%%%%
\subsection{Perturbed finite dimensional inverse problem}
%%%%%%%%%%%%%%%%%%%%%%%%%%%%%%%%%%%%%%%%%%%%%%
Consider the perturbed inverse problem
\begin{equation}\label{eqn:ipFinitePertApp}
 \begin{aligned}
 \widehat{\theta}^{\rm a}  =  &\underset{\theta} {\text{arg\, min}}\,
&& \J\left(\x, \theta\right) + \Delta \J\left(\x, \theta\right) \\
  & \text{subject to}
  & & \cbf\left(\x,\theta\right) + \Delta \cbf\left(\x, \theta\right) = 0\,. \\
 \end{aligned}
\end{equation}
The perturbed Lagrangian is given by
\begin{equation} \label{eqn:lagFinitePertApp}
 \widehat{\mathcal{L}} = \J + \Delta \J - \lambda^{\rm T} \cdot \left(\cbf + \Delta \cbf \right)  \,.
\end{equation}
For convenience we use the short notation
\[
\cbf := \cbf\left(\x, \theta \right) \,, \quad \widehat{\cbf} := \cbf\left(\widehat{\x}, \widehat{\theta} \right), \quad
\widehat{\x} = \x + \Delta\x, \quad \widehat{\theta} = \theta + \Delta\theta, \quad \widehat{\lambda} = \lambda + \Delta\lambda.
\]
The KKT conditions for equation \eqref{eqn:lagFinitePertApp} are
\begin{subequations} 
\label{eqn:kktFinitePertApp}
 \begin{eqnarray}
\label{kkt:fwdFinitePertApp}
\textnormal{forward model:}~~ & 0& = \widehat{\cbf} + \Delta \widehat{\cbf} \,, \\
\label{kkt:adjFinitePertApp}
\textnormal{adjoint model:}~~ &0& = \widehat{\J}_{\x} + \Delta \widehat{\J}_{\x} - \widehat{\lambda}^{\rm T} \cdot \left(\widehat{\cbf}_{\x} + \Delta \widehat{\cbf}_{\x} \right)\,, \\
\label{kkt:optFinitePertApp}
\textnormal{optimality:} ~~&0& =  \widehat{\J}_{\theta} + \Delta \widehat{\J}_{\theta} - \widehat{\lambda}^{\rm T} \cdot \left(\widehat{\cbf}_{\theta} + \Delta \widehat{\cbf}_{\theta} \right)\,.
\end{eqnarray}
\end{subequations}
%
%We consider the Lagrangian of the function $\mathcal{E}(\widehat{\theta}^{\rm a})$ with the KKT conditions in equation \eqref{eqn:kktFinitePertApp} as the constraints
%%
%\begin{eqnarray} \label{eqn:supLagFinitePertApp}
%\mathcal{\widehat{L}_E} &=& \widehat{\mathcal{E}} - \nu^{\rm T} \cdot \left(\cbf + \Delta \cbf\right)  - \left(\J_{\x} + \Delta \J_\x - \lambda^{\rm T} \cdot \left(\cbf_{\x} + \Delta \cbf_{\x} \right) \right) \mu \\
%\nonumber && - \left(\J_{\theta} + \Delta \J_\theta - \lambda^{\rm T} \cdot \left(\cbf_{\theta} + \Delta \cbf_{\theta} \right) \right) \zeta \,.
%\end{eqnarray}
%%
%Subtracting \eqref{eqn:supLagFiniteApp} from \eqref{eqn:supLagFinitePertApp} we obtain
%%
%\begin{eqnarray}
% \mathcal{\widehat{L}_E} - \mathcal{L^E} = \Delta \mathcal{E}  - \nu^{\rm T} \cdot \Delta \cbf -\left(\Delta \J_{\x} -\lambda^{\rm T} \cdot \Delta \cbf_{\x} \right) \mu -\left(\Delta \J_{\theta} -\lambda^{\rm T} \cdot \Delta \cbf_{\theta} \right) \zeta\,.
%\end{eqnarray}
%%
%Hence we obtain the estimate
%%
%\begin{eqnarray}
% \Delta \mathcal{E} =  \nu^{\rm T} \cdot \Delta \cbf +\left(\Delta \J_{\x} -\widehat{\lambda}^{\rm T} \cdot \Delta \cbf_{\x} \right) \mu +\left(\Delta \J_{\theta} -\lambda^{\rm T} \cdot \Delta \cbf_{\theta} \right) \zeta\,.
%\end{eqnarray}
%
Linearize \eqref{eqn:kktFinitePertApp} around the ideal optimal solution \eqref{eqn:kktFiniteApp}:
\begin{subequations} 
\label{eqn:kktFinitePertAppLin}
 \begin{eqnarray}
\label{kkt:fwdFinitePertAppLin}
0& =&  \cbf + \Delta \cbf + \left( \cbf + \Delta \cbf\right)_\x\, \Delta\x + \left( \cbf + \Delta \cbf\right)_\theta\, \Delta\theta  \,, \\
%\nonumber
%& \approx&  \cbf + \Delta \cbf + \cbf_\x\, \Delta\x + \cbf_\theta\, \Delta\theta \\
%
\label{kkt:adjFinitePertAppLin}
0& =&  (\J + \Delta\J)_{\x}  + (\J + \Delta\J)_{\x,\x}\, \Delta\x + (\J + \Delta\J)_{\x,\theta}\, \Delta\theta \\
\nonumber
&& - \Delta\lambda^{\rm T} \cdot \cbf_\x  - \lambda^{\rm T} \cdot \left( \cbf + \Delta\cbf\right)_{\x} \\
\nonumber
&& - \lambda^{\rm T} \cdot \left( \cbf + \Delta\cbf\right)_{\x,\x}\, \Delta\x   - \lambda^{\rm T} \cdot \left( \cbf + \Delta\cbf\right)_{\x,\theta}\, \Delta\theta \,, \\
\label{kkt:optFinitePertAppLin}
0& =&  (\J + \Delta\J)_{\theta}  + (\J + \Delta\J)_{\theta,\x}\, \Delta\x + (\J + \Delta\J)_{\theta,\theta}\, \Delta\theta \\
\nonumber
&& - \Delta\lambda^{\rm T} \cdot \cbf_\theta  - \lambda^{\rm T} \cdot \left( \cbf + \Delta\cbf\right)_{\theta} 
 - \lambda^{\rm T} \cdot \left( \cbf + \Delta\cbf\right)_{\theta,\x}\, \Delta\x  \\
\nonumber
&&  - \lambda^{\rm T} \cdot \left( \cbf + \Delta\cbf\right)_{\theta,\theta}\, \Delta\theta.
\end{eqnarray}
\end{subequations}
\par\noindent {\bf Assumption:} $\Delta\cbf$, $\Delta\J$, their first derivatives $\Delta\cbf_\x$, $\Delta\J_\x$, $\Delta\cbf_\theta$, $\Delta\J_\theta$, and their second order derivatives $\Delta\cbf_{\x,\x}$,  $\Delta\cbf_{\x,\theta}$, \dots, $\Delta\J_{\theta,\theta}$ are small (their norms are bounded by $\varepsilon$). 

Then ignoring products of small terms in \eqref{eqn:kktFinitePertAppLin} leads to
\begin{subequations} 
\label{eqn:kktFinitePertAppSimpleA}
 \begin{eqnarray}
\label{kkt:fwdFinitePertAppSimpleA}
0& =&  \cbf + \Delta \cbf + \cbf_\x\, \Delta\x + \cbf_\theta\, \Delta\theta \\
\label{kkt:adjFinitePertAppSimpleA}
0& =&  (\J + \Delta\J)_{\x}  + \J _{\x,\x}\, \Delta\x + \J_{\x,\theta}\, \Delta\theta \\
\nonumber
&& - \Delta\lambda^{\rm T} \cdot \cbf_\x  - \lambda^{\rm T} \cdot \left( \cbf + \Delta\cbf\right)_{\x} 
 - \lambda^{\rm T} \cdot \cbf_{\x,\x}\, \Delta\x   - \lambda^{\rm T} \cdot \cbf_{\x,\theta}\, \Delta\theta \,, \\
\label{kkt:optFinitePertAppSimpleA}
0& =&  (\J + \Delta\J)_{\theta}  + \J_{\theta,\x}\, \Delta\x + \J_{\theta,\theta}\, \Delta\theta \\
\nonumber
&& - \Delta\lambda^{\rm T} \cdot \cbf_\theta  - \lambda^{\rm T} \cdot \left( \cbf + \Delta\cbf\right)_{\theta} 
 - \lambda^{\rm T} \cdot \cbf_{\theta,\x}\, \Delta\x   - \lambda^{\rm T} \cdot \cbf_{\theta,\theta}\, \Delta\theta.
\end{eqnarray}
\end{subequations}
Using the ideal KKT conditions \eqref{eqn:kktFiniteApp} and after rearranging terms the above expressions \eqref{eqn:kktFinitePertAppSimpleA} become
\begin{subequations} 
\label{eqn:kktFinitePertAppSimple}
 \begin{eqnarray}
\label{kkt:fwdFinitePertAppSimple}
0& =&  \Delta \cbf + \cbf_\x\, \Delta\x + \cbf_\theta\, \Delta\theta \\
\label{kkt:adjFinitePertAppSimple}
0& =&  \Delta\J_{\x}^T -  \cbf_\x^T \cdot \Delta\lambda - \Delta\cbf_{\x}^T \cdot \lambda    \\
\nonumber
&& + \left( \J _{\x,\x} - \lambda^{\rm T} \cdot \cbf_{\x,\x} \right)\, \Delta\x  + \left(\J_{\x,\theta} - \lambda^{\rm T} \cdot \cbf_{\x,\theta}\right)\, \Delta\theta \,, \\
\label{kkt:optFinitePertAppSimple}
0& =&  \Delta\J_{\theta}^T -  \cbf_\theta^T \cdot \Delta\lambda - \Delta\cbf_{\theta}^T \cdot \lambda    \\
\nonumber
&& + \left( \J _{\theta,\x} - \lambda^{\rm T} \cdot \cbf_{\theta,\x} \right)\, \Delta\x  + \left(\J_{\theta,\theta} - \lambda^{\rm T} \cdot \cbf_{\theta,\theta}\right)\, \Delta\theta \,.
\end{eqnarray}
\end{subequations}
From \eqref{kkt:fwdFinitePertAppSimple}
\[
\Delta\x  = - \cbf_\x^{-1}\,\Delta \cbf  - \cbf_\x^{-1}\,\cbf_\theta\, \Delta\theta.
\]
From \eqref{kkt:adjFinitePertAppSimple}
\begin{eqnarray*}
 \Delta\lambda & =&  \cbf_\x^{-T} \cdot \Delta\J_{\x}^T  -  \cbf_\x^{-T} \cdot\Delta\cbf_{\x}^T \cdot \lambda  -  \cbf_\x^{-T} \cdot\left( \J _{\x,\x} - \lambda^{\rm T} \cdot \cbf_{\x,\x} \right)\,  \cbf_\x^{-1}\,\Delta \cbf   \\
&& -  \cbf_\x^{-T} \cdot\left( \J _{\x,\x} - \lambda^{\rm T} \cdot \cbf_{\x,\x} \right)\, \cbf_\x^{-1}\,\cbf_\theta\, \Delta\theta  +  \cbf_\x^{-T} \cdot\left(\J_{\x,\theta} - \lambda^{\rm T} \cdot \cbf_{\x,\theta}\right)\, \Delta\theta
\end{eqnarray*}
From \eqref{kkt:optFinitePertAppSimple}
\begin{eqnarray*}
0& =&  \Delta\J_{\theta}^T - \Delta\cbf_{\theta}^T \cdot \lambda  
 -  \cbf_\theta^T\, \cbf_\x^{-T} \cdot \Delta\J_{\x}^T  +  \cbf_\theta^T\,  \cbf_\x^{-T} \cdot\Delta\cbf_{\x}^T \cdot \lambda  \\
&& +  \cbf_\theta^T\, \cbf_\x^{-T} \cdot\left( \J _{\x,\x} - \lambda^{\rm T} \cdot \cbf_{\x,\x} \right)\,  \cbf_\x^{-1}\,\Delta \cbf  
      - \left( \J _{\theta,\x} - \lambda^{\rm T} \cdot \cbf_{\theta,\x} \right)\,  \cbf_\x^{-1}\,\Delta \cbf \\
&& +  \cbf_\theta^T\,  \cbf_\x^{-T} \cdot\left( \J _{\x,\x} - \lambda^{\rm T} \cdot \cbf_{\x,\x} \right)\, \cbf_\x^{-1}\,\cbf_\theta\, \Delta\theta  -  \cbf_\theta^T\,  \cbf_\x^{-T} \cdot\left(\J_{\x,\theta} - \lambda^{\rm T} \cdot \cbf_{\x,\theta}\right)\, \Delta\theta \\
&&  - \left( \J _{\theta,\x} - \lambda^{\rm T} \cdot \cbf_{\theta,\x} \right)\, \cbf_\x^{-1}\,\cbf_\theta\, \Delta\theta  + \left(\J_{\theta,\theta} - \lambda^{\rm T} \cdot \cbf_{\theta,\theta}\right)\, \Delta\theta \,.
\end{eqnarray*}

Using the reduced Hessian equation \eqref{eqn:lagRedHessian} we have that 
\begin{eqnarray*}
\ell_{\theta,\theta}\, \Delta\theta  & =&  -\left( \Delta\J_{\theta} - \lambda^T \cdot \Delta\cbf_{\theta}  \right)^T
 +  \cbf_\theta^T\, \cbf_\x^{-T} \cdot \left( \Delta\J_{\x}  -  \lambda^T \cdot \Delta\cbf_{\x} \right)^T  \\
&& -  \cbf_\theta^T\, \cbf_\x^{-T} \cdot\left( \J _{\x,\x} - \lambda^{\rm T} \cdot \cbf_{\x,\x} \right)\,  \cbf_\x^{-1}\,\Delta \cbf  
      + \left( \J _{\theta,\x} - \lambda^{\rm T} \cdot \cbf_{\theta,\x} \right)\,  \cbf_\x^{-1}\,\Delta \cbf  \\
&=&  \Delta\mathcal{O} -  \cbf_\theta^T\, \cbf_\x^{-T} \cdot \Delta\mathcal{A} \\
&& -  \left( \cbf_\theta^T\, \cbf_\x^{-T} \,\left( \J _{\x,\x} - \lambda^{\rm T} \cdot \cbf_{\x,\x} \right)  
      - \left( \J _{\theta,\x} - \lambda^{\rm T} \cdot \cbf_{\theta,\x} \right)\,\right)  \cbf_\x^{-1}\,\Delta\mathcal{F} \\
&=& \beta.
\end{eqnarray*}
where the residuals in the three KKT equations are denoted by
\begin{eqnarray*}
\Delta\mathcal{F} &=& -\Delta\cbf, \\
\Delta\mathcal{A} &=& -\left( \Delta\J_{\x} - \lambda^T \cdot \Delta\cbf_{\x} \right)^T, \\
\Delta\mathcal{O} &=& -\left( \Delta\J_{\theta} - \lambda^T \cdot \Delta\cbf_{\theta} \right)^T.
\end{eqnarray*}
Solve
\[
\ell_{\theta,\theta}\, \zeta =  \mathcal{E}_\theta^T \quad \Rightarrow \quad \zeta^T =  \mathcal{E}_\theta \cdot \ell_{\theta,\theta}^{-1}\,.
\]
Then
\[
\Delta \mathcal{E} \approx \mathcal{E}_\theta \cdot \Delta\theta = \mathcal{E}_\theta \cdot \ell_{\theta,\theta}^{-1} \cdot \beta = \zeta^T \cdot \beta.
\]
Therefore
\begin{eqnarray*}
\Delta \mathcal{E} &\approx&  \zeta^T \cdot \Delta\mathcal{O} - \zeta^T \cdot  \cbf_\theta^T\, \cbf_\x^{-T} \cdot \Delta\mathcal{A} \\
&& - \zeta^T \cdot  \left( \cbf_\theta^T\, \cbf_\x^{-T} \,\left( \J _{\x,\x} - \lambda^{\rm T} \cdot \cbf_{\x,\x} \right)  
      - \left( \J _{\theta,\x} - \lambda^{\rm T} \cdot \cbf_{\theta,\x} \right)\,\right)  \cbf_\x^{-1}\,\Delta\mathcal{F}
\end{eqnarray*}
Use the tangent linear model
\[
 \mu = -  \cbf_{\x}^{-\rm 1}\, \cbf_{\theta}\, \cdot \zeta \quad \Leftrightarrow \quad  \cbf_{\x} \cdot \mu = - \cbf_{\theta}\, \cdot \zeta.
\]
The error estimate becomes:
\begin{eqnarray*}
\Delta \mathcal{E} &\approx&  \zeta^T \cdot \Delta\mathcal{O} + \mu^T \cdot \Delta\mathcal{A} \\
&& + \left( \mu^T\,\left( \J _{\x,\x} - \lambda^{\rm T} \cdot \cbf_{\x,\x} \right)  
      +  \zeta^T \cdot\left( \J _{\theta,\x} - \lambda^{\rm T} \cdot \cbf_{\theta,\x} \right)\,\right)  \cbf_\x^{-1}\,\Delta\mathcal{F}
\end{eqnarray*}
Using the second order adjoint model
\[
 \cbf_{\x}^T \, \nu = - \left(\J_{\x,\x} - \lambda^{\rm T} \cdot \cbf_{\x,\x} \right)^T\, \mu
 - \left(\J_{\theta, \x} - \lambda^{\rm T} \cdot \cbf_{\theta, \x} \right)^T\, \zeta
\]
The error estimate becomes the familiar one:
\begin{eqnarray*}
\Delta \mathcal{E} &\approx&  \zeta^T \cdot \Delta\mathcal{O} + \mu^T \cdot \Delta\mathcal{A} - \nu^T\,\Delta\mathcal{F}.
\end{eqnarray*}
%

%%%%%%%%%%%%%%%%%%%%%%%%%%%%%%%
\subsection{Perturbed super-Lagrange parameters}
%%%%%%%%%%%%%%%%%%%%%%%%%%%%%%%
%
Recall the ideal KKT conditions \eqref{eqn:kktFiniteApp}
\begin{subequations} 
%\label{eqn:kktFiniteApp}
 %
\begin{eqnarray*}
%\label{kkt:fwdFiniteApp}
\textnormal{forward model:}~~ & 0& = \cbf \,, \\
%\label{kkt:adjFiniteApp}
\textnormal{adjoint model:}~~ &0& = \J_{\x}^T - \cbf_{\x}^T\, \lambda \\
%\label{kkt:optFiniteApp}
\textnormal{optimality:} ~~&0& =  \J_{\theta}^T - \cbf_{\theta}^T\, \lambda\,,
\end{eqnarray*}
\end{subequations}
and linearize them about $\widehat{\x}$, $\widehat{\theta}$
\begin{subequations} 
\label{eqn:idealKKTFiniteApp}
\begin{eqnarray}
 0 &=& \widehat{\cbf} -\widehat{\cbf}_\x\, \Delta \x - \widehat{\cbf}_{\theta}\, \Delta \theta \,, \\
 0 &=&\widehat{\J}_{\x}^T - \widehat{\J}_{\x,\x}\, \Delta\x  - \widehat{\J}_{\x,\theta}\, \Delta\theta \\
 && \nonumber - \left( \widehat{\cbf}_{\x} - \widehat{\cbf}_{\x, \x}\,  \Delta \x - \widehat{\cbf}_{\x, \theta}\, \Delta \theta \right)^T\, \left(\widehat{\lambda} - \Delta \lambda\right )  \\
 &=&\nonumber \widehat{\J}_{\x}^T - \widehat{\cbf}_{\x}^T \, \widehat{\lambda} + \widehat{\cbf}_{\x}^T \,\Delta \lambda \\
 && \nonumber - \left( \widehat{\J}_{\x,\x} - \widehat{\lambda} ^T\, \widehat{\cbf}_{\x, \x} \right)\, \Delta\x   - \left( \widehat{\J}_{\x,\theta} - \widehat{\lambda} ^T\, \widehat{\cbf}_{\x, \theta}\right)\, \Delta \theta \\
 && \nonumber  + \Delta \lambda^T\, \widehat{\cbf}_{\x, \x}\, \Delta\x  + \Delta \lambda^T\, \widehat{\cbf}_{\x, \theta}\, \Delta\theta,  \\
0&=& \widehat{\J}_{\theta}^T - \widehat{\J}_{\theta,\x}\, \Delta \x - \widehat{\J}_{\theta,\theta}\, \Delta \theta   \\ 
&& \nonumber - \left( \widehat{\cbf}_{\theta} - \widehat{\cbf}_{\theta, \x}\, \Delta \x  - \widehat{\cbf}_{\theta, \theta}\, \Delta \theta \right)^T\,
\left(\widehat{\lambda} - \Delta \lambda \right) \\
&=&\nonumber \widehat{\J}_{\theta}^T - \widehat{\cbf}_{\theta}^T \, \widehat{\lambda} + \widehat{\cbf}_{\theta}^T \,\Delta \lambda \\
 && \nonumber - \left( \widehat{\J}_{\theta,\x} - \widehat{\lambda} ^T\, \widehat{\cbf}_{\theta, \x} \right)\, \Delta\x   - \left( \widehat{\J}_{\theta,\theta} - \widehat{\lambda} ^T\, \widehat{\cbf}_{\theta, \theta}\right)\, \Delta \theta \\
&& \nonumber  + \Delta \lambda^T\, \widehat{\cbf}_{\theta, \x} \, \Delta\x  + \Delta \lambda^T\, \widehat{\cbf}_{\theta, \theta} \, \Delta\theta. 
\end{eqnarray}
\end{subequations}
Note that
\begin{eqnarray*}
\lambda^T\, \cbf(\x,\theta) &=& \sum_i \lambda_i \, \cbf_i (\x,\theta) \\
\frac{d\; \lambda^T\, \cbf(\x,\theta)}{d\x_j}  &=& \sum_i \lambda_i \, \frac{d\; \cbf_i (\x,\theta)}{d\x_j} =  \sum_i \lambda_i \, (\cbf_\x)_{i,j} =  \lambda^T \, (\cbf_\x)_{:,j}\\
\left( \frac{d\; \lambda^T\, \cbf(\x,\theta)}{d\x} \right)^T &=& \cbf_\x^T\, \lambda \\
\frac{d\; (\cbf_\x^T\, \lambda)_j }{d\x_k} &=&  \sum_i \lambda_i \, \frac{d\; (\cbf_\x)_{i,j} }{d\x_k}  =   \sum_i \lambda_i \, \frac{d^2\; \cbf_i }{d\x_j\, d\x_k} = \sum_i \lambda_i (\cbf_{\x,\x})_{i,j,k} \\
\frac{d\; (\cbf_\x^T\, \lambda) }{d\x} \, \Delta \x &=& \lambda^T\, \cbf_{\x,\x} \, \Delta\x = (\cbf_{\x,\x} \, \Delta\x)^T \, \lambda.
\end{eqnarray*}

Subtract the linearized ideal KKT conditions \eqref{eqn:idealKKTFiniteApp} from the perturbed KKT conditions  \eqref{eqn:kktFinitePertApp} to obtain
\begin{subequations}
\label{eqn:kktFinitePertDiff}
 \begin{eqnarray}
0&=& \Delta \widehat{\cbf} + \widehat{\cbf}_\x\, \Delta \x + \widehat{\cbf}_{\theta}\, \Delta \theta\,, \\
0&=& \Delta \widehat{\J}_{\x}^T - \Delta \widehat{\cbf}_{\x}^T\, \widehat{\lambda} - \widehat{\cbf}_{\x}^T \,\Delta \lambda \\
 && \nonumber + \left( \widehat{\J}_{\x,\x} - \widehat{\lambda} ^T\, \widehat{\cbf}_{\x, \x} \right)\, \Delta\x  + \left( \widehat{\J}_{\x,\theta} - \widehat{\lambda} ^T\, \widehat{\cbf}_{\x, \theta}\right)\, \Delta \theta \\
0&=&   \Delta \widehat{\J}_{\theta}^T - \Delta \widehat{\cbf}_{\theta}^T\, \widehat{\lambda} - \widehat{\cbf}_{\theta}^T \,\Delta \lambda \\
 && \nonumber + \left( \widehat{\J}_{\theta,\x} - \widehat{\lambda} ^T\, \widehat{\cbf}_{\theta, \x} \right)\, \Delta\x + \left( \widehat{\J}_{\theta,\theta} - \widehat{\lambda} ^T\, \widehat{\cbf}_{\theta, \theta}\right)\, \Delta \theta.
\end{eqnarray}
\end{subequations}
Note the similarity of \eqref{eqn:kktFinitePertDiff} with \eqref{eqn:kktFinitePertAppSimple}. While in  \eqref{eqn:kktFinitePertAppSimple} the functions are evaluated at 
the exact optimum, in \eqref{eqn:kktFinitePertDiff} they are evaluated at the perturbed optimum (which is the one we actually compute).

By substitution we arrive at the following:
\begin{subequations}
 \begin{eqnarray}
 \Delta \x &=& - \widehat{\cbf}_\x^{-1}\, \left( \Delta \widehat{\cbf} +  \widehat{\cbf}_{\theta}\, \Delta \theta \right) \\ 
\Delta \lambda &=& \widehat{\cbf}_{\x}^{-T} \,\left( \Delta \widehat{\J}_{\x}^T - \Delta \widehat{\cbf}_{\x}^T \, \widehat{\lambda} \right) \\
 && \nonumber - \widehat{\cbf}_{\x}^{-T} \,\left( \widehat{\J}_{\x,\x} - \widehat{\lambda} ^T\, \widehat{\cbf}_{\x, \x} \right)\,  \widehat{\cbf}_\x^{-1}\, \left( \Delta \widehat{\cbf} +  \widehat{\cbf}_{\theta}\, \Delta \theta \right)  \\
 && \nonumber+ \widehat{\cbf}_{\x}^{-T} \,\left( \widehat{\J}_{\x,\theta} - \widehat{\lambda} ^T\, \widehat{\cbf}_{\x, \theta}\right)\, \Delta \theta \\
\label{kkt:opt_perturbed}
0&=&   \Delta \widehat{\J}_{\theta}^T - \Delta \widehat{\cbf}_{\theta}^T\, \widehat{\lambda} \\
&& \nonumber - \widehat{\cbf}_{\theta}^T \,\widehat{\cbf}_{\x}^{-T} \,\left( \Delta \widehat{\J}_{\x}^T - \Delta \widehat{\cbf}_{\x}^T \, \widehat{\lambda} \right) \\
&& \nonumber + \widehat{\cbf}_{\theta}^T \,\widehat{\cbf}_{\x}^{-T} \,\left( \widehat{\J}_{\x,\x} - \widehat{\lambda} ^T\, \widehat{\cbf}_{\x, \x} \right)\,  \widehat{\cbf}_\x^{-1}\, \left( \Delta \widehat{\cbf} +  \widehat{\cbf}_{\theta}\, \Delta \theta \right) \\
&& \nonumber - \widehat{\cbf}_{\theta}^T \,\widehat{\cbf}_{\x}^{-T} \,\left( \widehat{\J}_{\x,\theta} - \widehat{\lambda} ^T\, \widehat{\cbf}_{\x, \theta}\right)\, \Delta \theta \\
 && \nonumber - \left( \widehat{\J}_{\theta,\x} - \widehat{\lambda} ^T\, \widehat{\cbf}_{\theta, \x} \right)\,  \widehat{\cbf}_\x^{-1}\, \left( \Delta \widehat{\cbf} +  \widehat{\cbf}_{\theta}\, \Delta \theta \right) \\
 && \nonumber  + \left( \widehat{\J}_{\theta,\theta} - \widehat{\lambda} ^T\, \widehat{\cbf}_{\theta, \theta}\right)\, \Delta \theta.
\end{eqnarray}
\end{subequations}
Consider the reduced perturbed Lagrangian
\begin{eqnarray}
\label{eqn:lagPertRedFiniteApp}
 \widehat{\ell}(\widehat{\theta}) =  \widehat{\J}(\widehat{\x}(\widehat{\theta}), \widehat{\theta} ) + \Delta \widehat{\J}(\widehat{\x}(\widehat{\theta}), \widehat{\theta} )  - \widehat{\lambda}(\widehat{\theta})^{\rm T} \cdot \left(\widehat{\cbf}(\widehat{\x}(\widehat{\theta}), \widehat{\theta} ) + \Delta \cbf (\widehat{\x}(\widehat{\theta}), \widehat{\theta} )\right)\,.
\end{eqnarray}
Similar to \eqref{eqn:lagRedHessian} the reduced perturbed Hessian evaluated at the perturbed optimal solution reads
\begin{eqnarray}
\label{eqn:lagRedHessianPertA}
 \widehat{\ell}_{\theta,\theta}
 &=& \widehat{\J}_{\theta,\theta} - \widehat{\lambda}^{\rm T} \, \cbf_{\theta,\theta} +  \Delta\widehat{\J}_{\theta,\theta} - \widehat{\lambda}^{\rm T} \, \Delta\cbf_{\theta,\theta}  \\
\nonumber
 && - \left( \J_{\theta,\x} - \lambda^{\rm T} \, \cbf_{\theta,\x} \right)\, (\widehat{\cbf}_\x+\Delta\widehat{\cbf}_\x)^{-1}\, (\widehat{\cbf}_\theta + \Delta\widehat{\cbf}_\theta)\, \\
\nonumber
 && - \left( \Delta\J_{\theta,\x} - \lambda^{\rm T} \, \Delta\cbf_{\theta,\x} \right)\, (\widehat{\cbf}_\x+\Delta\widehat{\cbf}_\x)^{-1}\, (\widehat{\cbf}_\theta + \Delta\widehat{\cbf}_\theta)\, \\
 \nonumber
 && -   (\widehat{\cbf}_\theta + \Delta\widehat{\cbf}_\theta)^T\,  (\widehat{\cbf}_\x+\Delta\widehat{\cbf}_\x)^{-T}\,\, \left( \J_{\x,\theta} - \lambda^{\rm T} \, \cbf_{\x,\theta} \right) \, \\
 \nonumber
 && -   (\widehat{\cbf}_\theta + \Delta\widehat{\cbf}_\theta)^T\,  (\widehat{\cbf}_\x+\Delta\widehat{\cbf}_\x)^{-T}\,\, \left( \Delta\J_{\x,\theta} - \lambda^{\rm T} \, \Delta\cbf_{\x,\theta} \right) \, \\
\nonumber
 && + (\widehat{\cbf}_\theta + \Delta\widehat{\cbf}_\theta)^T\,  (\widehat{\cbf}_\x+\Delta\widehat{\cbf}_\x)^{-T}\,  \left( \J_{\x,\x} - \lambda^{\rm T} \, \cbf_{\x,\x} \right) \, (\widehat{\cbf}_\x+\Delta\widehat{\cbf}_\x)^{-1}\, (\widehat{\cbf}_\theta + \Delta\widehat{\cbf}_\theta)\\
\nonumber
 && + (\widehat{\cbf}_\theta + \Delta\widehat{\cbf}_\theta)^T\,  (\widehat{\cbf}_\x+\Delta\widehat{\cbf}_\x)^{-T}\,  \left( \Delta\J_{\x,\x} - \lambda^{\rm T} \, \Delta\cbf_{\x,\x} \right) \, (\widehat{\cbf}_\x+\Delta\widehat{\cbf}_\x)^{-1}\, (\widehat{\cbf}_\theta + \Delta\widehat{\cbf}_\theta).
\end{eqnarray}
Assume that $\Vert \Delta\widehat{\cbf}_\x \Vert$ and $\Vert \Delta\widehat{\cbf}_\theta \Vert$ are small. Neglecting products of small terms we have that
\begin{eqnarray*}
(\widehat{\cbf}_\x+\Delta\widehat{\cbf}_\x)^{-1}\, (\widehat{\cbf}_\theta + \Delta\widehat{\cbf}_\theta) &\approx& (\widehat{\cbf}_\x^{-1}-\Delta\widehat{\cbf}_\x)\, (\widehat{\cbf}_\theta + \Delta\widehat{\cbf}_\theta) \\
&\approx& \widehat{\cbf}_\x^{-1}\, \widehat{\cbf}_\theta +\widehat{\cbf}_\x^{-1}\,  \Delta\widehat{\cbf}_\theta -\Delta\widehat{\cbf}_\x\, \widehat{\cbf}_\theta.
\end{eqnarray*}
We also assume that  $\Vert \Delta\widehat{\cbf}_{\x,\x} \Vert$, $\Vert \Delta\widehat{\cbf}_{\x,\theta} \Vert$, and $\Vert \Delta\widehat{\cbf}_{\theta,\theta} \Vert$ are small. 

With this approximation, and after neglecting products of small terms, the reduced perturbed Hessian \eqref{eqn:lagRedHessianPertA} becomes
\begin{eqnarray}
\label{eqn:lagRedHessianPertB}
 \widehat{\ell}_{\theta,\theta}
 &=& \widehat{\J}_{\theta,\theta} - \widehat{\lambda}^{\rm T} \, \cbf_{\theta,\theta} \\
\nonumber
 && +  \Delta\widehat{\J}_{\theta,\theta} - \widehat{\lambda}^{\rm T} \, \Delta\cbf_{\theta,\theta}  \\
\nonumber
 && - \left( \J_{\theta,\x} - \lambda^{\rm T} \, \cbf_{\theta,\x} \right)\, (\widehat{\cbf}_\x^{-1}\, \widehat{\cbf}_\theta +\widehat{\cbf}_\x^{-1}\,  \Delta\widehat{\cbf}_\theta -\Delta\widehat{\cbf}_\x\, \widehat{\cbf}_\theta)\, \\
\nonumber
 && - \left( \Delta\J_{\theta,\x} - \lambda^{\rm T} \, \Delta\cbf_{\theta,\x} \right)\, (\widehat{\cbf}_\x^{-1}\, \widehat{\cbf}_\theta)\, \\
 \nonumber
 && - (\widehat{\cbf}_\x^{-1}\, \widehat{\cbf}_\theta +\widehat{\cbf}_\x^{-1}\,  \Delta\widehat{\cbf}_\theta -\Delta\widehat{\cbf}_\x\, \widehat{\cbf}_\theta)^{T}\, \left( \J_{\x,\theta} - \lambda^{\rm T} \, \cbf_{\x,\theta} \right) \, \\
 \nonumber
 && - (\widehat{\cbf}_\x^{-1}\, \widehat{\cbf}_\theta)^{T}\, \left( \Delta\J_{\x,\theta} - \lambda^{\rm T} \, \Delta\cbf_{\x,\theta} \right) \, \\
\nonumber
 && + (\widehat{\cbf}_\x^{-1}\, \widehat{\cbf}_\theta)^{T}\,  \left( \J_{\x,\x} - \lambda^{\rm T} \, \cbf_{\x,\x} \right) \, (\widehat{\cbf}_\x^{-1}\, \widehat{\cbf}_\theta +\widehat{\cbf}_\x^{-1}\,  \Delta\widehat{\cbf}_\theta -\Delta\widehat{\cbf}_\x\, \widehat{\cbf}_\theta)\\
\nonumber
 && + (\widehat{\cbf}_\x^{-1}\, \widehat{\cbf}_\theta +\widehat{\cbf}_\x^{-1}\,  \Delta\widehat{\cbf}_\theta -\Delta\widehat{\cbf}_\x\, \widehat{\cbf}_\theta)^{T}\,  \left( \J_{\x,\x} - \lambda^{\rm T} \, \cbf_{\x,\x} \right) \, (\widehat{\cbf}_\x^{-1}\, \widehat{\cbf}_\theta)\\
\nonumber
 && + (\widehat{\cbf}_\x^{-1}\, \widehat{\cbf}_\theta)^{T}\,  \left( \Delta\J_{\x,\x} - \lambda^{\rm T} \, \Delta\cbf_{\x,\x} \right) \, (\widehat{\cbf}_\x^{-1}\, \widehat{\cbf}_\theta).
\end{eqnarray}
After neglecting products of small terms
\begin{eqnarray*}
 \widehat{\ell}_{\theta,\theta}\cdot \Delta\theta
 &=& \left( \widehat{\J}_{\theta,\theta} - \widehat{\lambda}^{\rm T} \, \cbf_{\theta,\theta}\right) \Delta \theta \\
\nonumber
 && - \left( \J_{\theta,\x} - \lambda^{\rm T} \, \cbf_{\theta,\x} \right)\, (\widehat{\cbf}_\x^{-1}\, \widehat{\cbf}_\theta)\, \Delta \theta \\
\nonumber
  && - (\widehat{\cbf}_\x^{-1}\, \widehat{\cbf}_\theta )^{T}\, \left( \J_{\x,\theta} - \lambda^{\rm T} \, \cbf_{\x,\theta} \right) \,\Delta\theta \\
\nonumber
 && + (\widehat{\cbf}_\x^{-1}\, \widehat{\cbf}_\theta)^{T}\,  \left( \J_{\x,\x} - \lambda^{\rm T} \, \cbf_{\x,\x} \right) \, (\widehat{\cbf}_\x^{-1}\, \widehat{\cbf}_\theta) \Delta\theta.
\end{eqnarray*}

The last equation \eqref{kkt:opt_perturbed} reads
 \begin{eqnarray}
 0&=&   \Delta \widehat{\J}_{\theta}^T - \Delta \widehat{\cbf}_{\theta}^T\, \widehat{\lambda} - \widehat{\cbf}_{\theta}^T \,\widehat{\cbf}_{\x}^{-T} \,\left( \Delta \widehat{\J}_{\x}^T - \Delta \widehat{\cbf}_{\x}^T \, \widehat{\lambda} \right) \\
&& \nonumber + \widehat{\cbf}_{\theta}^T \,\widehat{\cbf}_{\x}^{-T} \,\left( \widehat{\J}_{\x,\x} - \widehat{\lambda} ^T\, \widehat{\cbf}_{\x, \x} \right)\,  \widehat{\cbf}_\x^{-1}\, \Delta \widehat{\cbf}  \\
 && \nonumber - \left( \widehat{\J}_{\theta,\x} - \widehat{\lambda} ^T\, \widehat{\cbf}_{\theta, \x} \right)\,  \widehat{\cbf}_\x^{-1}\, \Delta \widehat{\cbf}  \\
&& \nonumber + \widehat{\cbf}_{\theta}^T \,\widehat{\cbf}_{\x}^{-T} \,\left( \widehat{\J}_{\x,\x} - \widehat{\lambda} ^T\, \widehat{\cbf}_{\x, \x} \right)\,  \widehat{\cbf}_\x^{-1}\,\widehat{\cbf}_{\theta}\, \Delta \theta \\
&& \widehat{\ell}_{\theta,\theta}\cdot \Delta \theta.
\end{eqnarray}

The derivation follows identical to the unperturbed case.

 \end{appendices}

\end{document}